%% file: conjecture0402.tex
\documentclass[amssymb,11pt]{article}
\usepackage{amsmath,amsthm,amscd}
\usepackage[psamsfonts]{amssymb}
\usepackage{graphicx,psfrag}
\usepackage{floatflt}
 
  \def\CC{{\mathbb C}}

 \def\NN{{\mathbb N}}  
 \def\RR{{\mathbb R}}  \def\TT{{\mathbb T}}
   
 \def\ZZ{{\mathbb Z}}

  \def\cG{{\cal G}} \def\cM{{\cal M}} 
\def\cB{{\cal B}}   \def\cN{{\cal N}} \def\cT{{\cal T}}
\def\cC{{\cal C}}   \def\cO{{\cal O}} \def\cU{{\cal U}}
\def\cD{{\cal D}}    \def\cV{{\cal V}}
  \def\cK{{\cal K}}  \def\cW{{\cal W}}
\def\cF{{\cal F}}   \def\cR{{\cal R}} \def\cX{{\cal X}}
  \def\cZ{{\cal Z}}

 \def\La{\Lambda} \def\De{\Delta} \def\Om{\Omega}
\def\Ga{\Gamma} 
   
 \def\ve{\varepsilon} \def\eps{\varepsilon}

\newtheorem{theorem}{Theorem}[section]
\newtheorem{theo}[theorem]{Theorem}
\newtheorem*{theo*}{Theorem}
\newtheorem{proposition}[theorem]{Proposition}
\newtheorem{prop}[theorem]{Proposition}
\newtheorem*{prop*}{Proposition}
\newtheorem{lemma}[theorem]{Lemma}
\newtheorem{lemm}[theorem]{Lemma}

\newtheorem{coro}[theorem]{Corollary}
\newtheorem*{affi}{Claim}
\newtheorem*{theoA}{Theorem A}
\newtheorem*{theoB}{Theorem B}
\newtheorem*{theoB'}{Theorem B'}
\newtheorem*{maintheo}{Main Theorem}
\newtheorem{ques}{Question}

\theoremstyle{definition}

\newtheorem{conj}[theorem]{Conjecture}
\newtheorem{defi}[theorem]{Definition}

\theoremstyle{remark}
\newtheorem{remark}[theorem]{Remark}
\newtheorem{rema}[theorem]{Remark}

\newenvironment{demo}[1][{\noindent Proof}]{\noindent{\bf #1. }}{\hfill$\Box$\medskip}

\def\jac{{\operatorname{det}}} \def\det{{\operatorname{det}}} \def\Det{\operatorname{det} }
\def\diff{\operatorname{Diff}} \def\Diff{\operatorname{Diff} }
\def\diam{{\operatorname{diam}}}
\def\interior{\operatorname{Int} }
\def\Per{\operatorname{Per} }

\def\Lip{\operatorname{Lip} }
\def\orb{\operatorname{orb} }

\def\id{\operatorname{id}}

\def\bar{\overline}

\def\transverse{\,\raise2pt\hbox to1em{\hfil$\top$\hfil}\hskip -1em \hbox to1em{\hfil$\cap$\hfil}\,} 

\title{The $C^1$ generic diffeomorphism has trivial centralizer}
\author{C. Bonatti, S. Crovisier and A. Wilkinson}

\begin{document}

\maketitle
\begin{abstract}
Answering a question of Smale, we prove that the space of $C^1$
diffeomorphisms of a compact manifold contains a residual subset of
diffeomorphisms whose centralizers are trivial. \vskip 2mm

%\begin{description}
%em[\bf Key words:] Centralizer,  $C^1$ generic properties.
%\item[\bf MSC 2000:] ******.
%\end{description}
\end{abstract}

%\newpage
\tableofcontents

\newpage
\input{intro0402.tex}

\input{strategy0402.tex}
\input{tidy0402.tex}
\input{LD-reduction0402.tex}
\input{LD-perturbation0402.tex}
\input{UD-reduction0402.tex}
\input{UD-perturbation0402.tex}
\input{appendix0402.tex}

\vspace{10pt}

\noindent \textbf{Christian Bonatti (bonatti@u-bourgogne.fr)}\\
\noindent  CNRS - Institut de Math\'ematiques de Bourgogne, UMR 5584\\
\noindent  BP 47 870\\
\noindent  21078 Dijon Cedex, France\\
\vspace{10pt}

\noindent \textbf{Sylvain Crovisier (crovisie@math.univ-paris13.fr)}\\
\noindent CNRS - Laboratoire Analyse, G\'eom\'etrie et Applications, UMR 
7539,\\
\noindent Institut Galil\'ee, Universit\'e Paris 13, Avenue J.-B. 
Cl\'ement,\\
\noindent 93430 Villetaneuse, France\\
\vspace{10pt}

\noindent \textbf{Amie Wilkinson (wilkinso@math.northwestern.edu)}\\
\noindent Department of Mathematics, Northwestern University\\
\noindent 2033 Sheridan Road \\
\noindent Evanston, IL 60208-2730,  USA

\end{document}

%% file: intro0402.tex
\section{Introduction}

\subsection{The centralizer problem, and its solution in $\Diff^1(M)$} 
It is a basic fact from linear algebra that any two commuting matrices in $GL(n, \CC)$ can be simultaneously triangularized. 
The reason why is that if $A$ commutes with $B$, then $A$ preserves
the generalized eigenspaces of $B$, and vice versa.  Hence the relation $AB=BA$ is quite special
among matrices, and one might not expect it to be satisfied very often.
In fact, it's not hard to see that the generic pair of matrices $A$ and $B$ generate a free group;
any relation in $A$ and $B$ defines a nontrivial subvariety in the space of pairs of matrices, so the set of 
pairs of matrices that satisfy {\em no relation} is {\em residual}: it contains a countable intersection of open-dense
subsets.  On the other hand, the generic $n\times n$ matrix is diagonalizable and commutes 
with an $n$-dimensional space of matrices.  That is, a residual set of matrices have large centralizer.

Consider the same sort of questions, this time for the group of $C^1$ diffeomorphisms $\diff^1(M)$ of a compact manifold
$M$.  If $f$ and $g$ are diffeomorphisms satisfying $fg=gf$, then $g$ preserves the set of orbits of $f$, as well as
all of the smooth and topological dynamical invariants of $f$, and vice versa.  Also in analogy
to the matrix case, an easy transversality argument (written in~\cite[Proposition 4.5]{ghys} for circle homeomorphisms), 
shows that \emph{for a generic $(f_1, \ldots, f_p) \in \left(\Diff^r(M)\right)^p$ with $p\geq 2$ and $r\geq 0$,
the group $\langle f_1, \ldots, f_p\rangle$ is free.}  In contrast with the matrix case, however, the generic
$C^1$ diffeomorphism cannot have a large centralizer.   This is the content of this paper.
Our main result is:
\begin{maintheo} Let $M$ be any closed, connected smooth manifold. There is a residual subset $\cR\subset \diff^1(M)$
such that for any $f\in\cR$, and any $g\in \diff^1(M)$, if $fg=gf$, then $g=f^n$, for some $n\in\ZZ$.
\end{maintheo}

This theorem\footnote{This result has been announced in~\cite{BCW2}.}
 gives an affirmative answer in the $C^1$ topology to the following question, posed by S. Smale.
We fix a closed manifold $M$ and consider the space $\Diff^r(M)$ of  $C^r$ diffeomorphisms
of $M$, endowed with the $C^r$ topology.  The 
{\em centralizer} of  $f\in \Diff^r(M)$ is defined 
as $$Z^r(f):=\{g\in \Diff^r(M): fg=gf\}.$$
Clearly $Z^r(f)$ always contains the cyclic group 
$\langle f\rangle$ of all the powers of $f$.  We 
say that $f$ has {\em trivial centralizer} if $Z^r(f) = \langle f\rangle$. 
Smale asked the following:

\begin{ques}[\cite{Sm1,Sm2}]\label{q.smale} 
Consider the set of $C^r$ diffeomorphisms of a compact manifold $M$ with trivial centralizer.  
\begin{enumerate} 
\item Is this set dense in $\Diff^r(M)$?
\item Is it {\em residual} in $\Diff^r(M)$?  That is, does it contain a dense $G_\delta$?
\item Does it contain an open and dense subset of $\Diff^r(M)$? 
\end{enumerate}
\end{ques}

For the case $r=1$ we now have a complete answer to this question. The theorem above shows that for 
any compact manifold $M$, there is a residual subset of $\Diff^1(M)$ consisting of diffeomorphisms with trivial centralizer,
giving an  affirmative answer to the second (and hence the first) part of Question~\ref{q.smale}.
Recently, with G. Vago, we have also shown:

\begin{theo*}\cite{BCVW}
For any compact manifold $M$, the set of $C^1$ diffeomorphisms
with trivial centralizer does not contain any open and dense subset.
\end{theo*}

This result gives a negative answer to the third part of Question~\ref{q.smale}; 
on any compact manifold, \cite{BCVW} exhibits a family of $C^\infty$ diffeomorphisms with 
large centralizer that is $C^1$ dense in a nonempty
open subset of $\Diff^1(M)$.

The history of Question~\ref{q.smale}  goes back to the work of N. Kopell \cite{Ko}, who
gave a complete answer for $r\geq 2$ and the circle $M=S^1$: the
set of diffeomorphisms with trivial centralizer contains an open and dense subset of $\Diff^r(S^1)$.
For $r\geq 2$ on higher dimensional manifolds, there are partial results with additional
dynamical assumptions, such as hyperbolicity \cite{PY1, PY2, Fi} and partial hyperbolicity \cite{Bu}.
In the $C^1$ setting, Togawa proved that generic Axiom A diffeomorphisms have trivial centralizer.  In an earlier work
\cite{BCW}, we showed that for $\dim(M)\geq 2$,  the $C^1$ generic conservative (volume-preserving or symplectic)
diffeomorphism has trivial centralizer in $\Diff^1(M)$.  A more complete list of previous results can be found in \cite{BCW}.

\subsection{The algebraic structure and the topology of $\Diff^1(M)$} 

These results suggest that the topology of the set of diffeomorphisms with trivial centralizer is complicated and motivate
the following questions.
\begin{ques}
\begin{enumerate}
\item {\em Consider the set of diffeomorphisms whose centralizer is trivial.}
What is its interior?
\item Is it a Borel set?
{\em (See \cite{FRW} for a negative answer to this question in the measurable context.)}
\item {\em The set $\{(f,g) \in \Diff^1(M)\times \Diff^1(M): fg=gf\}$ is closed}.
What is its local topology? For example, is it locally connected?
\end{enumerate}
\end{ques}

Beyond just proving the genericity of diffeomorphisms with trivial centralizer, we find a precise collection
of dynamical properties of $f$ that imply that $\cZ^1(f) = \langle f \rangle$.   As an illustration of  what we
mean by this, consider the Baumslag-Solitar relation $gfg^{-1} = f^n$, where $n>1$ is a fixed integer.  Notice that 
if this relation holds for some $g$, then the periodic points of $f$ cannot be hyperbolic.  
This implies that for the $C^1$ generic $f$, there is {\em no diffeomorphism} $g$ satisfying $gfg^{-1} = f^n$.
In this work, we consider the relation $gfg^{-1} = f$.  The dynamical properties of $f$ that forbid this relation 
for $g\notin \langle f \rangle$ are the Large Derivative (LD) and Unbounded Distortion (UD) properties (described in greater detail in the subsequent sections).  

Hence the Main Theorem illustrates the existence of links between the dynamics of a diffeomorphism and its algebraic 
properties as a element of the group $\Diff^1(M)$.  In that direction, it seems natural to propose 
the following (perhaps naive) generalizations of Question~\ref{q.smale}.
\begin{ques} 
\begin{enumerate}
\item Consider $G=\langle a_1, \dots, a_k\, |\, r_1, \dots, r_m\rangle$
a finitely presented group where $\{a_i\}$ is a generating set and $r_i$ are relations. 
How large is the set of diffeomorphisms $f\in \Diff^r(M)$ such that there is an injective morphism $\rho\colon G\to\Diff^r(M)$ with $\rho(a_1)=f$?
{\em The Main Theorem implies that this set is meagre if $G$ is abelian (or even nilpotent).}
\item Does there exist a diffeomorphism $f\in\Diff^r(M)$ such that, for every $g\in\Diff^r(M)$, the group generated by $f$ and $g$ is either 
$\langle f\rangle$ if $g\in \langle f\rangle$ or  the free product $\langle f\rangle \ast\langle g\rangle$ ? 
\item  Even more, does there exist a diffeomorphism $f\in\Diff^r(M)$ such that, for every  finitely generated group $G\in\Diff^r(M)$ with $f\in G$, there is a subgroup $H\subset G$ such that $G=\langle f\rangle \ast H$ ?
\item If the answer to either of the two previous items is ``yes", then how large are the corresponding sets of diffeomorphisms with these properties? 
\end{enumerate}
\end{ques} 

\subsection{On $C^1$-generic diffeomorphisms}

Our interest in Question~\ref{q.smale} comes also from another motivation. The study of the dynamics of 
diffeomorphisms from the perspective of the $C^1$-topology has made substantial progress 
in the last decade, renewing hope for a reasonable description of 
$C^1$-generic dynamical systems. The elementary question on centralizers considered here
presents a natural challenge for testing the strength of the new tools. 

Our result certainly uses these newer results on $C^1$-generic dynamics, 
in particular those related to  Pugh's closing lemma, Hayashi's connecting lemma, 
the extensions of these techniques incorporating Conley theory,  and the concept of 
\emph{topological tower} introduced in \cite{BC}.
Even so, the Main Theorem is far from a direct consequence 
of these results;  in this subsection, we  focus on the new techniques and tools
we have developed in the process of answering Question~\ref{q.smale}. 

Our proof of the Main Theorem goes back to the property  
which is at the heart of most perturbation results specific to the $C^1$-topology.

\subsubsection{A special property of the $C^1$-topology}

If we focus on a very small neighborhood of an orbit segment under a differentiable map, 
the dynamics will appear to be linear around each point; 
thus locally, iterating a diffeomorphism consists in multiplying 
invertible linear maps.  

On the other hand, the $C^1$-topology is the unique smooth topology that is invariant under
(homothetical) rescaling.  More precisely, consider a perturbation $f\circ g$  
of  a diffeomorphism $f $,  where $g$ is  a diffeomorphism supported in a ball $B(x,r)$ and $C^1$-close to the identity.  The $C^1$-size of the perturbation does not increase if we replace $g$  by its conjugate $h_\lambda g h_\lambda^{-1}$ where $h_\lambda$ is the homothety of ratio $\lambda<1$. The new perturbation is now supported on the ball $B(x,r\lambda)$. 
When $\lambda$ goes to $0$, we just see a perturbation of the linear map $D_{x}f$. 

This shows that  local $C^1$-perturbation results are closely related to perturbations of linear cocycles. This connection is quite specific to the $C^1$-topology: this type of renormalization of a perturbation causes its $C^2$-size to increase proportionally to the inverse of the radius of its support. The rescaling invariance property of the $C^1$ topology is shared with the Lipschitz topology; however, bi-Lipschitz homeomorphisms do not look like linear maps at small scales.   

This special property of the $C^1$-topology was first used in the proof of Pugh's closing lemma, 
and perhaps the deepest part of the proof consists in understanding 
perturbations of a linear cocycle. Pugh introduced the  fundamental,  simple idea that, 
if we would like to perform a large perturbation in a neighborhood of a point $x$, we can spread this perturbation along the orbit of $x$ and obtain the same result, but by $C^1$-small perturbations supported in the neighborhood of an orbit segment. The difficulty in this idea is that, if one performs a small perturbation of $f$ in a very small neighborhood of the point $f^i(x)$, the effect of this perturbation when observed from a small neighborhood of $x$ is deformed by a conjugacy under the linear map $Df^i$; this deformation is easy to understand if $Df^i$ is an isometry or a conformal map, but it has no reason in general to be so. 
%For this reason, one deep part of Pugh closing lemma consists in solving a linear algebra problem, 
%associating an orthogonal basis to every sequence $\{A_i\}$ of matrices in such a way 
%that the deformation induced by the product of matrices 
%$A_i\dots A_0$ will be understandable in that basis, for many values of $i$. 

Our result is based on Propositions~\ref{p.tidycube} and~\ref{p.pertAreduced}, which both produce a perturbation of a linear cocycle supported in the iterates of a cube $Q$. The aim of Proposition~\ref{p.tidycube} is to perturb the derivative $Df$ in order to obtain a large norm  $\|Df^n\|$ or $\|Df^{-n}\|$, for a given $n$ and for every orbit that meets a given subcube $\delta Q\subset Q$; the aim of Proposition~\ref{p.pertAreduced} is to obtain a large variation of the jacobian $\det Df^n$, for some integer $n$ and for all the points of a given subcube $\theta Q\subset Q$.   The first type of perturbation is connected to the Large Derivative (LD) property, and the second, to the Unbounded Distortion (UD) property.

The main novelty of these two elementary perturbation results  is that, in contrast to the case of the closing lemma, whose aim is to perturb a single orbit, our perturbation lemmas will be used to perturb 
all the orbits of a given compact set.  To perturb all orbits in a compact set, we first cover the set with open cubes, and then carry out the perturbation cube-by-cube, using different orbit segments for adjacent cubes. For this reason, we need to control the effect of the perturbation associated to a given cube on the orbits through all of its  neighboring cubes.  

To obtain this control,  for each cube in the cover, we perform a perturbation along the iterates of the cubes until we obtain the desired effect on the derivative, and then we use  more iterates of the cube to ``remove the perturbation."  By this method, we ensure that the long-term effect of the perturbation on the orbit will be as small as possible in the case of Proposition~\ref{p.pertAreduced}, and indeed completely removed in the case of Proposition~\ref{p.tidycube}.  In the latter case, 
we speak of \emph{tidy perturbations}. These perturbations, which are doing ``nothing or almost nothing," are our main tools. 

\subsubsection{Perturbing the derivative without changing the topological dynamics} 

In the proof of the Main Theorem, we show that every diffeomorphism $f$ can be $C^1$-perturbed in order to obtain simultaneously  the (UD)- and (LD)-properties,  which together imply the triviality of the centralizer.  However, unlike the (UD)-property,  the (LD)-property is not a generic property: to get both properties to hold simultaneously, we have to perform a perturbation that produces the (LD)-property 
while preserving  the (UD)-property.

Our solution consists in changing the derivative of $f$ without changing its topological dynamics. The only context in which an arbitrary perturbation of the derivative of $f$ will not alter its dynamics is one where  $f$ is structurally stable.  Here $f$ is not assumed to have any kind of stability property, and yet we can realize a substantial effect on the derivative by a perturbation preserving the topological dynamics.  For example, starting with an irrational rotation of the torus
$\TT^d$, we can obtain, via an arbitrarily $C^1$-small perturbation, a diffeomorphism $g$, conjugate to the original rotation, with the property:
$$\lim_{n\to\infty} \inf_{x\in \TT^d} \sup_{y\in{ \orb}_g(x)} \|Dg^{n} (y)\| + \|Dg^{-n}(y)\|  = \infty.$$

Let us state our result (this is a weak version of Theorem~B below):  

\begin{theo*} Let $f$ be a diffeomorphism whose periodic orbits are all hyperbolic. Then any $C^1$-neighborhood of $f$ contains a diffeomorphism $g$
such that 
\begin{itemize}
\item $g$ is conjugate to $f$ via a homeomorphism.
\item $g$ has the \emph{large derivative property}: for every $K>0$ there exists $n_K$ such that, for every $n\geq n_K$ and any non-periodic point $x$:   $$\sup_{y\in {\orb}_g(x)}\{\|Dg^n(y)\|,\|Dg^{-n}(y)\|\}>K.$$
\end{itemize}
\end{theo*}

As far as we know, this result is the first perturbation lemma which produces a perturbation of the derivative inside the topological conjugacy class of a given diffeomorphism (with the unique hypothesis that all periodic orbits are  hyperbolic, which is generic in any topology).   In the proof, we construct  the perturbation $g$ of $f$ as a limit of tidy perturbations $g_n$ of $f$ which are smoothly conjugate to $f$. As we think that tidy perturbations  will used in further works, we present them in some detail.

\subsubsection{Tidy perturbations}
Let $f$ be a diffeomorphism and let $U$ be an open set such that the first iterates $\bar U$, $f(\bar U)$, \dots, $f^n(\bar U)$ are pairwise disjoint, for some $n>0$. Consider a perturbation $g$ of $f$ with support in $V=\bigcup_0^{n-1} f^i(U)$ that has no effect at all on the orbits crossing $V$:  that is, for every $x\in U$, $f^n(x)=g^n(x)$. In our proof of Theorem~B, we construct such perturbations, using the first iterates $\bar U$, $f(\bar U)$, \dots , $f^i(\bar U)$, for some $i\in\{0,\dots,n-2\}$, for perturbing $f$ and getting the desired effect on the derivative, and using the remaining iterates $f^{i+1}(\bar U)$,\dots,$f^{n-1}(\bar U)$ for removing the perturbation, bringing the $g$-orbits of the points $x\in \bar U$ back onto their $f$-orbits. The diffeomorphism $g$ is smoothly conjugate to $f$ via some diffeomorphism that is the identity map outside $\bigcup_0^{n-1} f^i(U)$. 
Such a perturbation is called  a \emph{tidy perturbation}. 

Tidy perturbations require  open sets that are disjoint from many of their iterates.  To get properties to hold on all non-periodic points, we use such open sets that also cover all the orbits of $f$, up to finitely many periodic points of low period. Such open sets, called \emph{topological towers}, were constructed in \cite{BC} for the proof of a connecting lemma for pseudo-orbits.  
%%%%%%%%%%%%%%%%%% added
In Section~\ref{s.wanderingorbit}, we further refine the construction of topological towers;  in particular, we prove that it is always possible to construct topological towers disjoint from any compact set with the {\em wandering orbit property}.  The wandering orbit property is a weak form of wandering behavior, satisfied, for example, by nonperiodic points and compact subsets of the wandering set.

In the proof of Theorem~B, we construct an infinite sequence $(g_i)$,  where $g_i$ is a tidy perturbation of $g_{i-1}$, with support in the disjoint $n_i$ first iterates of some open set $U_i$. We show that, if the diameters of the $U_i$ decrease quickly enough, then the conjugating diffeomorphisms converge in the $C^0$-topology to a homeomorphism conjugating the limit diffeomorphism $g$ to $f$. 
With a weaker decay to $0$ of the diameters of the $U_i$, it may happen that the conjugating diffeomorphisms converge uniformly to some continuous noninvertible map. In that case, the limit diffeomorphism $g$ is merely semiconjugate to $f$. This kind of technique has already been used, for instance by M. Rees \cite{R}, who constructed homeomorphisms of the torus  with positive entropy and semiconjugate to an irrational rotation.

Controlling the effect of successive general perturbations is very hard.
For tidy perturbations it is easier to manage the effect of a successive sequence of them, since each of them
has no effect on the dynamics.
However this advantage leads to some limitations on the effect we can hope for, in particular
on the derivative.
We conjecture for instance that it is not possible to change the Lyapunov exponents:

\begin{conj} Let $g=hfh^{-1}$ be a limit of tidy perturbations $g_i=h_ifh_i^{-1}$ of $f$, where the $h_i$ converge to $h$. Then given  any ergodic invariant measure $\mu$ of $f$,  the Lyapunov exponents of the measure $h_*(\mu)$ for $g$  are the same as those of $\mu$ for $f$.
\end{conj}
To obtain the (LD)-property, we create some oscillations in the size of the derivative along orbits.  It seems natural to ask if, on the other hand, one could erase oscillations of the derivative. Let us formalize a question:

\begin{ques} Let $f$ be a diffeomorphism and assume that $\La$ is a minimal invariant compact set of $f$ that is uniquely ergodic
with invariant measure $\mu$. Assume that all the Lyapunov exponents of $\mu$ vanish.  Does there exist $g=hfh^{-1}$, a limit of tidy perturbations of $f$, such that the norm of $Dg^n$ and of $Dg^{-n}$ remain bounded on $h(\La)$?
\end{ques} 

Such minimal uniquely ergodic sets appear in the $C^1$ generic setting (among the dynamics 
exhibited by~\cite{BD}).

In this paper, starting with a diffeomorphism $f$ without the (LD)-property, we build a perturbation $g=hfh^{-1}$ as a limit of tidy perturbations %$g_i=h_ifh_i^{-1}$ of $f$, where the $h_i$ converge to $h$ in the $C^0$-topology, 
and such that $g$ satisfies the (LD)-property.  Then the linear cocycles $Df$ over $f$ and $Dg$ over $g$ are not conjugate by a continuous linear cocycle over $h$.
\begin{ques} Is there a measurable linear cocycle over $h$ conjugating $Dg$ to $Df$?
\end{ques}

\subsection{Local and global: the structure of the proof of the Main Theorem}
The proof of the Main Theorem breaks into two parts, a
``local'' one and a ``global'' one. This is also the general
structure of the proofs of the main results in \cite{Ko, PY1, PY2,
To1, To2, Bu2}:
\begin{itemize}
\item  The local part proves that for
the generic $f$,  if $g$ commutes with $f$, then $g=f^\alpha$ on an
open and dense subset  $W\subset M$, where $\alpha\colon W\to \ZZ$
is a locally constant function. 

This step consists in ``individualizing'' a dense collection
of orbits of $f$, arranging that the behavior of the diffeomorphism in a neighborhood 
of one orbit is different from the behavior in a neighborhood
of any other. Hence $g$ must preserve each of these orbits, 
which allows us to obtain the function $\alpha$ 
on these orbits.

This individualization of orbits happens whenever a property of
unbounded distortion (UD) holds between certain orbits of $f$, a
property which we describe precisely in the next section. 
Theorem~A shows that the (UD) property holds for a
residual set of $f$.  

\item The global part consists in proving
that for generic $f$, $\alpha$ is constant. 
We show that it is enough to verify that
the function $\alpha$ is bounded. This would be the case if the derivative $Df^n$ were to take
large values on each orbit of $f$, {\em for each large $n$}: the bound on $Dg$ would
then forbid $\alpha$ from taking arbitrarily large values.
Notice that this property is global in nature: we require large derivative of $f^n$ on each orbit, for each large $n$.

Because it holds for every large $n$, this large derivative (LD) property is not
generic, although we prove that it is dense. This lack of genericity
affects the structure of our proof: it is not possible to obtain
both (UD) and (LD) properties just by intersecting two residual
sets.  Theorem~B shows
that {\em among the diffeomorphisms satisfying (UD)}, the property
(LD) is dense.  This allows us to conclude that the set of
diffeomorphisms with trivial centralizer is $C^1$-dense.
\end{itemize}

There is some subtlety in
how we obtain a residual subset from a dense subset. Unfortunately we don't know
if the set of diffeomorphisms with
trivial centralizer form a $G_\delta$, i.e., a countable
intersection of open sets. For this reason, we consider centralizers defined inside of the
larger space of bi-Lipschitz homeomorphisms, and we use the compactness properties of this space. 
The conclusion is that if a $C^1$-dense set of diffeomorphisms has
trivial centralizer inside of the space of bi-Lipschitz
homeomorphisms, then this property holds on a $C^1$ residual set.

\subsection{Perturbations for obtaining the  (LD) and (UD) properties}\label{s.conclusion} 
To complete the proof of
the Main Theorem, it remains to prove Theorems A
and B. Both of these results split in two parts.  
\begin{itemize} 
\item The first
part is a local perturbation tool, which changes the derivative of $f$ in a very small
neighborhood of  a point, the neighborhood being chosen so small
that $f$ resembles a linear map on many iterates of this
neighborhood. 

\item In the second part, we perform the perturbations provided by the
first part at different places in such a
way that the derivative of every (wandering or non-periodic) orbit
will be changed in the desirable way. 
For the (UD) property on the
wandering set, the existence of open sets disjoint from all its
iterates are very helpful, allowing us to spread the perturbation
out over time. For the (LD) property, we need to control every
non-periodic orbit. The existence of \emph{topological towers} with
very large return time,  constructed in \cite{BC}, are the main tool,
allowing us again to spread the perturbations out over a long time
interval. 
\end{itemize}

\subsection*{Acknowledgments} This paper grew out of several visits between the 
authors hosted by their home institutions.  The authors would like to thank the 
Institut de Math\'ematiques de Bourgogne, the Northwestern University Math Department, and
the Institut Galil\'ee of the Universit\'e Paris 13 for their hospitality and support.
This work was supported by NSF grants DMS-0401326 and DMS-0701018.  We also thank 
Andr\'es Navas, who initially called our attention to the Lipschitz centralizer,
and who pointed out several references to us.

%% file: strategy0402.tex
\section{The local and global strategies}

In the remaining six sections, we prove the Main Theorem, following
the outline in the Introduction.  In this section, we reduce the
proof to two results, Theorems A and B, that together give
a dense set of diffeomorphisms with the (UD) and (LD) properties.

\subsection{Background on $C^1$-generic dynamics}\label{s.preliminaries}
The space $\Diff^1(M)$ is a Baire space in the $C^1$ topology. A
{\em residual} subset of a Baire space is one that contains a
countable intersection of open-dense sets; the Baire category
theorem implies that a residual set is dense.  We say that a
property holds for the {\em $C^1$-generic diffeomorphism} if it holds
on a residual subset of $\Diff^1(M)$.

For example, the Kupka-Smale Theorem asserts (in part)
that for a $C^1$-generic diffeomorphism $f$, the periodic orbits of 
$f$ are all hyperbolic. It is easy to verify that, furthermore, 
the $C^1$-generic diffeomorphism $f$ has the following property:
if $x,y$ are periodic points of $f$ with period
$m$ and $n$ respectively, and if their orbits are distinct, 
then the set of eigenvalues of $ Df^m(x)$ and of $Df^n(y)$ are disjoint. 
If this property holds, we
say that the \emph{periodic orbits of $f$ have distinct
eigenvalues}.

Associated to any homeomorphism $f$ of a compact metric space $X$
are several canonically-defined, invariant compact subsets that contain
the points in $X$ that are recurrent, to varying degrees, under $f$.
Here we will use three of these sets, which are the {\em closure of
the periodic orbits}, denoted here by $\overline{\Per(f)}$,
the {\em nonwandering set} $\Omega(f)$, and the {\em chain recurrent
set} $CR(f)$.  By the canonical nature of their construction, the
sets $\overline{\Per(f)}$,
$\Omega(f)$ and $CR(f)$ are all preserved by any homeomorphism $g$ that
commutes with $f$.

We recall their definitions. 
The nonwandering set  $\Omega(f)$ is the set of all points $x$
such that every neighborhood $U$ of $x$ meets some iterate of $U$:
$$U\cap \bigcup_{k>0} f^k(U) \ne \emptyset.$$

The chain recurrent set $CR(f)$ is the set of {\em chain recurrent
points} defined as follows. Given $\eps>0$, we say that a point $x$
is $\eps$-recurrent, and write $x\sim_\eps x$, if there exists an
$\eps$-pseudo-orbit, that is a sequence of points $x_0, x_1,\ldots, x_k$, $k\geq 1$
satisfying $d(f(x_i), x_{i+1}) < \eps$, for $i=0,\ldots, k-1$, such that
$x_0=x_k=x$. Then $x$ is {\em
chain recurrent} if $x\sim_\eps x$, for all $\eps>0$.
Conley theory implies that the complement of $CR(f)$
is the union of sets of the form $U\setminus \overline {f(U)}$
where $U$ is an open set which is attracting: $\overline {f(U)}\subset U$.
The chain-recurrent set is partitioned into compact invariant sets
called the \emph{chain-recurrence classes}: two points $x,y\in CR(f)$
belong to the same class if one can join  $x$ to $y$ and $y$ to $x$
by $\eps$-pseudo-orbits for every $\eps>0$.
 
It is not difficult to see that for any $f$, the inclusions 
$\overline{\Per(f)} \subseteq \Omega(f)\subseteq CR(f)$ hold;  there exist examples where the inclusions are strict. For $C^1$ generic diffeomorphisms $f$, however, all three sets coincide; 
$\overline{\Per(f)} =\Omega(f)$ is a consequence of Pugh's closing lemma~\cite{Pu},
and $\Omega(f) = CR(f)$ was shown much more recently in \cite{BC}.

We have additional links between $\Om(f)$ and the periodic points in the case it has 
non-empty interior:
\begin{theo*}[\cite{BC}] For any diffeomorphism $f$ in a residual subset of $\diff^1(M)$,  any connected component $O$ of the interior of $\Omega(f)$ is contained in the closure
of the stable manifold of a periodic point $p\in O$.
\end{theo*}

Conceptually, this result means that for $C^1$ generic $f$, the
interior of $\Omega(f)$ and the {\em wandering set} $M\setminus \Omega(f)$
share certain nonrecurrent features.
While points in the interior of $\Omega(f)$ all have nonwandering dynamics,
if one instead considers the restriction of $f$ to a stable manifold of a periodic orbit
$W^s(p)\setminus \cO(p)$, the dynamics are no longer recurrent; in the
induced topology on the submanifold $W^s(p)\setminus \cO(p)$, every point
has a {\em wandering neighborhood} $V$ whose iterates are all disjoint from $V$.
Furthermore, the sufficiently large future iterates of such a 
wandering neighborhood are contained in a neighborhood of the periodic orbit.
While the forward dynamics
on the wandering set are not similarly ``localized'' as they are on a stable manifold,
they still share this first feature: on the wandering set, 
every point has a wandering neighborhood (this time
the neighborhood is in the topology on $M$).

Thus, the results in \cite{BC} imply
that for the $C^1$ generic $f$, we have the following picture: 
there is an $f$-invariant open and dense subset $W$ of $M$, consisting 
of the union of the interior of $\Omega(f)$ and the complement 
of $\Omega(f)$, and densely in $W$ the dynamics of $f$ 
can be decomposed into components with ``wandering strata.''  
We exploit this fact
in our local strategy, outlined in the next section.

\subsection{Conditions for the local strategy: the unbounded distortion (UD) properties}

In the local strategy, we control the dynamics of the $C^1$ generic $f$ on the
open and dense set $W=\interior(\Omega(f))\cup \left(M\setminus\Omega(f)\right)$.
We describe here the main analytic properties we use to
control these dynamics.

We say that diffeomorphism $f$ satisfies the \emph{unbounded distortion
property  on the wandering set (UD$^{M\setminus\Omega}$)} if there
exists a dense subset $\cX\subset M\setminus \Omega(f)$ such that,
for any $K>0$, any $x\in \cX$ and any $y\in M\setminus \Omega(f)$ 
not in the orbit of $x$,  there exists $n\geq 1$ such that:
$$|\log |\Det Df^n(x)|-\log|\Det Df^n(y)||>K.$$

A diffeomorphism $f$ satisfies the \emph{unbounded distortion
property  on the stable manifolds (UD$^s$)} if for any hyperbolic
periodic orbit $\cO$, there exists a dense subset $\cX\subset
W^s(\cO)$ such that, for any $K>0$, any $x\in \cX$ and any $y\in W^s(\cO)$
not in the orbit of $x$, there exists $n\geq 1$ such that:
$$|\log |\Det Df_{|W^s(\cO)}^n(x)|-\log|\Det Df_{|W^s(\cO)}^n(y)||>K.$$

Our first main perturbation result is:
\begin{theoA}[Unbounded distortion]
The diffeomorphisms in a residual subset of $\diff^ 1(M)$ satisfy
the (UD$^{M\setminus\Omega}$) and  the (UD$^s$) properties.
\end{theoA}

A variation of an argument due to Togawa \cite{To1,To2} detailed in 
\cite{BCW} shows the (UD$^s$) property holds for a $C^1$-generic diffeomorphism. 
To prove Theorem A, we
are thus left to prove that the (UD$^{M\setminus\Omega}$) property holds
for a $C^1$-generic diffeomorphism.  This property is significantly
more difficult to establish $C^1$-generically than the (UD$^s$) property.
The reason is that points on the stable manifold of a periodic
point all have the same future dynamics, and these dynamics are
``constant'' for all large iterates: in a neighborhood of the periodic
orbit, the dynamics of $f$ are effectively linear.  In the wandering
set, by contrast, the orbits of distinct points can be completely unrelated
after sufficiently many iterates.

Nonetheless, the proofs that the (UD$^{M\setminus\Omega}$) and (UD$^s$) properties
are $C^1$ residual share some essential features, and both rely on the essentially non-recurrent
aspects of the dynamics on both the wandering set and the stable manifolds.

\subsection{Condition for the global strategy: the large derivative (LD) property}

Here we describe the analytic condition on the $C^1$-generic $f$ we
use to extend the local conclusion on the centralizer of $f$ to a global conclusion.

A diffeomorphism $f$ satisfies the \emph{large derivative property (LD) on a set $X$}
if, for any $K>0$, there exists $n(K)\geq 1$ such that
for any $x\in X$ and $n\geq n(K)$, there exists
$j\in \ZZ$ such that:
$$\sup\{\|Df^n(f^j(x))\|,\|Df^{-n}(f^{j+n}(x))\|\}>K;$$
more compactly:
$$\lim_{n\to\infty} \inf_{x\in X} \sup_{y\in{\small\orb}(x)} \{Df^n(y), Df^{-n}(y)\} = \infty.$$
Rephrased informally, the (LD) property on $X$ means that 
the derivative $Df^n$ ``tends to $\infty$'' {\em uniformly} on 
all orbits passing through $X$.  
We emphasize that the large derivative property is a property of the
{\em orbits} of points in $X$, and if it holds for $X$, it also
holds for all iterates of $X$.

Our second main perturbation result  is:
\begin{theoB}[Large derivative]
Let $f$ be a diffeomorphism whose periodic orbits are hyperbolic.
Then, there exists a diffeomorphism $g$ arbitrarily close to $f$ in $\diff^1(M)$
such that the property (LD) is satisfied on $M\setminus \Per(f)$.

Moreover,
\begin{itemize}
\item $f$ and $g$ are conjugate via a homeomorphism $\Phi$, i.e. $g=\Phi f\Phi^{-1}$;
\item for any periodic orbit $\cO$ of $f$, the derivatives of $f$ on $\cO$
and of $g$ on $\Phi(\cO)$ are conjugate (in particular the periodic orbits of $g$ are hyperbolic);
\item if $f$ satisfies the (UD$^{M\setminus\Omega}$) property, then so does $g$;
\item if $f$ satisfies the (UD$^s$) property, then so does $g$.
\end{itemize}
\end{theoB}
\bigskip

As a consequence of Theorems A and B we obtain:
\begin{coro}\label{c.UDLD}
There exists a dense subset $\cD$ of $\diff^1(M)$
such that any $f\in \cD$ satisfies the following properties:
\begin{itemize}
\item the periodic orbits are hyperbolic and have distinct eigenvalues;
\item any component $O$ of  the interior of $\Omega(f)$ contains a periodic point whose stable manifold is dense in $O$;
\item $f$ has the (UD$^{M\setminus\Omega}$) and the (UD$^s$) properties;
\item $f$ has the (LD) property on $M\setminus \Per(g)$.
\end{itemize}
\end{coro}

\subsection{Checking that the centralizer is trivial}
We now explain why properties (UD) and (LD) together imply that the centralizer is trivial.
\begin{proposition}\label{p=denseC1}
Any diffeomorphism $f$ in the $C^1$-dense subset $\cD\subset\Diff^1(M)$ given by Corollary~\ref{c.UDLD}
has a trivial centralizer $Z^1(f)$.
\end{proposition}

\begin{demo}[Proof of Proposition~\ref{p=denseC1}]
Consider a diffeomorphism $f\in \cD$.
Let $g\in Z^{1}(f)$ be a diffeomorphism commuting with $f$, 
and let $K>0$ be a Lipschitz constant for $g$ and $g^{-1}$.
Let $W = \interior(\Omega(f))\cup\left(M\setminus \Omega(f) \right)$
be the $f$-invariant, open and dense subset of $M$ whose properties are discussed in Section~\ref{s.preliminaries}.

Our first step is to use the ``local hypotheses'' (UD$^{M\setminus\Omega}$)
and (UD$^s$) to construct a function $\alpha\colon W \to\ZZ$ that is constant on each connected component
of $W$ and satisfies $g=f^\alpha$. We then use the ``global hypothesis'' (LD) to show that $\alpha$ is 
bounded on $W$, and therefore extends to a constant function on $M$.

We first contruct $\alpha$ on the wandering set $M\setminus\Omega(f)$.
The basic properties of Lipschitz functions and the relation $f^n g = g f^n$ imply that
 for any $x\in M$, and any $n\in \ZZ$, we have
\begin{eqnarray}\label{e=bounded}
|\log \det(Df^n(x))-\log \det(Df^n(g(x)))|\leq 2d \log K,
\end{eqnarray}
where $d=\dim M$. On the other hand, $f$ satisfies the
UD$^{M\setminus\Omega(f)}$ property, and hence there exists a dense subset
$\cX \subset M\setminus\Omega(f)$, each of whose points has unbounded
distortion with respect to any point in the wandering set not on the
same orbit. That is, for any $x\in \cX$, and $y\in M\setminus \Omega(f)$
not on the orbit of $x$, we have:
$$\limsup_{n\to \infty} |\log |\det Df^n(x)|-\log|\det Df^n(y)|| = \infty.$$
Inequality (\ref{e=bounded}) then implies that $x$ and $y=g(x)$ lie on
the same orbit, for all $x\in\cX$, hence $g(x)=f^{\alpha(x)}(x)$.
Using the continuity of $g$ and the fact that the points
in $M\setminus\Omega(f)$ admit wandering neighborhoods whose  $f$-iterates
are pairwise disjoint, we deduce that the map  $\alpha\colon \cX\to \ZZ$ is
constant in the neighborhood of any point in $M\setminus\Omega(f)$.
Hence the function $\alpha$ extends on $M\setminus\Omega(f)$ to a
function that is constant on each connected component of
$M\setminus\Omega(f)$. Furthermore,  $g=f^\alpha$ on
$M\setminus\Omega(f)$.

We now define the function $\alpha$ on the interior $\interior(\Omega(f))$ of the nonwandering set.
Since the periodic orbits of $f\in \cD$ have distinct
eigenvalues and since $g$ preserves the rate of convergence along the
stable manifolds, the diffeomorphism $g$ preserves each periodic orbit of $f$.
Using the $(UD^s)$ condition, one can extend the argument above for the
wandering set to the stable manifolds of each periodic orbit (see also~\cite[Lemma 1.2]{BCW}).
We obtain that for any periodic point $p$, the
diffeomorphism $g$ coincides with a power $f^\alpha$ on each connected
component of $W^s(p)\setminus \{p\}$. For $f\in\cD$, each connected
component $O$ of the interior of $\Omega(f)$ contains a periodic
point $p$ whose stable manifold is dense in $O$. It follows that
$g$ coincides with some power $f^\alpha$ of $f$ on each connected
component of the interior of $\Omega(f)$.

We have seen that there is a locally constant function $\alpha\colon W\to\ZZ$ 
such that $g=f^\alpha$ on the $f$ invariant, open and dense subset $W\subset M$. 
We now turn to the global strategy.
Notice that, since
$f$ and $g$ commute, the function $\alpha$ is constant along the
non-periodic orbits of $f$. Now $f\in\cD$ satisfies the (LD) property. 
Consequently there exists $N>0$ such that, for every non-periodic point
$x$, and for every $n\geq N$ there is a point $y=f^i(x)$ such that
either $\|Df^n(y)\|>K$ or $\|Df^{-n}(y)\|>K$. This implies that the
function $|\alpha|$ is bounded by $N$: otherwise, $\alpha$ would be
greater than $N$ on the invariant open set $W$ of $M$. This open set
contains a non-periodic point $x$ and an iterate $y=f^i(x)$ such
that either $\|Df^\alpha(y)\|>K$ or $\|Df^{-\alpha}(y)\|>K$. This
contradicts the fact that $g$ and $g^{-1}$ are $K$-Lipschitz.

We have just shown that $|\alpha|$ is bounded by some integer $N$.
Let $\Per_{2N}$ be the set of periodic points of $f$ whose period is less than $2N$,
and for $i\in \{-N,\dots,N\}$ consider the set
$$P_i=\{x\in M\setminus \Per_{2N},\; g(x)=f^i(x)\}.$$
This is a closed invariant subset of $M\setminus \Per_{2N}$.
What we proved above implies that $M\setminus \Per_{2N}$
is the union of the sets $P_i$, $|i|\leq N$.
Moreover any two sets $P_i,P_j$ with $i\neq j$ are disjoint since a point in $P_i\cap P_j$
would be $|i-j|$ periodic for $f$.

If $\dim(M)\geq 2$, since $M$ is connected and $\Per_{2N}$ is finite, the set
$M\setminus \Per_{2N}$ is connected. It follows that only one set $P_i$ is non-empty,
implying that $g=f^i$ on $M$. This concludes the proof in this case.

If $\dim(M)=1$, one has to use that $g$ is a diffeomorphism and is not only Lipschitz:
this shows that on the two sides of a periodic orbit of $f$,
the map $g$ coincides with the same iterate of $f$. This proves again that only one set $P_i$
is nonempty.
\end{demo}

\subsection{From dense to residual: compactness and semicontinuity}
The previous results show that the set of diffeomorphisms having a trivial centralizer
is dense in $\diff^1(M)$, but this is not enough to conclude the proof of the Main Theorem.
Indeed the dense subset $\cD$ in Corollary~\ref{c.UDLD} is {\em not} a
residual subset if $\dim(M)\geq 2$: in the appendix
we exhibit a nonempty open set in which $C^1$-generic
diffeomorphisms do not satisfy the (LD)-property.

Fix a metric structure on $M$. A homeomorphism $f:M\to M$ is {\em
$K$-bi-Lipschitz} if both $f$ and $f^{-1}$ are Lipschitz, with
Lipschitz norm bounded by $K$.  A homeomorphism that is
$K$-bi-Lipschitz for some $K$ is called a  {\em bi-Lipschitz
homeomorphism}, or {\em lipeomorphism}.  We denote by $\Lip^K(M)$
the set of $K$-bi-Lipschitz homeomorphisms of $M$ and by $\Lip(M)$
the set of bi-Lipschitz homeomorphisms of $M$. The Arz{\`e}la-Ascoli
theorem implies that $\Lip^K(M)$ is compact in the uniform  ($C^0$)
topology. Note that $\Lip(M)\supset \Diff^1(M)$.
For $f\in \Lip(M)$, the set $Z^{\Lip}(f)$ is defined analogously to the $C^r$ case:
$$Z^{\Lip}(f):=\{g\in \Lip(M): fg=gf\}.$$

In dimension 1, the Main Theorem was a consequence of Togawa's work~\cite{To2}.
In higher dimension, the Main Theorem is a direct corollary of:
\begin{theo}\label{t=liptriv}
If $\dim(M)\geq 2$,
the set of diffeomorphisms $f$ with trivial centralizer $Z^{\Lip}(f)$ is residual in $\Diff^1(M)$.
\end{theo}
The proof of Theorem~\ref{t=liptriv} has two parts.

\begin{prop}\label{p=denselip}
If $\dim(M)\geq 2$,
any diffeomorphism $f$ in the $C^1$-dense subset $\cD\subset\Diff^1(M)$ given by Corollary~\ref{c.UDLD}
has a trivial centralizer $Z^{\Lip}(f)$.
\end{prop}
The proof of this proposition from Theorems A and B is the same as the proof of Proposition~\ref{p=denseC1} (see also Lemma~1.2 in \cite{BCW}).

\begin{prop}\label{p=liptocr}
Consider the set $\cT$ of diffeomorphisms $f\in \Diff^1(M)$ having a trivial centralizer $Z^{\Lip}(f)$.
Then, if $\cT$ is dense in $\Diff^1(M)$, it is also residual.
\end{prop}

\begin{remark}
The proof of Proposition~\ref{p=liptocr} also holds in the $C^r$ topology
$r\geq 2$ on any manifold $M$ on which the $C^r$-generic
diffeomorphism has at least one hyperbolic periodic orbit (for example, on
the circle, or on manifolds of nonzero Euler characteristic).
On the other hand, Theorem~\ref{t=liptriv} is false for general manifolds in the $C^2$ topology,
at least for the circle.
In fact, a simple folklore argument (see the proof of Theorem B in~\cite{navas}) 
implies that for {\em any} Kupka-Smale diffeomorphism
$f\in \Diff^2(S^1)$, the set $Z^{\Lip}(f)$ is {\em infinite dimensional}.
It would be interesting to find out what is true in higher dimensions.
\end{remark}

\begin{demo}[Proof of Proposition~\ref{p=liptocr}]
For any compact metric space $X$ we denote
by $\cK(X)$ the set of non-empty compact subsets of $X$, endowed with the Hausdorff distance $d_H$.
We use the following classical fact. 

\begin{prop*}\label{p=semi} Let $\cB$ be a Baire space,
let $X$ be a compact metric space, and let $h:\cB\to \cK(X)$ be an
upper-semicontinuous function. Then the set of continuity points of
$h$ is a residual subset of $\cB$.

In other words, if $h$ has the property that for all $b\in \cB$,
$$b_n\to b\,\implies\, \limsup b_n= \bigcap_n\overline{\bigcup_{i>n} h(b_i)} \subseteq h(b),$$
then there is a residual set $\cR_h\subset \cB$ such that, for all
$b\in \cR_h$,
$$b_n\to b\implies\, \lim d_H(b_n,b)=0.$$
\end{prop*}

To prove Proposition~\ref{p=liptocr}, we note that for a fixed $K>0$,
the set $Z^{\Lip}(f)\cap \Lip^K(M)$ is a closed subset (in the $C^0$
topology) of the compact metric space $\Lip^K(M)$.  This is a simple
consequence of the facts that $Z^{\Lip}(f)$ is defined by the
relation $fgf^{-1}g^{-1} = \id$, and that composition and inversion
are continuous. Thus there is well-defined map $h_K$ from
$\Diff^1(M)$ to $\cK(\Lip^K(M))$, sending $f$ to $h_K(f) =
Z^{\Lip}(f)\cap \Lip^K(M)$. It is easy to see that $h_K$ is
upper-semicontinuous: if $f_n$ converges to $f$ in $\diff^1(M)$ and
$g_n\in h_K(f_n)$ converges uniformly to $g$ then $g$ belongs to
$h_K(f)$.

Let $\cR_K\subset \Diff^1(M)$
be the set of points of continuity of $h_K$; it is
a residual subset of $\Diff^1(M)$, by Proposition~\ref{p=semi}.  Let
$\cR_{Hyp}\subset \Diff^1(M)$ be the set of diffeomorphisms
such that each $f\in \cR_{Hyp}$ has at least one
hyperbolic periodic orbit (the $C^1$ Closing Lemma implies
that $\cR_{Hyp}$ is residual).  Finally, let
$$\cR = \cR_{Hyp}\cap \bigcap_{K=1}^{\infty} \cR_K.$$

Assuming that $\cT$ is dense in $\Diff^1(M)$,
we claim that the set $\cR$ is contained in $\cT$, implying that $\cT$ is residual.
To see this, fix $f\in \cR$, and let $f_n\to f$ be a sequence of
diffeomorphisms in $\cT$ converging to $f$ in the $C^1$ topology.
Let $g\in Z^{\Lip}(M)$ be a $K$-bi-Lipschitz homeomorphism satisfying
$fg=gf$.  Since $h_K$ is continuous at $f$, there is a sequence
$g_n\in Z^{\Lip}(f_n)$ of $K$-bi-Lipschitz homeomorphisms
with $g_n\to g$ in the $C^0$ topology.
The fact that $f_n\in\cT$ implies that
the centralizer $Z^{\Lip}(f_n)$ is trivial,  so
there exist integers $m_n$ such that $g_n = f^{m_n}$.

If the sequence $(m_n)$ is bounded, then passing to a subsequence,
we obtain that $g = f^m$, for some integer $m$.
If  the sequence $(m_n)$ is
not bounded, then we obtain a contradiction as follows.  Let $x$ be a
hyperbolic periodic point of $f$, of period $p$.
For $n$ large, the map $f_n$ has a periodic orbit $x_n$ of period $p$,
and the derivatives $Df^p_n(x_n)$ tend to the derivative $Df^p(x)$.
But then $|\log\|Df_n^{m_n}\||$ tends to infinity as $n\to\infty$.
This contradicts the fact that the diffeomorphisms
$f^{m_n}_n=g_n$ and $f^{-m_n}_n=g_n^{-1}$ are both $K$-Lipschitz,
concluding the proof.
\end{demo}

%% file: tidy0402.tex
\section{Perturbing the derivative without changing the dynamics} 
In order to prove Theorem B, one needs to perturb a diffeomorphism $f$ and change the dynamical properties of its derivative without changing its topological dynamics: the resulting diffeomorphism
is still conjugate to $f$. We develop in this section an important technique for the proof,
which we call  \emph{tidy perturbations} of the dynamics.

\subsection{Tidy perturbations}

\begin{defi} Let $f:M\to M$ be a homeomorphism and let $X\subset M$.  We say that a homeomorphism
$g$ is a {\em tidy perturbation of $f$ supported on $X$} if:
\begin{enumerate}
\item $g(x) = f(x)$, for all $x\in M\setminus X$,
\item if $x\in M\setminus X$ and $f^m(x)\in M\setminus X$, for some $m\geq 1$,
then $g^m(x) = f^m(x)$; 
\item if $x\in M\setminus X$ and $g^m(x)\in M\setminus X$, for some $m\geq 1$,
then $g^m(x) = f^m(x)$;
\item for all $x\in M$, there exists $m\in \ZZ$ such that 
$f^{m}(x)\in M\setminus \overline{X}$;
\item for all $x\in M$, there exists $m\in \ZZ$ such that 
$g^{m}(x)\in M\setminus \overline{X}$.
\end{enumerate}
\end{defi}
Note that this definition is symmetric in $f$ and $g$.

\begin{lemm} If $g$ is a tidy perturbation of $f$ supported on $X$, then $g$ is conjugate to $f$ by a homeomorphism $\varphi$ such that $\varphi=\id$ on $M\setminus (X\cap f(X))$.  
Furthermore, if $g$ and $f$ are diffeomorphisms, then $\varphi$ is a diffeomorphism as well.
\end{lemm}

\begin{proof} Given $f$ and a tidy perturbation $g$ of $f$ supported on $X$, we construct a homeomorphism
$\varphi$ as follows.  Property 4 of tidy perturbations implies that for each $x\in M$, there exists an integer
$m_x$ such that $f^{m_x}(x)\in M\setminus \overline{X}$.  We set $\varphi(x) = g^{-m_x} f^{m_x}(x)$.

Then $\varphi$ is well-defined, for suppose that $f^{m_1}(x)\notin \overline{X}$ and  $f^{m_2}(x)\notin \overline{X}$,
for some integers $m_1<m_2$.  Let $y= f^{m_1}(x)$.  Then $y\notin \overline{X}$, and 
$f^{m_2-m_1}(y) = f^{m_2}(x) \notin \overline{X}$.  Property 2 of tidy perturbations implies
that $g^{m_2-m_1}(y) = f^{m_2-m_1}(y)$; in other words, $g^{-m_1}f^{m_1}(x) = g^{-m_2} f^{m_2}(x)$.
Hence the definition of $\varphi(x)$ is independent of the integer $m_x$.
In particular $g\circ \varphi=\varphi\circ f$.

To see that $\varphi$ is continuous, note that for every $x\in M$, if $f^{m}(x)\notin \overline{X}$, then there exists
a neighborhood $U$ of of $x$ such that $f^m(U)\cap \overline{X} \neq \emptyset$.  Hence $\varphi= g^{-m}\circ f^m$ on
$U$.  This implies that $\varphi$ is a local homeomorphism.

Let $\varphi^-$ be the local homeomorphism obtained by switching the roles of $f$ and $g$.
Clearly $\varphi^{-}$ is the inverse of $\varphi$, so that $\varphi$ is a homeomorphism.
If $f$ and $g$ are diffeomorphisms, it is clear from the construction of $\varphi$ that $\varphi$ is a diffeomorphism.

Properties 1 and 2 of tidy perturbations imply that
for any point $x\in M\setminus (X\cap f(X))$, we have $f^{m_x}(x)=g^{m_x}(x)$.
This gives $\varphi=\id$ on $M\setminus (X\cap f(X))$.
\end{proof}

Each the tidy perturbations we will consider is supported in the union of succesive disjoint iterates of an open set. 
In that case, the characterization of tidy perturbations is much easier, as explained in the next lemma:  

\begin{lemm} Let $f$ be a homeomorphism, $U\subset M$ an open set,
and $m\geq 1$ an integer such that the iterates 
$\overline{U}, f(\overline{U}), \ldots , f^{m}(\overline{U})$ are pairwise disjoint.

Let $g$ be a homeomorphism such that:
\begin{itemize}
\item $g=f$ on the complement of the union $X=\bigcup_{i=0}^{m-1} f^i(U)$, and
\item for all $x\in U$,  $f^{m}(x) = g^{m}(x)$.
\end{itemize}
Then $g$ is a tidy perturbation of $f$ supported on $X$.
\end{lemm}

\begin{proof}
Property 1 of tidy perturbations is immediate; property 4 is as well, since
$\overline{U}, f(\overline{U}), \ldots , f^{m}(\overline{U})$ are pairwise disjoint.

Let $x$ be a point in $M\setminus X$ whose $f$-orbit enters and leaves $X$, and let $i>0$ be its first entry in $X$;  then $f^i(x) = g^i(x)$ by the first hypothesis on $g$. The $f$-orbit of $x$ leaves $X$ at $f^{m+i}(x) = g^{m+i}(x)$, by the second hypothesis on $g$. Proceeding inductively on the successive entrances and exits of the $f$-orbit of $x$ we get that $g^n(x)=f^n(x)$ for every integer $n$ such that $f^n(x)\in M\setminus X$, proving Property 2 of tidy perturbations.

%Since $f=g$ in a neighborhood of $M\setminus U$, it follows
%that $f(U)=g(U)$. 

There is a neighborhood $V$ of $\bar U$ such that $f=g$ on $V\setminus U$. It
follows that   $f(U)=g(U)$.
Inductively, we see that $g^i(U) = f^i(U)$ for $i\in\{1, \ldots , m\}$. 
This shows that the hypotheses of the lemma still are satisfied if we switch the roles of
$f$ and $g$. This implies Properties 3 and 5 of tidy perturbations.
\end{proof}

If $g$ is a small $C^1$-perturbation of $f$ that is a tidy perturbation supported in the disjoint union $\bigcup_{i=0}^{m-1} f^i(U)$, 
it is tempting to think that the conjugating diffeomorphism $\varphi$ must be  $C^1$-close to the identity as well. This is not always the case.
The derivative $Dg$ is a small perturbation of $Df$ but
for $x\in U$ and $i\in \{1,\dots,m\}$, the maps $Df^i(x)$ and $Dg^i(x)$ could be very different.
There is however a straighforward \emph{a priori} bound for $D\varphi$:   

\begin{lemm}\label{l.apriori} Let $f$ be a diffeomorphism, $U\subset M$ an open set
and  $m\geq 1$ an integer such that the iterates 
$\overline{U}, f(\overline{U}), \ldots , f^{m}(U)$ are pairwise disjoint.
Suppose that $g$ is a tidy perturbation of $f$ supported on $X=\bigcup_{i=0}^{m-1} f^i(\overline{U})$ and that  $\varphi:M\to M$ is the  diffeomorphism such that $g = \varphi \circ f\circ \varphi^{-1}$ and  $\varphi = \id$ on $M\setminus X$. 

Then  $$\max\{\|D\varphi\|, \|D\varphi^{-1}\|\} \leq C^{m},$$ where $C = \sup\{\|Dg\|, \|Dg^{-1}\|, \|Df\|, \|Df^{-1}\|\}$.  
\end{lemm}

\begin{demo} For every $x\in M$ there exists $i\in\{0,\dots,m\}$ such that $f^i(x)\notin \overline{X}$, so that $\varphi$ coincides with $g^{-i}f^i$ and $g^{-i+m}f^{i-m}$ in  a neighborhood of $x$.
We conclude by noting that either $2i$ or $2(m-i)$ is less than $m$.
\end{demo}

\subsection{Sequences of tidy perturbations}

The aim of this section is to control the effect of infinitely many successive tidy perturbations $g_i$ of a diffeomorphism $f$ and to give a sufficient condition for  the sequence $g_i$ to converge to a diffeomorphism $g$ conjugate to $f$ by a homeomorphism.  

\begin{lemm}\label{l=conjugacy}
Let $f$ be a diffeomorphism and $\ve,C>0$ be constants such that
for every $g$ with $d_{C^1}(f,g)<\ve$, we have $\|Dg\|, \|Dg^{-1}\| < C$. Consider:

\begin{itemize}
\item $(U_i)_{i\geq 1}$, a sequence of open subsets of $M$,
\item $(m_i)_{i\geq 1}$, a sequence of positive integers,
\item $(\varepsilon_i)_{i\geq 1}$, a sequence of positive numbers such that $\sum \ve_i < \ve$, 
\item $(g_i)_{i\geq 0}$, a sequence of diffeomorphisms such that $g_0=f$ and for each $i\geq 1$,
\begin{itemize}
\item the sets $\overline{U_i}, g_{i-1}(\overline{U_i}), \ldots , g_{i-1}^{m_i}(\overline{U_i})$ are pairwise disjoint,
\item $g_{i}$ is a tidy $\varepsilon_{i}$-perturbation of $g_{i-1}$, supported in $\bigcup_{k=0}^{m_i-1}g_{i-1}^k(U_i)$, 
\end{itemize}
\item $(\rho_i)_{i\geq 1}$, a sequence of positive numbers such that $g_{i-1}^k(U_i)$ has diameter bounded by $\rho_i$, for $k\in\{0,\ldots,m_i\}$.  
\end{itemize}
Denote by $\varphi_i$ the  diffeomorphism such that $g_i = \varphi_i \circ g_{i-1}\circ \varphi_i^{-1}$ and  
$\varphi_i = \id$ on $M\setminus \bigcup_{k=1}^{m_i-1}g_{i-1}^k(U_i)$. 
Let $\Phi_i = \varphi_i\circ \cdots\circ \varphi_1$ and $M_i = \Pi_{k=1}^i C^{m_{k}}$.
If one assumes that
$$\sum_{i\geq 1} \rho_i M_{i-1}< \infty,$$
then
\begin{enumerate}
\item $(g_i)$ converges in the $C^1$ metric to a diffeomorphism $g$ with $d_{C^1}(f,g)<\ve$,
\item for all $i<j$ one has $d_{unif}(\Phi_j\Phi_i^{-1}, \id) < \sum_{k=i+1}^j \rho_k$,
\item $(\Phi_i)$ converges uniformly to a homeomorphism $\Phi$ satisfying $g=\Phi f \Phi^{-1}$.
\end{enumerate}
\end{lemm}

\begin{rema} The uniform metric $d_{unif}$ on continuous self-maps on $M$ 
in part 3 of this lemma is not complete for the space of homeomorphisms
of $M$; to avoid confusion, we denote by $d_{C^0}$ the complete metric defined by:
$$d_{C^0}(f,g) = d_{unif}(f,g) + d_{unif}(f^{-1},g^{-1}).$$
One difficulty in the proof of Lemma~\ref{l=conjugacy}, which explains the role of $M_i$, is
to control the distances $d_{unif}(\Phi_i^{-1}\Phi_j, \id)$.
\end{rema}

\begin{demo} The first conclusion is clear.
Since $\Phi_j\Phi_i^{-1} = \varphi_j\circ\cdots\circ \varphi_{i+1}$, the second one is immediate from the
hypothesis that $\varphi_i$ is a tidy perturbation and the fact that the connected components of the support of $\varphi_i$ have diameter less than $\rho_i$.
The hypothesis $\sum_{k\geq 1}\rho_i < \infty$ implies that $(\Phi_i)_{i\in\NN}$ is a Cauchy sequence and converges to a continuous
map $\Phi$ satisfying $g\Phi =\Phi f$. Finally, the a priori 
bound in Lemma~\ref{l.apriori} implies that $\|D\Phi^{-1}_i\| < M_i$, so that
$$\sup_{d(x,y) < \rho_i} d(\Phi_{i-1}^{-1}(x), \Phi_{i-1}^{-1}(y)) < \rho_i M_{i-1}.$$
But for any $x$, we have $d(x,\varphi_i(x)) < \rho_i$. So the previous calculation
implies $d(\Phi_{i-1}^{-1}(x), \Phi_{i}^{-1}(x)) < \rho_iM_{i-1}$.  By
hypothesis, $\sum_{i\geq 1} \rho_i M_{i-1} < \infty$, which implies that $(\Phi_i^{-1})_{i\in\NN}$ is a Cauchy sequence in the $d_{unif}$ metric.  Hence
$(\Phi^{-1}_i)$ converges as $i\to \infty$ to the inverse of $\Phi$ and so $\Phi$ is a homeomorphism. 
\end{demo}

\subsection{Topological towers}

For each tidy perturbation we construct in this paper, we will use an open set with many disjoint iterates.
The large number of iterates will allow us to spread the effects of the perturbation out over
a large number of steps, effecting a large change in the derivative with a small perturbation. For these perturbations to have a global effect, we need to have most orbits in $M$ visit $U$. The technique
of {\em topological towers}, developed in \cite{BC}, allows us to choose $U$ to have many disjoint iterates, while simultaneously guaranteeing that most orbits visit $U$.

\begin{theo}[Topological towers, \cite{BC}, Th\'eor\`eme 3.1]\label{t.tower}
For any integer $d\geq 1$, there exists a constant $\kappa_d>0$ such that
for any integer $m\geq 1$ and for any diffeomorphism $f$ of a $d$-dimensional manifold $M$, 
whose periodic orbits of period less than $\kappa_d.m$ are hyperbolic,
there exists an open set $U$ and a compact subset $D\subset U$
having the following properties:

\begin{itemize}
\item Any point $x\in M$ that does not belong to a periodic orbit of period $< m$
has an iterate $f^i(x)$ in the interior of $D$. 
\item The sets $\bar U,f(\bar U),\dots,f^{m-1}(\bar U)$ are pairwise disjoint.
\end{itemize}
Moreover, the connected components of $\overline{U}$ can be chosen with arbitrarily small diameter.
\end{theo}

\subsection{Towers avoiding certain sets}\label{s.wanderingorbit}

In constructing our sequence $(g_i)$ of tidy perturbations, we need to ensure that the
effects of the $(i+1)$st perturbations do not undo the effects of the $i$th perturbation.
In addition, we aim to produce large derivative without affecting unbounded distortion. For these reasons,
it is desirable to choose the towers in the tidy perturbations to avoid certain subsets in $M$.
It is not possible to choose a tower avoiding an arbitrary subset, but it turns out that
certain sets with a {\em wandering orbit property} can be avoided.  We now define this property.

\begin{defi} Let $N, J\geq 1$ be integers.  An {\em $(N,J)$-wandering cover} of a compact set $Z$ 
is a finite cover $\cU$ of $Z$ such that every $V\in \cU$ has $N$ iterates
$f^j(V), f^{j+1}(V), \ldots, f^{j+N-1}(V)$ with $j\in \{1,\dots, J\}$ that are disjoint from $Z$.
\end{defi}

\begin{defi} Let $f$ be a homeomorphism and $Z\subset M$ be a compact set.

$Z$ has the {\em $(N,J)$-wandering orbit property} if it has an $(N,J)$-wandering cover.

$Z$ has the {\em $N$-wandering orbit property} if it has the $(N,J)$-wandering orbit property for some $J\geq 1$.

$Z$ has the {\em wandering orbit property} it has the $N$-wandering orbit property, for every $N\geq 1$.
\end{defi}

The next lemma explains that topological towers can be constructed avoiding any compact
set with the wandering orbit property.

\begin{lemm}\label{l.wotower} Suppose that the periodic orbits of $f$ are
all hyperbolic.

For all $m_1, m_2 \geq 1$, and $\rho>0$,
if $Z$ is any compact set with the $m_1$-wandering orbit property, then there
exists an open set $U\subset M$ with the following properties:
\begin{enumerate}
\item The diameter of each connected component of $\overline U$ is less than $\rho$.
\item The iterates $\overline U, f(\overline U), \ldots, f^{m_1-1}(\overline{U})$
are disjoint from the set $Z$.
\item The iterates $\overline U, f(\overline U), \ldots, f^{m_1+m_2-1}(\overline{U})$
are pairwise disjoint.
\item There is a compact set $D\subset U$ such that every nonperiodic point $x\in M$ has an iterate in $D$.
\end{enumerate}
\end{lemm}

This will be proved at the end of this subsection.

\begin{rema}

\begin{itemize}
\item If $Z$ has the $(N,J)$-wandering orbit property, then for every
$\rho>0$, there is an $(N,J)$-wandering cover all of whose elements
have diameter less than $\rho$. 
\item  If $Z$ has the $(N,J)$-wandering orbit property, then so does $f(Z)$.  
\item More generally, if $Z$ has the $(N,J)$-wandering orbit
property for $f$, and $g=\Phi f\Phi^{-1}$, then $\Phi(Z)$ has 
the $(N,J)$-wandering orbit property for $g$.
\end{itemize}
\end{rema}

\begin{lemm}\label{l.wosets}
\begin{enumerate}
\item\label{a} If $X$ is any compact set such that the iterates $X$, $f(X)$, \ldots, $f^{m+N}(X)$ are
pairwise disjoint, then $\bigcup_{i=0}^{m-1}f^{i}(X)$ has the $(N,m)$-wandering orbit property.
\item\label{b} If $x$ is any point whose period lies in $(N+1,\infty]$, then $\{x\}$ has the $(N,1)$-wandering orbit property.
\item\label{c} If $x$ is any nonperiodic point, then $\{x\}$ has the wandering orbit property.
\item\label{d} Any compact subset $Z\subset \cW^s(\orb(p))\setminus \orb(p)$ has the wandering orbit property, 
for any hyperbolic periodic point $p$.
\item\label{e} Any compact subset of the wandering set $M\setminus\Om(f)$ of $f$ has the wandering orbit property.
\end{enumerate}
\end{lemm}

\begin{demo} \ref{a}), \ref{b}) and \ref{c}) are easy.  To prove \ref{d}), we consider
$Z\subset W^s(\orb(p))\setminus \orb(p)$, for some hyperbolic periodic point $p$.  Let $\cN$ be a neighborhood of the orbit of $p$ that is disjoint from $Z$.  There exists $m>0$ such that
$f^n(Z)\subset \cN$ for any $n\geq m$. Given $N$, there is a covering $\cV$ of $Z$ such that, for every $V\in\cV$, 
the sets $f^m(V), \ldots, f^{m+N}(V)$ are all contained in $\cN$.  Then $\cV$ is an $(N,m)$-wandering cover of $Z$.

Finally we prove \ref{e}).
If $Z$ is a compact subset of $M\setminus \Om(f)$ then every point $z\in Z$ has finitely many returns in $Z$. So for every $N>0$ and every $x\in Z$ there is a neighborhood $U_x$ and an integer $j_x>0$ such that $f^{j_x}(U_x),\dots, f^{j_x+N-1}(U_x)$ are disjoint from $Z$. Extracting a finite cover $U_{x_i}$ and chosing $J=\max_i j_{x_i}$ we get an $(N,J)$-wandering cover. 
\end{demo}

\begin{lemm}\label{l.wounion} Suppose $Z$ has the $(N,J)$-wandering orbit property
and $W$ has the $(2N+J, K)$-wandering orbit property. Then
$Z\cup W$ has the $(N, K + 2J + 2N)$ wandering orbit property.

If $Z$ has the $N$-wandering orbit property and $W$ has the wandering orbit property, then $Z\cup W$ has
the $N$-wandering orbit property.

If $Z$ and $W$ both have the wandering orbit property, then $Z\cup W$ has
the wandering orbit property.
\end{lemm}

\begin{demo} The second two claims follow from the first. To prove the first,
fix a $(N,J)$-wandering cover $\cU_Z$ of $Z$ and a $(2N+J, K)$-wandering cover $\cV_W$ of $W$ with the property that, for
every $V \in \cV_W$ and every $i\in\{0,\ldots, N+K-1\}$, if
$f^i(V)$ intersects $Z$, then $f^i(V)$ is contained in an
element of $\cU_Z$.
Let $\cV_Z$  be an $(N,J)$-wandering cover of
$Z$ such that, for every $V \in \cV_Z$ and every
$i\in\{0,\ldots, N+J - 1\}$, if $f^i(V)$ intersects $W$, then
$f^i(V)$ is contained in an element of $\cV_W$.  We claim that $\cV_Z\cup \cV_W$ is an $(N, K + 2J + 2N)$-wandering
cover of $Z\cup W$: 

\begin{itemize}
\item Consider  first $V\in \cV_W$. As $\cV_W$ is a $(2N+J,K)$-wandering cover of $W$, there exists 
$k\in\{1,\dots, K\}$ such that 
$f^{k}(V),\dots, f^{k+2N+J-1}(V)$ are disjoint from $W$. If $f^{k}(V),\dots,f^{k+N-1}(V)$ are disjoint from $Z$, then
they are disjoint from $W\cup Z$ which is what we want. Otherwise, there exists $V'\in \cU_Z$ and $j\in\{1,\dots, N-1\}$ such that $f^{k+j}(V)\subset V'$. Now there exists $i\in\{1,\dots,J\}$ such that $f^i(V'),\dots, f^{i+N-1}(V')$ are disjoint from $Z$; let $j'=k+i+j\in \{1,\dots, K+J+N-1\}$.  Then $f^{k+i+j}(V),\dots, f^{k+i+j+N-1}(V)$ are disjoint from $Z\cup W$. 
\item Now consider $V\in\cV_Z$. As $\cV_Z$ is a $(N,J)$-wandering cover of $W$, there exists $j\in\{1,\dots, J\}$ 
such that 
$f^{j}(V),\dots, f^{j+N-1}(V)$ are disjoint from $Z$. If $f^{j}(V),\dots,f^{k+N-1}(V)$ are disjoint from $W$, then
they are also disjoint from $W\cup Z$, which is what we want. Otherwise, there exists $V'\in \cV_W$ and $i\in\{1,\dots, N-1\}$ such that $f^{i+j}(V)\subset V'$. By the previous case, we know there exists $k\in\{1,\dots,K+J=N-1\}$ such that the iterates $f^k(V'),\dots,f^{k+N-1}(V')$ are disjoint from $Z\cup W$; hence $j'=i+j+k\in\{1,\dots,K+2N+2J-1\}$, and the iterates $f^{j'}(V),\dots,f^{j'+N-1}(V)$ are disjoint from $Z\cup W$.
\end{itemize} 
\end{demo}

\begin{demo}[Proof of Lemma~\ref{l.wotower}]
Let $m_1, m_2 \geq 1$ and $\rho>0$ be given.  Let $Z$ be a compact set with the $m_1$-wandering property,
and let $\cU$ be an $(m_1,J)$-wandering cover of  $Z$.
Let $U_0$ be an open set given by Theorem~\ref{t.tower} with the following properties:
\begin{itemize}
\item for each connected component $O$ of $\overline{U_0}$ and each $j\in \{0,\ldots, m_1+J-1\}$, the diameter of
$f^j(O)$ is less than $\rho$;
\item for each component $O$ of $\overline{U_0}$, if $f^i(O)\cap Z\ne\emptyset$, for some
$i\in \{0,\ldots,m_1-1\}$, then $f^i(O)$ is contained in some element of $\cU$;
\item the iterates $\overline U_0, f(\overline U_0), \ldots, f^{2m_1 + m_2 + J - 1}(\overline U_0)$
are pairwise disjoint;
\item there is a compact set $D_0\subset U_0$ such that every nonperiodic point $x\in M$ has an iterate in $D_0$.
\end{itemize}
For each component $O$ of $\overline U_0$, either the first $m_1$ iterates $O, f(O), \ldots, f^{m_1-1}(O)$ 
are disjoint from $Z$, or one of these iterates $f^i(O)$ is contained in an element of $\cU$.
In the latter case, there exists a $j\in \{1,\ldots, m_1+J-1\}$ such that the iterates
$f^j(O),\ldots, f^{j+m_1-1}(O)$ are disjoint from $Z$.  
In either case, there exists
an iterate $f^j(O)$ of $O$, $j\in \{0,\ldots, m_1+J-1\}$ of diameter less than $\rho$, 
whose first $m_1$ iterates are disjoint from $Z$.
Selecting one such iterate  $f^j(O)$ for each component of $\overline{U_0}$
and taking the union of the $f^j(U_0\cap O)$ over all the components $O$,
we obtain our desired set $U$. The set $D$ is obtained by intersecting the first $2m_1+m_2 +J-1$ iterates of $D_0$ with the  components of $U$ and taking their union. One hence obtain the conclusions 1, 2 and 4.

The sets of the form $f^{j+k}(U_0\cap O)$, where the sets $f^j(U_0\cap O)$ define $U$
and $k\in \{0,\dots,m_1-1\}$, are pairwise disjoint, since $\overline{U_0}$ is disjoint from
its $2m_1+m_2+J-1$ first iterates. This gives conclusion 3 and ends the proof of the lemma.
\end{demo}

\subsection{Linearizing the germ of dynamics along a nonperiodic orbit}\label{ss.linearization}

In this subsection we present an elementary tool which will be used to construct perturbations of a diffeomorphism
$f$ using perturbations of linear cocycles.
\paragraph{Lift of $f$ by the exponential map.}
Let $M$ be a compact manifold endowed with a Riemannian metric. Denote by $\exp\colon TM\to M$ the exponential map associated to this metric, and by $R>0$ the radius of injectivity of the exponential map: for every $x\in M$ the exponential map at $x$ induces a diffeomorphism $\exp_x\colon B_x(0,R)\to B(x,R)$ where $B_x(0,R)\subset T_xM$ and $B(x,R)$ are the balls of radius $R$ centered at $0_x\in T_xM$ and at $x$, respectively. It is important for our purposes to notice that the derivative $D_x\exp_x$ can be identified with the identity map of $T_xM$. 

Let $f$ be a diffeomorphism and let $C>1$ be a bound for $\|Df\|$ and for $\|Df^{-1}\|$. 
For each vector $u\in T_xM$ with  $\|u\|< R$ and $d(f(x),f(\exp_x(u)))< R$,  we define
$$\tilde f_x(u)=\exp_{f(x)}^{-1}(f(\exp_x(u))).$$
This formula defines a diffeomorphism $\tilde f\colon U\to V$ where
$$U=\{(x,u)\in TM \mid \|u\|<R\quad\hbox{ and }\quad d(f(x),f(\exp_x(u)))<R\}$$
$$V=\{(x,u)\in TM \mid  \|u\|<R\quad\hbox{ and }\quad  d(f^{-1}(x),f^{-1}(\exp_x(u)))<R\}.$$
The sets $U$ and $V$ are neighborhoods of the zero-section of the tangent bundle, both containing the set of vectors whose norm is bounded by $C^{-1}R$.  We denote by $\tilde f_x\colon U_x\to V_{f(x)}$ the induced diffeomorphism on $U_x=U\cap T_xM$.

For every $n\in\ZZ$ we let $U^n$ be the domain of definition of $\tilde f^n$, we let $V_n=\tilde f^n(U_n)$ and  we let $\tilde f^n_x \colon U_{x,n}\to V_{f^n(x),n}$ be the induced map on the fibers.  

\paragraph{Linearizing coordinates of $\tilde f$ along an orbit.}

The map $\tilde f^n$ defines a germ of the dynamics in a neighborhood of the zero section. The differential $Df$ is also a homeomorphism of $TM$ that induces a continuous family of linear diffeomorphisms on the fibers. 

For $n\geq 0$, we let $\psi_n = Df^n\circ \tilde f^{-n}\colon V_n\to TM$ and $\psi_{x,n}\colon V_{x,n}\to T_xM$ be
the induced maps. We have the following properties:
\begin{itemize}
\item For every integer $n$,  the family $\{\psi_{x,n}\}_{x\in M}$ is a compact family of $C^1$ diffeomorphisms depending continuously (in the $C^1$-topology) on $x$. Indeed, $\{\tilde{f}^{-n}_x\}_{x\in M} $ and $\{Df^n_{x}\}_{x\in M} $ are uniformly continuous families of diffeomorphisms. 
\item The derivative of the diffeomorphism $\psi_{x,n}$ at the zero vector $0_x$ is the identity map:
$$D_{0_x}(\psi_{x,n})=Id_{T_xM}.$$

\item For every $x\in M$ the family of diffeomorphisms $\{\psi_{f^n(x),n}\}_{n\in \NN}$ are linearizing coordinates along the orbit of $x$; more precisely, we have the following commutative diagram:
\scriptsize{
$$
\begin{array}{ccccccccccc}
\cdots& \stackrel{\tilde f}{\longrightarrow}&T_{f^{-1}(x)}M& \stackrel{\tilde f}{\longrightarrow}&T_xM& \stackrel{\tilde f}{\longrightarrow}& \dots&\stackrel{\tilde f}{\longrightarrow}&T_{f^n(x)}M\\
      &                                      &\quad\downarrow \id& 
      &\quad\quad \downarrow \psi_{f(x),1}&                      & && \quad\quad\downarrow \psi_{f^{n}(x),n}\\     
\cdots& \stackrel{D f}{\longrightarrow}&T_{f^{-1}(x)}M& \stackrel{D f}{\longrightarrow}&T_xM& \stackrel{Df}{\longrightarrow}&\dots&\stackrel{Df}{\longrightarrow}&T_{f^n(x)}M
\end{array}
.$$
}
\end{itemize}

\paragraph{Perturbations of the linearized dynamics.}

Given a point $x\in M$ we consider the derivative of $f$ along the orbit $\orb_f(x)$ of $x$ as being a diffeomorphism of the (non-compact, non-connected) manifold  $TM|_{\orb_f(x)}$ endowed with the Euclidean metric on the fibers. Now we will consider perturbations of $Df$ in $\diff^1(TM|_{\orb_f(x)})$ in the corresponding $C^1$-topology.

\begin{lemm}\label{l.linearize} Let $f$ be a diffeomorphism of a compact manifold. For every pair $0<\tilde \varepsilon<\varepsilon$ and every $n>0$ there exists $\rho>0$ with the following property. 

Consider an open set $U$ whose iterates $f^i(U)$, $i\in\{0,\dots,n\}$, are pairwise disjoint. Fix a point $x\in M$  and assume that $U$ is contained in $B(x,\rho)$;  let $\tilde U=\exp_x^{-1}(U)$. Then we have the following.

\begin{enumerate}
\item For every $i\in\{0,\dots, n\}$, the map $\Psi_{f^i(x),i}=\psi_{f^i(x),i}\circ\exp_{f^i(x)}^{-1}$ induces a diffeomorphism from $f^i(U)$ onto $Df^i(\tilde U)$ such that $$\max\{\|D\Psi_{f^i(x),i}\|, \|D\Psi_{f^i(x),i}^{-1}\|,
|\det D\Psi_{f^i(x),i}|, |\det D\Psi_{f^i(x),i}^{-1}|\}<2.$$
\item For every perturbation $\tilde g$ of $Df$ with support in $\bigcup_{i=0}^{n-1} Df^i(\tilde U)$
such that $d_{C^1}(\tilde g,Df)<\tilde \varepsilon$, let $g$ be the map  defined by 
\begin{itemize}
\item $g(y)=f(y)$ if $y\notin \bigcup_{i=0}^{n-1} f^i(U)$
\item $g(y)= \Psi_{f^{i+1}(x),i+1}^{-1}\circ \tilde g \circ \Psi_{f^i(x),i}(y)$
if $y\in f^i(U)$, $i\in \{0,\dots, n-1\}$.
\end{itemize}
Then the map $g\colon M\to M$ is a diffeomorphism. It is a perturbation of $f$ with support in $\bigcup_{i=0}^{n-1} f^i(U)$ such that $d_{C^1}(g,f)<\varepsilon$.
\item If furthermore $\tilde g$ is a tidy perturbation of $Df$, then $g$ is a tidy perturbation of $f$. 
\end{enumerate}
\end{lemm}
\begin{demo} The unique nontrivial statement in this lemma is the fact that
$g$ is an $\eps$ perturbation of $f$ if $\tilde g$ is an $\tilde\eps$ perturbation of $Df$.
The proof of this is a simple calculation using the facts that:
\begin{itemize}
\item  the derivative $D_{0_x}(\exp_x)$ is an isometry of $T_xM$, 
\item the derivative $D_{0_{x}}\psi_{x,i}$ is the identity for all $x$ and $i$,
\item the family $\{\psi_{x,i}, x\in M, 0\leq i\leq n\}$ induces a compact family of $C^1$-diffeomorphisms on the balls  $B(0_x,\frac R{C^n})$ having $0_x$ as a fixed point.
\end{itemize}
\end{demo}

%% file: LD-reduction0402.tex
\section{(LD) property: reduction to a perturbation result in towers}\label{s.largederiv}
In this section, we show how to reduce the proof of Theorem B to
the following perturbative result. Its proof is deferred to Section~\ref{s.perturbation}.

\begin{theoB'}[Large derivative by perturbation]
For any $d\geq 1$ and any $C,K,\ve>0$, there
exists an integer $n_0=n_0(C,K,\ve)\geq 1$ with the
following property. Consider any diffeomorphism $f$ of a
$d$-dimensional manifold $M$ with $\|Df\|,\|Df^{-1}\|<C$ and any
$n\geq n_0$. Let $N=2^{d+1}n$. Then there exists
$\rho_0=\rho_0(n,f,K,\ve)>0$ such that,
\begin{itemize}
\item for any open set $U$ with diameter $<\rho_0$ and where the iterates
$\overline{U}$, $f(\overline{U})$,\dots, $f^{N-1}(\overline{U})$ are pairwise disjoint,
\item for any compact set $\Delta\subset U$,
\end{itemize}
there exists $g\in \diff^1(M)$ such that
\begin{itemize}
\item the $C^1$ distance from $f$ to $g$ is less than $\ve$;
\item $g$ is a tidy perturbation of $f$ with support in $U\cup f(U)\cup\cdots\cup f^{N-1}(U)$;
\item for any $x\in \Delta$ there exists $j\in \{0,\dots, N-n\}$ such that
$$\max\{\|Dg^n(g^j(x))\|,\|Dg^{-n}(g^{j+n}(x))\|\}>K.$$
\end{itemize}
\end{theoB'}

\subsection{Our shopping list for the proof of Theorem B}\label{ss.shopping}
Let $M$ be a compact manifold with dimension $d$.
Consider a diffeomorphism $f$ whose periodic orbits are hyperbolic,
a constant $\ve>0$, a sequence $(K_i)$ that tends to $+\infty$ (for instance $K_i=i$)
and a constant $C>0$ that bounds the norms of $Dg$ and $Dg^{-1}$, for any diffeomorphism
$g$ that is $\ve$-close to $f$ in the $C^1$-distance.

Our perturbation $g$ of $f$ will emerge as the limit of a sequence of 
perturbations $g_0=f, g_1, g_2, \ldots$.
Each perturbation $g_i$ will satisfy a large derivative property for an interval of times $[n_{i}, n_{i+1}-1]$. 
The objects involved at step $i$ in the construction are:
\begin{itemize}
\item $g_{i}$, the perturbation, and $\ve_i$, the $C^1$ distance from $g_i$ to $g_{i-1}$.
\item $[n_i,n_{i+1}-1]$, the interval of time for which prescribed large derivative for $g_i$ occur, and
$K_i$, the magnitude of large derivative created for these iterates.
\item $U_{i}$, an open set such that the iterates $g_{i-1}^k(\overline{U_i})$, $k\in \{0,\dots,m_i\}$
are pairwise disjoint, where $m_i=2^{d+1}(n_i+n_1+1+\dots+n_{i+1}-1)$.
The perturbation $g_i$ of $g_{i-1}$ is supported in the union of these iterates
and is tidy with respect to their union. $\Phi_i$ is the $C^1$ conjugacy between $g_i$ and $f$ such that
$\Phi_i\circ \Phi_{i-1}^{-1}$ is supported in $\cup_{k=1}^{m_i-1}g_{i-1}^k(U_i)$.
The integer $M_i=\prod_{k=1}^iC^{m_k}$ bounds the $C^1$ norm of the conjugacy $\Phi_i$.
\item $\rho_{i}>0$, an upper bound on the diameter of the first $m_i$ iterates of $U_{i}$ under $g_{i-1}$.
\item $D_{i}$, a compact set contained in $U_i$ that meets every aperiodic $g_{i-1}$-orbit
(and hence also every aperiodic orbit of the tidy perturbation $g_i$).
\item $\De_{i}$, a compact set with $D_i\subset \hbox{int} \De_i \subset \De_i \subset U_i$.
Large derivative will occur for $g_i$ in the  first $m_i$ iterates of $\De_{i}$.
\item $Z_i$, a compact set, $x_i$, a point, and $N_i$, an iterate,
such that property (UD) holds for $f$ between $x_i$ and points of $Z_i$ in time less than $N_i$.
\end{itemize}

Some objects in this construction are fixed prior to the iterative construction of the 
perturbations, and others are chosen concurrently with the perturbations.  Thus
we have two types of choices, a priori and inductive, and they are chosen in roughly the following
order:
\begin{itemize}
\item {\em a priori choices:} $\ve_i, n_i, Z_i, x_i, N_i$;
\item {\em inductive choices:} $\rho_{i}$, $U_{i}$, $D_{i}$, $\De_{i}$, $g_{i}$. 
\end{itemize}

\subsection{A priori choices}

\subsubsection{Choice of $\ve_i$, $n_i$, $m_i$, $M_i$}\label{s.ven}

We first describe how to select inductively the positive numbers $\varepsilon_i$ and the integers $n_i$.
We initially choose $\ve_1 < \ve/2$ and $n_1 = n_0(C, 2K_1, \ve_1)$, according to Theorem B'.  Now given $(\ve_i, n_i)$, we let
$$\ve_{i+1} = \min\{\ve_i/2, \frac{K_{i-1}}{4 n_i C^{n_i -1}}\}
\quad \text{ and } \quad n_{i+1} = n_0(C, 2K_{i+1}, \ve_{i+1}).$$
With these definitions we have for any $\ell\geq 1$,
$$\sum_{k\geq \ell} \ve_k < 2 \ve_\ell<\varepsilon.$$
From this it follows that the sequence of perturbations $(g_i)_{i\in \NN}$ such that
$d_{C^1}(g_{i-1},g_i) < \ve_i$ and $g_0 = f$ will have a limit $g_i\to g$ as $i\to\infty$ that is
a diffeomorphism satisfying $d_{C^1}(f,g)<\ve$.
The bound for $\ve_{i}$ is justified by the following lemma which allows us to pass to $g$
the Large Derivative Property satisfied by $g_i$.

\begin{lemm}\label{l.chooseve}
Suppose that $\ve_{i+1} < \frac{K_{i-1}}{4 n_i C^{n_i -1}}$.
Then, for any $n\leq n_i$ and any sequences of matrices $A_1, \ldots, A_{n}$ and $B_1,\ldots, B_{n}$ satisfying:
\begin{itemize}
\item  $\|A_k\|, \|A_k^{-1}\|,\|B_k\|, \|B_k^{-1}\|<C $ for $k\in\{1,\ldots, n\}$,
\item $\|A_k -B_k \|, \|A_k^{-1} -B_k^{-1} \| < 4\ve_{i+1}$ for $k\in\{1,\ldots, n\}$,
\item $\max\{\|A_{n}\cdots A_1\|, \|(A_{n}\cdots A_1)^{-1}\| \} > 2K_{i-1},$
\end{itemize}
we have:
$$\max\{\|B_{n}\cdots B_1\|, \|(B_{n}\cdots B_1)^{-1}\| \} > K_{i-1}.$$
\end{lemm}

\begin{demo}
Without loss of generality we may assume that $\|A_n\dots A_1\|>2K_{i-1}$.
We have:
\begin{eqnarray*} 
 \|  A_{n}\cdots A_1\| - \|B_{n}\cdots B_1\|
 &\leq & \sum_{j=1}^{n} \|A_{n}\cdots A_{j+1}\cdot(A_j-B_j)\cdot B_{j-1}\cdots B_1\| \\
 &\leq & 4 n \ve_{i+1} C^{n-1}\leq 4n_i\varepsilon_{i+1}C^{n_i-1} \\
 &<& K_{i-1} \mbox{, by our choice of $\varepsilon_i$.}
\end{eqnarray*}
We obtain the desired conclusion: $\|B_{n_i}\cdots B_1\|> \|A_{n_i}\cdots A_1\| - K_{i-1} > K_{i-1}$.
\end{demo}

\subsubsection{Choice of $x_i$, $Z_i$ and $N_i$}\label{s.xZN}

Assume $f$ has the unbounded distortion
property  on the wandering set (UD$^{M\setminus\Omega}$) and on the stable manifolds (UD$^s$).
In this section we choose a countable  family of pairs $(x_i,Z_i)$ where $Z_i$ is a compact set disjoint from the orbit of $x_i$ which will be used for testing the (UD$^{M\setminus\Omega}$) and the (UD$^s$)-properties of the perturbations $g_i$ we will construct.

By definition of the (UD$^{M\setminus\Omega}$)-property there
exists a countable and dense subset $\cX^{M\setminus\Om}\subset M\setminus \Omega(f)$ such that,
for any $K>0$, any $x\in \cX^{M\setminus \Omega}$ and any $y\in M\setminus \Omega(f)$ 
not in the orbit of $x$,  there exists $n\geq 1$ such that:
$$|\log |\det Df^n(x)|-\log|\det Df^n(y)||>K.$$
For every $x\in \cX^{M\setminus\Om}$
we choose a countable set $\cZ_x$ of compact subsets of $M\setminus\Om_f$ disjoint from the orbit of $x$ and covering the complement in $M\setminus\Om$ of the orbit of $x$.
We define  $$\cZ^{M\setminus\Om}=\{(x,Z) | x\in \cX^{M\setminus\Om}, Z\in \cZ_x\}.$$

By definition of  the (UD$^s$)-property, for any hyperbolic
periodic orbit $\cO$, there exists a dense countable subset $\cX^\cO\subset
W^s(\cO)\setminus \cO$ such that, for any $K>0$, any $x\in \cX^\cO$ and any $y\in W^s(\cO)$
not in the orbit of $x$, there exists $n\geq 1$ such that:
$$|\log |\Det Df_{|W^s(\cO)}^n(x)|-\log|\Det Df_{|W^s(\cO)}^n(y)||>K.$$
For every $x\in \cX^{\cO}$
we choose a countable set $\cZ_x$ of compact subsets of $W^s(\cO)$ disjoint from the orbit of $x$ and covering the complement in $W^s(\cO)$ of the orbit of $x$.
We define  $$\cZ^{\cO}=\{(x,Z) | x\in \cX^{\cO}, Z\in \cZ_x\}.$$
 
We set
$$\cZ^s=\bigcup_{\cO\in Per(f)} \cZ^{\cO},\mbox{ and } \cZ=\cZ^{M\setminus \Om}\sqcup\cZ^s \mbox{ the disjoint union of }\cZ^{M\setminus \Om}\mbox{ and }\cZ^s .$$
Since, by hypothesis, the periodic orbits of $f$ are all hyperbolic, the set of periodic orbits is countable.
Consequently $\cZ$ is a countable set.

For each pair $(x,Z)\in \cZ$, we will need to verify that our perturbations preserve countably 
many conditions of the form 
$$ |\log |\Det Dg_i^{n}(x)|-\log|\Det Dg_i^{n}(y)||> L_i, \quad \hbox{ for every } y\in \Phi_i(Z),$$
where $L_i$ is a sequence of natural numbers tending to $\infty$.  To this
end, we fix an enumeration $\{(x_i, Z_i)\}_{i\in \NN}$ of the disjoint union $\bigsqcup_{n\in\NN} \cZ$;
each pair $(x,Z)\in \cZ$ appears infinitely many times as a pair $(x_{i_k}, Z_{i_k})$ in this choice of indexing.

We now fix a sequence of integers $N_i$ as follows:
\begin{itemize}
\item If $(x_i,Z_i)\in \cZ^{M\setminus\Om}$ then we fix $N_i$ such that for every $y\in Z_i$ there is $n\in\{1,\dots, N_i\}$ with
$$|\log |\Det Df^{n}(x_i)|-\log|\Det Df^{n}(y)||> K_i+4d\log M_i.$$

\item If $(x_i,Z_i)\in \cZ^{\cO}$ for some periodic orbit $\cO$ then we fix $N_i$ such that for every $y\in Z_i$ there exists $n\in\{1,\dots, N_i\}$ with
$$|\log |\Det Df_{|W^s(\cO)}^{n}(x_i)|-\log|\Det Df_{|W^s(\cO)}^{n}(y)||> K_i+4d\log M_i.$$
\end{itemize}

The $M_i$ appears in these expressions because $i$ will index the $i$th perturbation in our construction, and
the effects of the previous perturbations on the (UD) property will be taken into account.  The number $M_i^2$ bounds the effect of conjugacy by $\Phi_i$ on the derivative.

We set $Y_{i,0}=\bigcup_{k=0}^{N_i} f^k(Z_i\cup \{x_i\})$. Lemma~\ref{l.wosets} implies that the sets $Y_{i,0}$ have the wandering orbit property for $f$.

\subsection{Inductive hypotheses implying Theorem B}
We now describe conditions on the inductively-chosen objects that must satisfied and explain how they imply the
conclusion of Theorem B.
In a later subsection we explain how these choices can be made so that the inductive conditions are satisfied. 

\subsubsection{Conditions on $\rho_i$ so that $(\Phi_i)$ converges}

According to Lemma~\ref{l=conjugacy}, the diffeomorphisms $\Phi_i$
(conjugating $g_i$ to $f$) converge uniformly to a homeomorphism $\Phi$ (conjugating $g$ to $f$)
provided the following conditions hold.
\begin{equation}\label{e.rho}
\rho_i<2^{-1}\min\{\rho_{i-1},M_{i-1}^{-1}\}.
\end{equation}

\subsubsection{Conditions on $\rho_i$ preserving prolonged visits to towers}

Our perturbation at step $i$ will produce prescribed 
large derivative for $g$ for orbits visiting the compact set $\De_i$.  
In order to verify that $g$ has the (LD) property, we will need to know that every nonperiodic orbit for $g$
visits each set $\De_i$, which imposes the following condition on the sequence $(\rho_i)$:

\begin{equation}\label{e.rhovisits}
\rho_{i+1} < 2^{-1} \inf_{k\in\{0,\ldots, m_i\}} d(g_i^k(D_i), M\setminus g_i^k(\De_i)).
\end{equation}

The (LD) property on $g$ will be proved in the next subsection by using the following lemma.
\begin{lemm}\label{l.rhodelta} Assume furthermore that condition (\ref{e.rhovisits}) holds.
Then for every $i$ and any $y\in M\setminus \Per(g)$, there exists an integer $t$ such that
$$g^t(y)\in \De_i\quad\hbox{ and }\quad g_i^t(x)\in D_i,$$
where $x = \Phi_i\circ\Phi^{-1}(y)$.
For all $k\in\{0,\ldots, m_i\}$, we also have $g^{t+k}(y) \in g_i^k(\De_i)$.
\end{lemm}

\begin{demo} Fix $j\geq i$.
Let $y$ be a non-periodic point for $g_j$ and $x = \Phi_i\circ \Phi_j^{-1}(y)$.
For any $n$, we have $g_j^n(y) = \Phi_j\circ \Phi_i^{-1} g_i^n (x)$.
The point $x$ is not periodic for $g_i$, since $y$ is not periodic for $g_j$ and $y$
is the image of $x$ under the conjugacy between $g_i$ and $g_j$.

By our hypothesis on $D_i$, there exists $t$
such that  $g_i^t(x) \in D_i$.  The hypothesis~(\ref{e.rhovisits}) on $\rho_i$ implies that
$\Phi_j\circ\Phi_i^{-1}(g_i^k(D_i))\subset g_i^k(\De_i)$ for $k\in \{0,\dots,0,\dots,m_i\}$,
since, by Lemma~\ref{l=conjugacy},
$d_{unif}(\Phi_j\circ \Phi_i^{-1}, \id) <  \sum_{k=i+1}^j \rho_k < 2\rho_{i+1}$.
This implies that $g_j^{n+k}(y)\in g_i^k(\De_i)$. 

Now consider $y\in M$ that is not periodic for $g$. 
Let $x = \Phi_i\circ \Phi^{-1}(y)$.  For every $j\geq i$, we set $y_j = \Phi_j\circ \Phi_i^{-1}(x)$;
observe that $y_j\to y$ as $j\to\infty$. The previous argument implies there exists a
$t$ such that for every $j$ and every $k\in \{0,\dots,0,\dots,m_i\}$,
we have $g_j^{t+k}(y_j)\in g_i^k(\De_i)$; compactness of $\De_i$ implies
that $g^{t+k}(y)\in g_i^k(\De_i)$.
\end{demo}

\subsubsection{Conditions on $\rho_i$, $U_i$ and $g_i$ for the (LD) property} 
The next lemma imposes one more requirement on $\rho_i,U_i$ and on the derivative $Dg_i$ in order
to obtain property (LD).

\begin{lemma}[Final conditions for (LD)]\label{l.nail} Suppose further that:
\begin{itemize}
\item For every $x,y\in M$, 
\begin{equation}\label{e.rhoi}
d(x,y) < 2\rho_{i}\Rightarrow \|Dg_{i-1}(x) - Dg_{i-1}(y)\| < \ve_{i+1}.
\end{equation}
\item The support of $g_{i+1}$ is disjoint from the support of $g_i$:
\begin{equation}\label{e.Ui}
\bar U_{i-1},\dots,g_{i-1}^{m_{i-1}-1}(\bar U_{i-1}),
\bar U_i,\dots,g_{i-1}^{m_i-1}(\bar U_i)
\mbox{ are pairwise disjoint.}
\end{equation}
\item For any $n\in [n_i, n_{i+1}-1]$ and $x\in \Delta_i$, there is $j\in \{0,\dots, m_i-n\}$ such that
\begin{equation}\label{e.Dgi}
\max\{\|Dg_i^n(g_i^j(x))\|,\|Dg_i^{-n}(g_i^{j+n}(x))\|\}>2K_i.
\end{equation}
\end{itemize}
Then $g$ has the large derivative (LD) property.
\end{lemma}

\begin{demo} Let $g$ be the limit of the $g_i$.  Let $K$ be a large positive number and
choose $i_0$ such that $K_{i_0} > K$.
We claim that for every $y\in M\setminus \Per(g)$, and for every $n\geq n_{i_0}$, there exists an 
integer $j>0$ such that
$$\max\{\|Dg (g^j(y))\|,\|Dg^{-n}(g^{j+n}(y))\|\}> K.$$

Fix such an $n\geq n_{i_0}$, and choose $i\geq i_0$
such that $n\in [n_i,n_{i+1}-1]$.  Let $y\in M\setminus \Per(g)$.
Let $x = \Phi_i\circ\Phi^{-1}(y)$.  Note that $x \in M\setminus \Per(g_i)$,
so by our hypothesis on $D_i$, there exists an integer $t$ such that
$g_i^t(x)\in D_i$.  Lemma~\ref{l.rhodelta} implies that $g^t(y)\in \De_i$.  
Since the (LD) property is a property of 
orbits, it suffices to assume that $y\in \De_i$ and $x=\Phi_i\circ\Phi^{-1}(y)\in D_i$.

Now $x\in D_i \subset \De_i$ implies there exists $j\in \{0, \ldots, m_i - n\}$ such that
$$\max\{\|Dg_i^n(g_i^j(x))\|,\|Dg_i^{-n}(g_i^{j+n}(x))\|\}>2K_i.$$

For any $k \in \{0, \ldots, m_i-1\}$, we have 
\begin{equation*}
\begin{split}
\|Dg(&g^k (y))-Dg_i(g_i^k(x))\|\\
&\leq \|Dg(g^k(y)) - Dg_{i+1}(g^k(y))\|+ \|Dg_{i+1}(g^k(y)) - Dg_i(g_i^k(x))\|\\
&\leq 2\ve_{i+2} + \|Dg_{i+1}(g^k(y)) - Dg_i(g^k(y))\| + \| Dg_i(g^k(y))  - Dg_i(g_i^k(x))\|,
\end{split}
\end{equation*}
where we have used the fact that the $C^1$ distance from $g$ to $g_{i+1}$ is bounded by $\sum_{k\geq i+2} \ve_k < 2\ve_{i+2}$.

We next bound the remaining terms in the inequality. Since $g^k(y) = \Phi\circ\Phi_i^{-1} (g^k_i(x))$ and
$d_{unif}(\Phi\circ\Phi_i^{-1}, \id) < 2\rho_{i+1}$, the hypothesis (\ref{e.Dgi}) implies that
$$ \| Dg_i(g^k(y))  - Dg_i(g_i^k(x))\| < \ve_{i+2}.$$

By Lemma~\ref{l.rhodelta}, we know that for $k \in \{0, \ldots, m_i-1\}$, the point $g^k(y)$  belongs
to the set $g_i^k(\Delta_i)$.  Since the support of $g_{i+1}$ is disjoint from $g_i^k(U_i)$, which
contains $g_i^k(\Delta_i)$, we obtain
that $g_{i+1}$ and $g_{i}$ agree in a neighborhood of $g^k(y)$, for $k \in \{0, \ldots, m_i-1\}$.
From this it follows that $ \|Dg_{i+1}(g^k(y)) - Dg_i(g^k(y))\| = 0$.

We conclude that 
\begin{eqnarray*}
\|Dg(g^k (y)) - Dg_i(g_i^k(x))\| &\leq & 3\ve_{i+2}.
\end{eqnarray*}
But now Lemma~\ref{l.chooseve} and the Chain Rule imply that  
$$\max\{\|Dg^n(g^j(y))\|,\|Dg^{-n}(g^{j+n}(y))\|\}>K_i.$$
\end{demo}

\subsubsection{The derivative at the periodic orbits is preserved}

Let $\cO$ be a periodic orbit of $f$. Note that its image by $\Phi_{i}$ should
be disjoint from $\overline{U_{i+1}},\dots,g_{i}(\overline {U_{i+1}})$
for $i$ large. In particular, the maps $g_{i+1}$ and $g_i$ coincide in a neighborhood
of the periodic orbit $\Phi_{i}(\cO)$. This proves that $\Phi(\cO)=\Phi_i(\cO)$
and that $Dg$ coincides with $Dg_i$ at points of $\Phi(\cO)$. Since
$g_i$ and $f$ are conjugate by the diffeomorphism $\Phi_i$,
we conclude that the derivatives of $f$ on $\cO$ and of $g$ on $\Phi(\cO)$
are conjugate.

\subsubsection{Conditions on $U_i$ for preserving the (UD) property}

We assume here that $f$ satisfies the unbounded distortion (UD) property on the stable manifolds and
on the wandering set, and we consider the sequences  $(x_i)$, $(Z_i)$, $(N_i)$ and
$(Y_{i,0})$ defined in Subsection~\ref{s.xZN}.
Let $$Y_i = \bigcup_{k=0}^{N_i} g_i^k(\Phi_i(Z_i\cup \{x_i\})) = \Phi_i(Y_{i,0}).$$
We introduce the following condition
\begin{equation}\label{e.UD}
\text{The sets }
\overline{U_i},\dots,g_{i-1}^{m_i}(\overline{U_i})
\text{ are disjoint from }
\cup_{k=1}^{i-1} Y_k.
\end{equation}

\begin{lemm}\label{l.Yi} If in addition hypothesis~(\ref{e.UD}) is satisfied,
then $g$ has the unbounded distortion properties (UD)$^{M\setminus\Omega}$ and (UD)$^s$.
\end{lemm}

\begin{demo} We prove that $g$ has property (UD)$^{M\setminus\Omega}$; the proof of (UD)$^s$ is similar.
Recall the dense set $\cX^{M\setminus\Omega}$ in $M\setminus\Om(f)$ used to define
$(x_i)$ and $(Z_i)$. Clearly $\Phi(\cX^{M\setminus\Omega})$ is dense in $M\setminus\Om(g)$.

Fix $x\in \Phi(\cX^{M\setminus\Omega})$ and $y\in M\setminus \Omega(g)$ that are not on the same orbit.
Let $K>0$ be some large constant.
We claim that there exists $i\in \NN$ such that:
\begin{itemize}
\item  $\Phi^{-1}(x)=x_i$,
\item $\Phi^{-1}(y)\in Z_i$ (and we set $y_i=\Phi^{-1}(y)$),
\item $K_i>K$.
\end{itemize}
Such an $i$ exists because $\Phi^{-1}(y)\in M\setminus(\Om(f)\cup \orb_f(x))$ and so, by definition of $\cZ_x$, there exists $Z\in \cZ_x$ containing $\Phi^{-1}(y)$. The pair $(x,Z)$ appears as $(x_i,Z_i)$ for infinitely many values of $i$. For $i$ sufficiently large, we have $K_i>K$, which proves the claim.

By definition of $N_i$, there exists $n\in\{1,\dots, N_i\}$ such that
$$|\log |\det Df^n(x_i)|-\log|\det Df^n(y_i)||>K_i+4d\log M_i.$$
Since $g_i=\Phi_i\circ f\circ \Phi_i^{-1}$,  and $\|D\Phi_i\|$ and $\|D\Phi_i^{-1}\|$ are both  bounded by $M_i$,
 we obtain that 
$$|\log |\det Dg_i^n(\Phi_i(x_i))|-\log|\det Dg_i^n(\Phi_i(y_i))||>K_i.$$

Our assumption on the support of the tidy perturbations implies that for every $j>i$, 
$g_j$ is a tidy perturbation of $g_{j-1}$ whose support is disjoint from the compact set $Y_i$. This implies that $g_j$ and $g_i$ coincide in a neighborhood of the points 
$\Phi_i(x_i),g_i(\Phi_i(x_i)),\dots,g^{N_i}(\Phi_i(x_i))$ and 
$\Phi_i(y_i),g_i(\Phi_i(y_i)),\dots,g^{N_i}(\Phi_i(y_i))$. In particular:
\begin{itemize}
\item $\Phi_j(x_i)=\Phi_i(x_i)$ and $\Phi_j(y_i)=\Phi_i(y_i)$. 
Since the points $\Phi_j(x_i)$ and $\Phi_j(y_i)$ converge to $x$ and $y$ when $j\to\infty$,
it follows that $\Phi_i(x_i)=x$ and $\Phi_i(y_i)=y$.
\item $ Dg_i^n(\Phi_i(x_i))= Dg_j^n(x)$ and  $Dg_i^n(\Phi_i(y_i))=Dg_j^n(y)$.
\end{itemize}
It follows that, for $j\geq i$,
$$ |\log |\det Dg^n_j(x)|-\log|\det Dg^n_j(y)||\geq K_i,$$

Since $Dg_j^n$ tends to $Dg^n$ as $j\to\infty$, it follows that 
$$ |\log |\det Dg^n(x)|-\log|\det Dg^n(y)||\geq K_i>K,$$
which concludes the proof of the ($\mbox{UD}^{M\setminus\Om}$) property for $g$. 
The proof of the ($\mbox{UD}^s$) property is completely analogous. 

\end{demo}

\subsection{Satisfying the inductive hypotheses}

To finish the proof of Theorem B (assuming Theorem B'), we are left to
explain how to construct inductively $\rho_i,U_i,D_i,\Delta_i$ and $g_i$
satisfying the properties stated at Section~\ref{ss.shopping}
and properties~(\ref{e.rho}), (\ref{e.rhovisits}), (\ref{e.rhoi}), (\ref{e.Ui}), (\ref{e.Dgi})
and (\ref{e.UD}). For the construction we require the following extra property:
\begin{equation}\label{e.Uiextra}
\text{ The sets } g_{i-1}^j(\overline{U_i}) \text{ for }
j\in\{0,\dots,m_i+m_{i+1}-1\} \text{ are pairwise disjoint.}
\end{equation}
In the following we assume that all the objects have been constructed up to
$\rho_i, U_i, D_i, \De_i, g_i$, and we will construct $\rho_{i+1}, U_{i+1}, D_{i+1}, \De_{i+1}, g_{i+1}$.

\paragraph{The constant $\rho_{i+1}$.}
We choose $\rho_{i+1}$ satisfying:
\begin{enumerate}
\item $\rho_{i+1}$ is strictly less than the numbers $\rho_0(n,g_i,2K_{i+1},\ve_{i+1})$ given by Theorem~B'
for $n_{i+1}\leq n <n_{i+2}$;
\item $\rho_{i+1} < 2^{-1}\min\{\rho_i, M_i^{-1}\}$,
\item $\rho_{i+1} < 2^{-1} \inf_{k\in\{0,\ldots, m_i\}} d(g_i^k(D_i), M\setminus g_i^k(\De_i)).$
\item $\rho_{i+1}$ is less than the Lebesgue number associated to $Dg_i$ for $\ve_{i+2}$, so that, for every $x,y\in M$, if $d(x,y) < 2\rho_{i+1}$, then $\|Dg_{i}(x) - Dg_{i}(y)\| < \ve_{i+2}$.
\end{enumerate}
In particular conditions~(\ref{e.rho}), (\ref{e.rhovisits}) and~(\ref{e.rhoi}) are satisfied by $\rho_{i+1}$.

\paragraph{The sets $U_{i+1},D_{i+1},\Delta_{i+1}$.}
Observe that  $Y_k=\Phi_k(Y_{k,0})=\Phi_i(Y_{k,0})$ for $k\leq i$ have the wandering orbit property for $g_i$
because $Y_{k,0}$ have the wandering orbit property for $f$.
By induction property~(\ref{e.Uiextra}) is satisfied by $U_i$
and implies that the sets $g_{i}^j(\overline{U_i})$ for $j\in\{0,\dots,m_i+m_{i+1}-1\}$
are pairwise disjoint. By Lemma~\ref{l.wosets}, the set $\bigcup_{j=0}^{m_{i}-1} g_i^{j}(\overline U_{i})$ has the
$m_{i+1}$-wandering orbit property and so Lemma~\ref{l.wounion} implies that
$\bigcup_{k=0}^iY_k\cup \bigcup_{j=0}^{m_{i}-1} g_i^{j}(\overline U_{i})$ has the 
$m_{i+1}$-wandering orbit property.
Lemma~\ref{l.wotower} gives an open set $U_{i+1}$ and a compact set $D_{i+1}\subset U_{i+1}$ such that
\begin{enumerate}
\item the diameter of each connected component of $g_i^j(\overline{U_{i+1}})$, $j\in\{0,\dots,m_{i+1}\}$ is less than $\rho_{i+1}$;
\item the sets $g_{i}^j(\overline{U_{i+1}})$, $j\in\{0,\dots,m_{i+1}-1\}$,
are disjoint from the sets $Y_k$, $k\leq i$, and from the sets $g_{i}^{j}(\overline U_{i})$, $j\in\{0,\dots,m_{i+1}-1\}$;
\item the sets $g_{i}^j(\overline{U_{i+1}})$ for $j\in\{0,\dots,m_{i+1}+m_{i+2}-1\}$ are pairwise disjoint;
\item every nonperiodic point $x\in M$ has an iterate in $D_{i+1}$.
\end{enumerate}
In particular, conditions~(\ref{e.Ui}), (\ref{e.UD}) and~(\ref{e.Uiextra}) are satisfied by $U_{i+1}$.
We next fix some compact set $\De_{i+1}\subset U_{i+1}$ containing $D_{i+1}$ in its interior.

\paragraph{The perturbation $g_{i+1}$.}
For each $n\in \{n_{i+1},\dots,n_{i+2}-1\}$,
we will make a tidy perturbation $g_{n, i+1}$ producing the large derivative at the
$n$th iterate. Furthermore, these perturbations will have pairwise disjoint support.

To do this, we partition $\{0,1,\ldots, m_{i+1}-1\}$ into
intervals of the form $I_n=\{\alpha_n,\dots,\alpha_{n+1}-1\}$, $n_{i+1}\leq n <n_{i+2}$,
where $\alpha_n=2^{d+1}(n_{i+1}+\dots+n-1)$. This is possible since
$m_{i+1} = 2^{d+1}(n_{i+1} + (n_{i+1} + 1) + \cdots + n_{i+2} - 1)$.

The sets $g_i^{j}(\overline{U_{i+1}})$ for $j\in I_n$ are pairwise
disjoint and have a diameter less than $\rho_0(n,g_i,2K_{i+1},\ve_{i+1})$.
Hence, we can apply Theorem~B' to obtain a tidy perturbation $g_{n,i+1}$ of $g_i$
with support in $\bigcup_{j\in I_n} g_i^j(U_{i+1})$ 
such that $d_{C^1}(g_{n,i+1},g_i)<\varepsilon_{i+1}$ and
for any  $x\in g_i^{\alpha_n}(\Delta_{i+1})$, there exists $j\in \{\alpha_n,\dots, \alpha_{n+1}-n\}$ such that
$$\max\{\|Dg_{n,i+1}^n(g_{n,i+1}^j(x))\|,\|Dg_{n,i+1}^{-n}(g_{n,i+1}^{j+n}(x))\|\}>2K_{i+1}.$$
We denote by $\varphi_{n,i+1}$ the conjugating diffeomorphism associated to the tidy perturbation $g_{n,i+1}$, which is the identity outside of $\bigcup_{j\in I_n} g_i^j(U_{i+1})$. Notice that $\varphi_{n,i+1}$ is also the identity map on $g_i^{\alpha_n}(U_{i+1})$.

We now define the diffeomorphisms $g_{i+1}$ and $\varphi_{i+1}$
that coincide respectively with $g_i$ and $\id_M$ on $M\setminus \bigcup_{j=0}^{m_{i+1}-1} g_i^j(U_{i+1})$ and  with $g_{n,i+1}$ and $\varphi_{n,i+1}$ on $\bigcup_{j\in I_n} g_i^j(U_{i+1})$, for $n\in\{n_i,\dots, n_{i+1}-1\}$.
Using the fact that the tidy  perturbations $g_{n,i+1}$ have disjoint support, we obtain that $g_{i+1}$ and $\varphi_{i+1}$ have the following properties:
\begin{itemize}
\item $g_{i+1}$ is a tidy perturbation of $g_i$ with support in $\bigcup_{j=0}^{m_{i+1}-1} g_i^j(U_{i+1})$.
\item $g_{i+1}$ is conjugate to $g_i$ by $\varphi_{i+1}$.
\item $\varphi_{i+1}$ is the identity map on each $g_i^{\alpha_n}(U_{i+1})$; in particular, if $x\in\De_{i+1}$ then $g_{i+1}^{\alpha_n}(x)\in g_i^{\alpha_n}(\De_{i+1})$. 
\item Consequently, for any $n\in \{n_{i+1},\dots, n_{i+2}-1\}$ and any $x\in \Delta_{i+1}$, there exists $j\in \{\alpha_n,\alpha_{n+1}-n\}\subset \{0,\dots, m_{i+1}-n\}$ such that
$$\max\{\|Dg_{i+1}^n(g_{i+1}^j(x))\|,\|Dg_{i+1}^{-n}(g_{i+1}^{j+n}(x))\|\}>2K_{i+1}.$$
\end{itemize}

The proof of Theorem B assuming Theorem B' is now complete.

%% file: LD-perturbation0402.tex
\section{Large derivative by perturbation in towers}\label{s.perturbation}
The aim of this section is to prove Theorem~B', thereby
completing the proof of Theorem B. In the first three subsections we reduce the problem
to a linear algebra result, which is proved in the last section.

\subsection{Reduction to cocycles}
To any sequence $(A_i)$ in $GL(d,\RR)$ we associate its linear cocycle
as the map $f\colon \ZZ\times\RR^d\to \ZZ\times \RR^d$ defined by
$(i,v)\mapsto (i+1,A_i(v))$.
Theorem~B' is a consequence of the following corresponding result
for $C^1$ perturbations of linear cocycles.

\begin{proposition}\label{p.tidycocycle} For any $d\geq 1$ and any $C, K, 
\varepsilon>0$,
there exists $n_1=n_1(d,C,K,\varepsilon)\geq 1$ with the following
property.

Consider any sequence  $(A_i)$ in $GL(d,\RR)$ with $\|A_i\|,\|A_i^{-1}\|<C$
and the associated linear cocycle $f$. Consider any integer $n\geq n_1$ and
let $N=2^{d+1}n$. Then for any open set $U\subset [-1,1]^d$ and for any compact set $\Delta\subset U$, 
there exists a diffeomorphism $g$ of $\ZZ\times\RR^d$ such that:
\begin{enumerate}
\item the $C^1$-distance from $g$ to $f$ is bounded by $\varepsilon$;
\item $g$ is a tidy perturbation of $f$ supported on $\bigcup_{i=0}^{N-1} f^i(\{0\}\times U)$;
\item for any $x\in \{0\}\times \De$ there exists $j\in\{0,\dots, N-n\}$ such that
$$\max\{\|Dg^{n}(g^j(x))\|,\|Dg^{-n}(g^{j+n}(x))\|\}>K.$$
\end{enumerate}
\end{proposition}

\begin{demo}[Proof of Theorem~B' from Proposition~\ref{p.tidycocycle}]
Fix $\tilde\varepsilon <\varepsilon$, let $\tilde K = 4K$ and set
$n_0 = n_1(d,C, \tilde K,\tilde \varepsilon)$, according to
Proposition~\ref{p.tidycocycle}. Let $f$ be a diffeomorphism of the
$d$-manifold $M$ with $\|Df\|,\|Df^{-1}\|<C$ and fix $n\geq n_0$.
Let $N=2^{d+1}n$. We choose $\rho_0=\rho$ according
to Lemma~\ref{l.linearize} associated to $\tilde \varepsilon<\varepsilon$ and $N$.

Let us show that $\rho_0$ satisfies the conclusions of
Theorem~B'.  Let $U \subset M$ be an open set with
diameter less than $\rho_0$ whose iterates $U, f(U), \ldots,
f^{N-1}(U)$ are pairwise disjoint, and let $\De\subset U$ be
compact. Fix a point $x_0\in U$.

Lemma~\ref{l.linearize} asserts that there are diffeomorphisms 
$\Psi_{f^i(x_0),i}\colon f^i(U)\to T_{f^i(x_0)}M$, $i\in\{0,\ldots,N-1\}$ which conjugate $f$ to its
tangent map: if $y\in f^i(U)$, $i\in \{0,\dots, N-1\}$ then
$$f(y)= \Psi_{f^{i+1}(x_0),i+1}^{-1}\circ D_{f^i(x_0)}f \circ \Psi_{f^i(x_0),i}(y).$$
Moreover,
\begin{itemize}
\item the quantities $\|D\Psi_{f^i(x_0),i}\|$ and $\|D\Psi^{-1}_{f^i(x_0),i}\|$ are bounded by $2$;

\item any tidy $\tilde\varepsilon$-perturbation $\tilde g$ of $Df$ with support in $\bigcup_0^{N-1} Df^i(\tilde U)$, where $\tilde U=\Psi_{x_0,0}(U)$, induces a tidy $\varepsilon$-perturbation $g$ of $f$ supported on $\bigcup_0^{N-1} f^i(U)$, through a conjugacy by the diffeomorphisms $\Psi_{f^i(x_0),i}$.
\end{itemize} 

We apply Proposition~\ref{p.tidycocycle} to the cocycle induced by $Df$ on the tangent bundle $TM|_{\orb(x_0)}$ over the orbit of $x_0$, to the images $\tilde U,\tilde \Delta$ of $U,\Delta$ by $\Psi$, and to the integer $n$ (which is larger than $n_1$). We obtain a (nonlinear) cocycle $\tilde g:\ZZ \times \RR^d \to \ZZ \times\RR^d$
whose $C^1$-distance to $Df$ is smaller than $\varepsilon$,
which is a tidy perturbation of $Df$ supported in $\bigcup_{i=0}^{N-1} Df^i(\tilde U)$
and such that for any $x\in \tilde \Delta$ there exists $j\in \{0,\dots,N-n\}$ satisfying
\begin{equation}\label{e.ex}
\max\{\|D\tilde g^{n}(\tilde g^j(x))\|,
\|D\tilde g^{-n}(\tilde g^{j+n}(x))\|\}>\tilde K.
\end{equation}

By Lemma~\ref{l.linearize}, the cocycle $\tilde g$ defines a tidy perturbation $g$ of $f$ supported on $\bigcup_0^{N-1} f^i(U)$ and such that $d_{C^1}(f,g)<\varepsilon$.
Consider $x\in \Delta$, $\tilde x=\Psi_{x_0,0}(x)\in\tilde \De$ and the
integer $j\in \{0,\dots,N-n\}$ such that~(\ref{e.ex}) holds.
Note that on $g^j(U)$ we have $\tilde g^n=\Psi_{f^{j+n}(x_0),j+n}\circ\tilde g^n\circ \Psi_{f^j(x_0),j}^{-1}$.
Since the derivatives $\|\Psi_{f^{j+n}(x_0),j+n}\|$ and $\|\Psi_{f^j(x_0),j}^{-1}\|$ are bounded by $2$,
it follows that
$$\|D g^{n}(g^j(x))\|\geq \frac 14\|D\tilde g^{n}(\tilde g^j(\tilde x))\|,$$
$$\|D g^{-n}(g^{j+n}(x))\|\geq \frac 14\|D\tilde g^{-n}(\tilde g^{j+n}(\tilde x))\|.$$
By property~(\ref{e.ex}) above, this gives
$$\max\{\|Dg^{n}(g^j(x))\|,\|Dg^{-n}(g^{j+n}(x))\|\}>\frac14 \tilde K =K.$$

\end{demo}

\subsection{Reduction to a perturbation result in a cube}\label{ss.tidycocycle}

We now reduce Proposition~\ref{p.tidycocycle} to the case
$U$ is an iterate of the interior of the standard cube $Q=[-1,1]^d$
and $\Delta$ is an iterate of a smaller closed cube $\delta Q$.

\begin{proposition}\label{p.tidycube} For any $d\geq 1$, $C,K,\varepsilon>0$ and
$\delta\in (0,1)$, there exists an integer
$n_2=n_2(d,C,K,\varepsilon,\delta)\geq 1$ with the following property.

Consider any sequence $(A_i)$ in $GL(d,\RR)$ with $\|A_i\|,\|A_i^{-1}\|<C$
and the associated linear cocycle $f$. Then, for any $n\geq
n_2$ there exists a diffeomorphism $g$ of $\ZZ\times\RR^d$ such that:
\begin{enumerate}
\item the $C^1$-distance from $g$ to $f$ is bounded by $\varepsilon$;
\item $g$ is a tidy perturbation of $f$ supported on $\bigcup_{i=0}^{n-1} f^i(\{0\}\times Q)$;
\item there exists $k\in\{0,\dots, n\}$ such that for any $x\in \{0\}\times \delta Q$: 
$$\max\{\|Dg^{n}(g^{-k}(x))\|,\|Dg^{-n}(g^{n-k}(x))\|\}>K.$$
\end{enumerate}
\end{proposition}

\begin{rema}\label{r.tidycube} In Proposition~\ref{p.tidycube}, the time interval of the perturbation is $\{0,\dots,n-1\}$, and its support is 
$\bigcup_{i=0}^{n-1} f^i(\{0\}\times Q)$. 
The time $-k\in\{-n,\ldots, 0\}$ where we see the large derivative property might actually lie outside 
of the time interval of the
perturbation, and so we cannot compose two such perturbations with disjoint support without potentially
destroying the large derivative property.  Hence we need to consider the {\em effective} support of the perturbation, which is the larger
set  $\bigcup_{i=-n}^{n-1} f^i(\{0\}\times Q)$ corresponding to the time interval $\{-n,\ldots, n-1\}$. Note that the unique value of $i\in\{-n,\ldots, n-1\}$ 
for which we know that that the intersection of the effective support of the perturbation 
with $\{i\}\times \RR^d$ is a cube is $i=0$.  
\end{rema}
\noindent

\begin{demo}[Proof of Proposition~\ref{p.tidycocycle}, assuming Proposition~\ref{p.tidycube}]
We first define the integer $n_1$.
Let $C, K,\varepsilon>0$ be given.  Define $K_0=K^{2^{d+2}}$,
$C_1=C.K_0^2$, $\varepsilon_2=\frac\varepsilon{K_0^2}$, and
$K_1=K_0^3$. We fix $\delta=\frac9{10}$ and set:
$$n_1(d,C,K,\varepsilon)=n_2(d,C_1,K_1,\varepsilon_2,\delta).$$

Now consider $(A_i)$, $N$, $\Delta$ and $U$ as in the statement of Proposition~\ref{p.tidycocycle}. Note that we may assume that for any
$j\in\{0,\dots,N-n\}$, the product $A_{n-1+j}\cdots A_{j}$ and its inverse
have a norm bounded by $K$: otherwise, the conclusion
of Proposition~\ref{p.tidycocycle} is satisfied for the trivial perturbation $g=f$ and this value of $k$.

In order to apply Proposition~\ref{p.tidycube} we will tile the support of the perturbation. 
Recall our Remark~\ref{r.tidycube} that for a perturbation supported at times
$0,\dots,2n-1$, the tiles we use will have to be cubes ``at time $n$" and not ``at time $0$".
Thus we will tile the images $\{n\}\times \tilde U$ and $\{n\}\times\tilde \De$
of $\{0\}\times U$ and $\{0\}\times \De$ under $f^n$, where
$$\tilde U= A_{n-1}\dots A_0(U) \text{ and } \tilde \De= A_{n-1}\dots A_0(\De).$$
Let $a>0$ be such that $\frac {2a}{\delta}\sqrt d$ is less than the distance between $\tilde\De$ and the complement of $\tilde U$. We consider a regular tiling of $\RR^d$ by
cubes 
$$\tilde Q_{i_1,...,i_d}=[-a,a]^d+(2i_1a,\dots,2i_d a),\mbox{ where }i_j\in \ZZ. 
$$
We also consider the enlarged cubes
$$Q_{i_1,...,i_d}=\left[-\frac a{\delta},\frac a{\delta}\right]^d+(2i_1a,\dots,2i_d a).$$
The cubes $Q_{i_1,...,i_d}$ have diameter equal to $\frac{2a}{\delta}\sqrt d$.
By our choice of $a$, any cube $Q_{i_1,...,i_d}$ such
that $\tilde Q_{i_1,...,i_d}$ intersects the compact set $\De$ is
entirely contained in $U$. We denote by $\Ga$ the family of the  cubes
$Q_{i_1,...,i_d}$ such that $\tilde Q_{i_1,...,i_d}\cap \De\neq
\emptyset$.
Observe that two cubes $Q_{i_1,...,i_d}$ and $Q_{j_1,...,j_d}$ are
disjoint if and only if there exists $\ell\in\{1,\dots,,d\}$ with
$|i_\ell-j_\ell|\geq 2$.

Now consider the families $\{\Ga_{\bf
\mu}\}_{{\bf \mu}\in \{0,1\}^d}$, where $\Ga_{\mu_1,\dots,\mu_d}$ is
the collection of cubes $Q_{i_1,...,i_d}\in \Ga$ such that the index
$i_\ell$ is even if $\mu_\ell=0$ and odd if $\mu_\ell=1$. This gives
$2^d$ families of pairwise disjoint cubes contained in $\tilde U$ such that
the union of the corresponding smaller cubes $\tilde
Q_{i_1,...,i_d}$ covers $\tilde \De$.
To each $(\mu_1,\dots,\mu_d)\in\{0,1\}^d$ we associate the integer $\ell=\sum_{i=1}^d\mu_i 2^{i-1}$ 
(this formula induces a bijection from 
$\{0,1\}^d$ to $\{0,\dots,2^d-1\}$), and for simplicity we will use the notation $\Ga_\ell$ to
denote the family $\Ga_{\mu_1,\dots,\mu_d}$. We can thus write our partition 
$\Gamma = \Gamma_0 \cup \cdots \cup \Gamma_{2^d-1}$.

We will construct a tidy perturbation $g_Q$  of $f$ for each cube $Q\in\Ga$  
such that the supports of the perturbations $g_Q$ are pairwise disjoint.
The ultimate  perturbation $g$ 
will be obtained by combining all of these perturbations,
setting $g$ to equal $g_Q$ on the support of $g_Q$
and to equal $f$ outside the union of the supports.  
For each cube $Q\in\Ga_\ell$, for $\ell\in\{0,\dots,2^d-1\}$, 
the time interval of the perturbation will be 
$I_\ell=\{2n\ell,\dots,2n(\ell+1)-1\}\subset \{0,\dots,N-1\}$ 
and  the support of $g_Q$ will be 
$$W_Q = \bigcup_{i\in I_\ell} f^{i}(\{n\}\times Q).$$ 
The time intervals  $I_\ell$ form a partition of $\{0,\dots,N-1\}$ into $2^{d}$ intervals of length $2n$,
which implies that $W_Q\cap W_{Q'}=\emptyset$ for $Q\in\Ga_\ell$, $Q'\in \Ga_{\ell'}$ with $\ell\neq\ell'$. 
If $Q,Q'\in\Ga_\ell$, then the supports are disjoint as well, because $Q$ and $Q'$
are disjoint. 

We will now use Proposition~\ref{p.tidycube} to construct the perturbation $g_Q$ on $W_Q$. 
However there is an issue (as explained in Remark~\ref{r.tidycube}):
the time interval $I_\ell$ has length $2n$, as required, but 
the set $f^{2n(\ell-1)}(\{n\}\times Q)$ is not (in general) a cube (unless $\ell=0$),
and so does not satisfy the hypotheses of the proposition.
For this reason, for $Q\in \Ga_\ell$ and the corresponding smaller cube $\tilde Q\subset Q$
we first consider 
$$\widehat{W_Q}=f^{-2n\ell}(W_Q)=\bigcup_{j=0}^{2n-1} f^{j}(\{n\}\times Q).$$
 
Proposition~\ref{p.tidycube} ensures the existence of a tidy $\varepsilon_2$-perturbation 
$\widehat{g_Q}$ of $f$, supported in $\bigcup_{i=n}^{2n-1} f^{i}(\{n\}\times Q) \subset \widehat{W_Q}$ and of $k\in \{0,\dots,n\}$ such that
for every $x\in\{n\}\times \tilde Q$:
$$\max\{\|D\widehat{g_Q}^{n}(f^{-k}(x))\|,\|D\widehat{g_Q}^{-n}(\widehat{g_Q}^{n}\circ f^{n-k}(x))\|\}>K_1.$$
We now define $$g_Q=f^{2n\ell}\circ \widehat{g_Q}\circ f^{-2n\ell}.$$
Observe that:
\begin{itemize}
\item $g_Q$ is a tidy perturbation of $f$ supported in $\bigcup_{i=2n\ell+n}^{2n(\ell+1)-1} f^{i}(\{n\}\times Q) \subset W_Q$,
\item using the bound on the products $A_{j+n-1}\dots A_j$ and the choice of $\varepsilon_2$, we have:
$$d_{C^1}(g_Q,f)\leq (\max\{\|Df^n\|,\|Df^{-n}\|\})^{4\ell}d_{C^1}(\widehat{g_Q},f)\leq K^{2^{d+2}}\varepsilon_2\leq \varepsilon,$$
\item for every $x\in\{n\}\times \tilde Q$:
\begin{eqnarray}\label{e.tildederiv}
\max\{\|D{g_Q}^{n}(f^{2n\ell-k}(x))\|,\|D{g_Q}^{-n}({g_Q}^{n}\circ f^{2n\ell-k}(x))\|\}>\frac{K_1}{K^{2^{d+2}}}>K.
\end{eqnarray}
\end{itemize}

We finally define $g$ by $g=g_Q$ on $W_Q$ for $Q\in \Ga$ and $g=f$
outside the union of the $W_Q$, $Q\in \Gamma$. Since the support of the tidy $\varepsilon$-perturbations $g_Q$ are disjoint we get that $g$ is a tidy $\varepsilon$-perturbation of $f$ supported on
$\bigcup_{i=0}^{N-1}f^i(\{0\}\times U)$, proving the two first conclusions of the proposition. Furthermore, since these perturbations are tidy, for every point $x\in \{0\}\times \RR^d$ and every $\ell\in\ZZ$ we have $f^{2n\ell}(x)=g^{2n\ell}(x)$.  What is more, since the support of the perturbation
$g_Q$ is confined to the final $n$ iterates in the time interval $I_n$, we have that
$f^{2n\ell+n-k}(x)=g^{2n\ell +n-k}(x)$, for any $j\in\{-n,\ldots, 0\}$.

Consider now a point $x\in\{0\}\times\De$.  We show that there exists a $j$ such that
$$\max\{\|D{g}^{n}(g^{j}(0,w))\|,\|D{g}^{-n}(g^{n+j}(0,w))\|\}>K.$$
By our choice of cubes $Q$ covering $\tilde \De$, there exists a smaller cube $\tilde Q$ such that
$f^n(x)\in \{n\}\times\tilde{Q}$.  The corresponding cube $Q$ belongs to some $\Gamma_\ell$. 
The argument above shows that there exists $k\in \{0,\ldots,-n\}$
such that (\ref{e.tildederiv}) holds
for $f^n(x)$. Using the fact that $f^{2n\ell+n-k}(x)=g^{2n\ell +n-k}(x)$, for any $k\in\{0,\ldots, n\}$, we obtain that:
$$\max\{\|D{g}^{n}(g^{j}(x))\|,\|D{g}^{-n}(g^{n+j}(x))\|\}>K,$$
where $j=2n\ell+n-k$.
This gives the last conclusion of the proposition.
\end{demo}

\subsection{Reduction to linear perturbations}
We first define notation and review some linear algebra.
For any element $A\in GL(d,\RR)$ we define the {\em eccentricity of
$A$}, denoted by $e(A)$, to be the ratio
$$\sup\left\{\frac{\|A(u)\|}{\|A(v)\|} \mbox{ for } u,v\in\RR^d,\|u\|=\|v\|=1  
\right\}=
\|A\|\cdot \|A^{-1}\|.$$ Note that $e(A) \geq 1$, $e(A^{-1}) =
e(A)$, and for any $B\in GL(d,\RR)$, we have $\|A B A^{-1}\| \leq
e(A) \|B\|$.  We will also use the conorm notation $\cM(A) = \|A^{-1}\|^{-1}$.
We recall the basic fact from linear algebra that for any $A\in GL(d,\RR)$,
there exist orthogonal unit vectors $u$ and $v$ such that $\|Au\| = \|A\|$
and $\|Av\| = \cM(A)$.

Proposition~\ref{p.tidycube} is a consequence of the following proposition.

\begin{proposition}\label{p.linear}
For any $d\geq 1$ and $C,K,\eta>0$, there exists
$n_3=n_3(d,C,K,\eta)\geq 1$ with the following property.

For any sequence $(A_i)$ in $GL(d,\RR)$ satisfying $\|A_i\|,\|A_i^{-1}\|<C$
and $n\geq n_3$, there exist $H_0,\dots,H_{r-1}$, $0\leq r\leq\frac{n}2$, in $GL(d,\RR)$
and $k\in \{0,\dots,n\}$ such that:
\begin{itemize}
\item The product $P=A_{-k+n-1}\cdots A_0 H_{r-1}\cdots H_0 A_{-1}\cdots A_{-k}$
satisfies the estimate $\max\{\|P\|,\|P^{-1}\|\}>K$.
\item For $0\leq i<r$, the map $H_i$ sends the standard cube $Q$ into itself.
\item For $0\leq i\leq r$, there is the control
$$\left(e(A_{n-i-1}\cdots A_0)+e(A_{i-1}\cdots A_0)\right).\|H_i-\id\|<\eta.$$
\end{itemize}
\end{proposition}

We also state without proof a standard $C^1$ perturbation lemma.

\begin{lemma}\label{l.basicperturb}
For every $d\geq 1$, $\delta\in (0,1)$ there exists a neighborhood
$\cO$ of $\id$ in $GL(d,\RR)$ and for any $H\in \cO$
there exists a diffeomorphism $h(H)$ of $\RR^d$ such that
\begin{itemize}
\item $h(H)$ coincides with the identity map on $\RR^d\setminus Q$ and with $H$ on $\delta Q$.
\item The map $H\mapsto h(H)$ is $C^1$ in the $C^1$-topology.
\end{itemize}
\end{lemma}

\begin{demo}[Proof of Proposition~\ref{p.tidycube} from Proposition~\ref{p.linear}]
Lemma~\ref{l.basicperturb} associates to any $H$ in a neighborhood $\cO$ of $\id$ in $GL(d,\RR)$
a diffeomorphism $h(H)$ that satisfies
$$d_{C^1}(h(H),\id)\leq \theta \|H-\id\|,$$
for some uniform constant $\theta>0$.
We then choose $\eta<\varepsilon/(C\theta)$
such that the ball centered at $\id$ and with radius $\eta$ in $GL(d,\RR)$ is contained in $\cO$.
Set $n_2=n_3(d,C,K,\eta)$.

Let $(A_i)$ be any sequence in $GL(d,\RR)$ satisfying $\|A_i\|,\|A_i^{-1}\|<K$
and any integer $n\geq n_2$. Let us consider the sequence $(H_0,\dots H_{r-1})$ and
the integer $0\leq k\leq n$ given by Proposition~\ref{p.linear}.
By our assumptions each matrix $H_i$ belongs to $\cO$ and can be associated to a diffeomorphism
$h_i=h(H_i)$ by Lemma~\ref{l.basicperturb}.

We define the cocycle $g:\ZZ\times \RR^d\to \ZZ\times \RR^d$ as follows.
$$g(x) =
\begin{cases}
f^{i + 1} \circ h_i \circ f^{-i}(x),& \hbox{on } \{i\}\times \RR^d \text{ with } 0\leq i\leq r,\\
f^{n-i} \circ h^{-1}_i \circ f^{n-i-1}(x),& \hbox{on } \{n-i-1\}\times \RR^d \text{ with } 0\leq i\leq r,\\
f(x) & \hbox{otherwise.}\end{cases}$$
By construction, $g$ is a tidy perturbation of $f$ supported on $\bigcup_{i=0}^{n-1}f^i(\{0\}\times Q)$.

On the set $\{i\}\times \RR^d$, for $0\leq i\leq r$, the $C^1$ distance
between $f$ and $g$ is bounded by
$$d_{C^1}(h_i,\id)\leq\max\{\|A_{i}\|,\|A_{i}^{-1}\|\}\cdot  e(A_{i-1}\cdots
A_{0}) \cdot \theta \cdot\|H_i-\id\| < C\theta \eta< \varepsilon.$$
The same estimate holds on $\{n-i-1\}\times \RR^d$. Consequently the distance $d_{C^1}(f,g)$ is bounded by $\varepsilon$.

Since the map $h_i$ coincides with $H_i$ on $\delta Q$, it maps $\delta Q$ into itself.
For each $0 \leq i < r$, the map $g$ sends the set $f^i(\{0\}\times \delta Q)$
into the set $f^{i+1}(\{0\}\times \delta Q)$ and coincides with
$(A_{i}\dots A_0)H_0(A_{i}\dots A_0)^{-1}$.
It follows that on the set $g^{-k}(\{0\}\times \delta Q)$ the map $g^n$ is linear and coincides
with the product $P=A_{-k+n-1}\cdots A_0 H_{r-1}\cdots H_0 A_{-1}\cdots A_{-k}$.
In particular $\max\{\|Dg^n(g^{-k}(x)\|,\|Dg^{-n}(g^{n-k}(x))\|\}$ is larger than $K$ on
$\{0\}\times \delta Q$.
\end{demo}

\subsection{Huge versus bounded intermediary products}
We now come to the proof of Proposition~\ref{p.linear}.
Let $d\geq 1$, $K>1$ and $C,\eta>0$ be given. We choose $e_0>0$ and $s\in (0,1)$
such that
$$e_0>\eta^{-1}(K^2+1)^2 \text{ and } 2e_0\frac s {1+s}<\eta.$$
Next, we choose $n_3$ satisfying
$$(1+s)^{n_3/2-1}>K^2.$$
Let us consider $(A_i)$ and $n\geq n_3$ as in the statement of the proposition.
We may assume that the products $A_{-k+n-1}\cdots A_k$ and their inverses
have a norm bounded by $K$, since otherwise the conclusion of Proposition~\ref{p.linear}
holds already. Two cases are possible.

\paragraph{Huge intermediary products.}
We first assume that
$$\text{There exists } i_0\in\{1,\dots ,n\} \text{ such that }
e(A_{i_0-1}\cdots A_0)>e_0.$$

Decompose the linear map $A=
A_{i_0-1}A_{i_0-2}\dots A_{i_0-n}$ into a product $A=EF$, where
$E=A_{i_0-1}\dots A_{0}$ and $F=A_{-1}\dots A_{i_0-n}$.
By assumption we have $e(E) > e_0$, and $\cM(A) = \|A^{-1}\|^{-1} \geq \frac{1}{K}$. 

Let $u$ and $v$ be orthogonal unit vectors in $\RR^d$ satisfying $\|E u\| =
\|E\|$ and $\|E v\| = \cM(E)$. Then $\|E u\| 
= e(E) \|Ev\| > e_0 \|E v\|$.

Since $e_0\eta>(K^2 + 1)^2$, we can choose $t\in \left(\frac {K^2+1}{e_0},\frac\eta {K^2+1}\right)$.
Let $H\in GL(d,\RR)$ be a linear map satisfying
$H(v)=v+tu$ and $\|H-\id\|=t$.
Let $P= EHF$. We claim that the
norm of $P$ is greater than $K$.  To show this, we will use the
following elementary fact:
\begin{affi} Let $A, P\in GL(d, \RR)$.  Suppose there exists a
constant $\gamma>0$ and a nonzero vector $w\in \RR^d$ such that
$\|P w\| > \gamma \|A w\|$.  Then $\|P\| > \gamma \cM(A)$.
\end{affi}
To apply this lemma, we set $w = F^{-1}(v)$, and calculate:
$$
\|P w\| = \|EHF(F^{-1}(v))\| = \|EH(v)\| = \|E(v+t u)\|
\geq t \|E(u)\| - \|E(v)\|
$$
$$> (t e_0 - 1) \|E(v)\| = (t e_0 - 1) \| A w\|.$$
Applying the lemma with $\gamma = t e_0 - 1$, we
see that
$$\|P\| > (t e_0 - 1) \cM(A) \geq \frac{(te_0-1)}{K} > K,$$
by our lower bound on $t$.

By setting $r=0$ and $H_0=H$, we obtain the first conclusion of the proposition.
The second one is empty.
Since $\|A_{n-1}\dots A_0\|$ and $\|(A_{n-1}\dots A_0)^{-1}\|$
are bounded by $K$, we obtain that $e(A_{n-1}\dots A_0)<K^2$.
As a consequence, we have
$$(e(A_{n-1}\cdots A_0)+1)\|H_0-\id\|\leq (K^2+1)t<\eta,$$
by our upper bound on $t$. This gives the last conclusion.

\paragraph{Bounded intermediary products.}
Assume, on the other hand, that
$$\text{For any } i_0\in\{1,\dots ,n\} \text{ we have }
e(A_{i_0-1}\cdots A_0)\leq e_0.$$

Let $r=\lfloor\frac{n}{2}\rfloor$ and let $H\in GL(d,\RR)$ be the linear conformal dilation
$H=(1+s)^{-1}I$.
Since this is a linear contraction, $H$ maps the standard cube $Q$ into itself.
We define $H_i=H$ for any $i=0,\dots,r-1$.
Note that
$$(e(A_{n-i-1}\cdots A_0)+e(A_{i-1}\cdots A_0)).\|H_i-\id\|\leq 2e_0 \frac s {1+s}<\eta,$$
by our choice of $s$.

It is straightforward to check that for
$k=n-r$, the product $P=A_{-k+n-1}\cdots A_0 H_{r-1}\cdots H_0 A_{-1}\cdots A_{-k}$ satisfies
$$\|P^{-1}\|\geq (1+s)^r \|A_{-k+n-1}\cdots A_{-k}\|>(1+s)^{n_3/2-1}K^{-1}>K,$$
by our choice of $n_3$.
The conclusions of the proposition are thus satisfied.

%% file: UD-reduction0402.tex
\section{(UD) property: reduction to a perturbation result in a cube}
The aim of this section is to provide successive reductions for Theorem A.
At the end we are led to a perturbation result (Proposition~\ref{p.pertAreduced})
for cocycles that produces an arbitrarily large variation of the jacobian
along orbits in a cube. Many of the difficulties we meet in the proof of Theorem A come from the fact that we have not been able to create a large change in the jacobian of a linear cocycle, inside a cube, by a tidy perturbation.

\subsection{Reduction to a perturbation result in towers}
We show that Theorem B follows from a perturbation result
that produces arbitrarily large distortion between
a given orbit and the orbit of a wandering compact set.

\begin{prop}\label{p.pertA} Consider a diffeomorphism $f$,
a compact ball $\De$, an open set $U$, and a point $x$ of $M$ satisfying:
\begin{itemize}
\item $f(\overline U)\subset U$;
\item $\De \subset U\setminus f(\overline U)$;
\item the orbit of $x$ is disjoint from $\De$.
\end{itemize}
Then for any $K,\ve >0$ there exists a diffeomorphism $g$ with $d_{C^1}(f,g)<\ve$
having the following property:
for all $y\in \De$, there exists $n\geq 1$ such that
$$\left| \log\det Dg^n(x) - \log\det Dg^n(y)\right| > K.$$
Moreover, $f=g$ on a neighborhood of the chain-recurrent set $CR(f)$.
 \end{prop}

\begin{demo}[Proof of Theorem A from Proposition~\ref{p.pertA}]
Let $\cX\subset M$ be a countable dense set and let $\cK = \{\De_n\}$ be a countable collection of compact balls in $M$ satisfying:
\begin{itemize}
\item $\hbox{diam}\De_n\to 0$ as $n\to\infty$, and
\item $\bigcup_{n\geq n_1} \De_n = M$ for all $n_1\geq 1$.
\end{itemize}
For $\De\in \cK$, define the open subset of $\Diff^1(M)$:
$$\cO_\De = \{f\in \Diff^1(M)\,|\, \exists  \hbox{ open set } U, f(\overline U) \subset U, \De \subset U\setminus f(\overline U)\}.$$
For $x\in \cX$, define
$$\cU_{x,\De} = \{f\in \cO_\De \,|\, \orb_f(x)\cap \De = \emptyset\}.$$
Notice that  $f\in\cU_{x,\De}$ means that there is an open set $U$ such that $(\De,U,x)$ satisfies the hypotheses of Proposition~\ref{p.pertA}.

The set $\cU_{x,\De}$ is not open. The next lemma gives a simple criterion for $f$ to belong to its interior:
\begin{lemm}\label{l.intU} Consider  $f\in\cU_{x,\De}$ and  an open subset $U$ of $M$ with $f(\overline{U})\subset U$, $\De\subset U\setminus f(\overline{U})$. 
Assume that the orbit of $x$ meets $U\setminus f(\overline{U})$. Then $f$ belongs to the interior of $\cU_{x,\De}$: the orbit of $x$ under any diffeomorphism $g$ sufficiently $C^1$-close to $f$ is disjoint from $\De$.
\end{lemm}
\begin{demo}
Consider $i\in \ZZ$ such $f^i(x)\in U\setminus\overline{f(U)}$. Such a number $i$ is unique, 
because $U\setminus\overline{f(U)}$
is the fundamental domain of an attracting region $U$ and hence is disjoint
from all its iterates.
Moreover, there is a neighborhood $\cU$ of $f$ such that every $g\in \cU$ satisfies:
\begin{itemize}
\item $g(\overline{U})\subset U$,
\item $\De\subset U\setminus\overline{g(U)}$,
\item  $g^i(x)\in U\setminus \left(\overline{g(U)} \cup \De \right)$.
\end{itemize}
This shows that the open neighborhood $\cU$ of $f$ is contained in $\cU_{x,\De}$.
\end{demo}

In order to obtain a residual set, we must first produce a countable family of open and dense subsets of $\Diff^1(M)$. 
However the sets $\cU_{x,\De}$ are neither open nor closed.  The next lemma shows the way to bypass this difficulty:

\begin{lemm} The set $\interior\left(\cU_{x,\De}\right) \cup \interior\left(\cO_\De \setminus \cU_{x,\De}\right)$ is open
and dense in $\cO_\De$.
\end{lemm}
\begin{demo} The set $\interior\left(\cU_{x,\De}\right) \cup \interior\left(\cO_\De \setminus \cU_{x,\De}\right)$ is open
by definition; we just have to prove its density in $\cO_\De$.

Fix $f\in \cO_\De$.  Then there exists an open set $U\subset M$ and
an open neighborhood $\cU\subset \cO_\De$ of $f$ such that for $g\in \cU$
we have $g(\overline U) \subset U$ and $\De \subset U\setminus g(\overline U)$.
Now fix $x\in X$.  We will show that $\cU\cap (\interior\left(\cU_{x,\De}\right) \cup \interior\left(\cO_\De \setminus \cU_{x,\De}\right))$ is dense in $\cU$.

Let $\cV\subset \cU$ be the open subset of diffeomorphisms $g$  such that the orbit $\orb_g(x)$ meets $U\setminus g(\overline U)$. Any diffeomorphism $g\in \cU\setminus \cV$ belongs to $\cU_{x,\De}$: the orbit of $x$ is disjoint from $U\setminus g(\overline U)$, and so is disjoint \emph{a fortiori} from $\De$. Let $\cW_0=\cU\setminus\overline{\cV}$. By construction, $\cV\cup\cW_0$ is open and dense in $\cU$, and $\cW_0\subset \interior\left(\cU_{x,\De}\right)$.

Lemma~\ref{l.intU} asserts that $\cV_0=\cV\cap U_{x,\De}$ is open. Let $\cW_1=\cV\setminus\overline{\cV_0}$. Then $\cW_1$ is an open set contained in $\cO_\De \setminus \cU_{x,\De}$ and hence in $\interior\left(\cO_\De \setminus \cU_{x,\De}\right)$. Moreover, $\cV_0\cup\cW_1$ is open and dense in $\cV$.
We have thus shown that $\cV_0\cup\cW_1\cup\cW_0$ is an open and dense subset of $\cU$ contained in $\interior\left(\cU_{x,\De}\right) \cup \interior\left(\cO_\De \setminus \cU_{x,\De}\right)$, ending the proof.
\end{demo}

For $x\in \cX$, $\De\in\cK$ and  any  integer $K\in\NN$, we define:
 \begin{equation*}
\begin{split}
\cV_{x,\De,K} = \{f\in\interior(\cU_{x,\De}) \,|\,
\forall y\in \De, &\exists n\geq 1,\\
&\left|\log\jac Df^n(x) - \log\jac Df^n(y) \right| > K\}.
\end{split}
\end{equation*}
Note that $\cV_{x,\De,K}$ is open in $\interior(\cU_{x,\De})$. Furthermore  $\cV_{x,\De,K}$ is dense in $\interior(\cU_{x,\De})$,  by 
Proposition~\ref{p.pertA}. If follows that the set
$$
\cW_{x,\De,K} =  \cV_{x,\De,K} \cup \interior \left(\cO_\De\setminus \cU_{x,\De}\right) \cup \interior\left(\diff^1(M)\setminus\cO_\De\right) 
$$
is open and dense in $\diff^1(M)$.  Next, we set
$$\cG_0 = \bigcap \cW_{x,\De,K},$$
where the intersection is taken over $x\in \cX, \De\in \cK,$ and $K\in \NN$.  This set is residual in $\diff^1(M)$.
Finally, we take $\cG$ to be the intersection of $\cG_0$ with the residual set
of $f\in\diff^1(M)$ constructed in \cite{BC} where the chain recurrent and nonwandering sets coincide.

Consider $f\in \cG$. Fix $x\in \cX$ and $y\in M\setminus\Om(f)=M\setminus CR(f)$ such that $\orb(x)\cap\orb(y)=\emptyset$.  
Since $y$ is not chain recurrent,
Conley theory implies that there exists an open set $U\subset M$ such that $f(\overline U)\subset U$ and $y\in U\setminus \overline{f(U)}$
(see Section~\ref{s.preliminaries}).
Observe that $\orb(x)\cap U\setminus \overline{f(U)}$ contains at most one point;
it is distinct from $y$ by assumption.
Thus any compact set $\De\in \cK$ containing $y$ and with  sufficiently small diameter satisfies 
\begin{itemize}
\item $\De \subset U\setminus f(\overline U)$,
\item $\orb(x)\cap \De = \emptyset$.
\end{itemize}

We fix such a compact set $\De\in \cK$. This implies that $f\in  \cU_{x,\De}$.
Since $f\in\cG_0$, the definition of $\cG_0$ implies that for every $K\in \NN$, we have $f\in \cW_{x,\De,K}$; 
since  $f\in  U_{x,\De}$, we must have $f\in \cV_{x,\De,K}$. This means that, for every $K$, we have for some $n\geq 1$
$$ \left|\log\jac Df^n(x) - \log\jac Df^n(y) \right| > K.$$
Hence $f$ satisfies property (UD) on the nonwandering set.
\end{demo}

\subsection{Localization of the perturbation}
Here we reduce Proposition~\ref{p.pertA} to the case where
$\De$ has  small diameter.

We use the following notation. If $X\subset M$ is a compact set and $\delta>0$ then $U_\delta(X)$ denotes the \emph{$\delta$-neighborhood} of $X$:
$U_\delta(X) = \{y\in M\,|\, d(y,X) <\delta\}$.

\begin{prop}\label{p.pertAlocal} 
For any $d\geq 1$ and $C, K, \ve >0$,
there exists $n_0 = n_0(C,K,\ve)$ with the following property.

For any diffeomorphism $f$ of a $d$-dimensional manifold $M$
satisfying $\|Df\|,\|Df^{-1}\|<C$,
there exists $\rho_0 = \rho_0(d,C,K,\ve)$ such that for
any $\eta>0$, any compact set $\De\subset M$ and $x\in M$  satisfying:
\begin{itemize}
\item $\diam(\De) < \rho_0$,
\item $\De$ is disjoint from its first $n_0$ iterates $\{f^i(\De): 1\leq i\leq n_0\}$,
\item $\orb(x)\cap \De = \emptyset$,
\end{itemize}
there exists a diffeomorphism $g\in\diff^1(M)$ such that 
\begin{itemize}
\item  $d_{C^1}(f,g)<\ve$,
\item  $d_{C^0}(f,g) < \eta$,
\item for all $y\in \De$, there exists an integer $n\in \{1,\ldots,n_0\}$ such that:
$$\left| \log\det Dg^n(x) - \log\det Dg^n(y)\right| > K.$$
\end{itemize}
Moreover, $f=g$ on the complement of  $U_\eta(\bigcup_{i=0}^{n_0-1} f^i(\De))$. 
\end{prop}

\begin{demo}[Proof of Proposition~\ref{p.pertA} from Proposition~\ref{p.pertAlocal}]
Let $f$, $\De$, $U$, $x$ be as in the statement of Proposition~\ref{p.pertA}
Choose $n_0 = n_0(d,C,2K,\ve)$ and $\rho_0 = \rho_0(d,C,K,\ve)$
according to Proposition~\ref{p.pertAlocal}. We set $N=2^dn_0$.

Cover $\De$ by a finite collection $\cF$ of compact sets satisfying:
\begin{itemize}
\item $\De \subset \bigcup_{D\in \cF} \interior(D) \subset U\setminus f(\overline{U})$, so that for each $D\in \cF$ the iterates
$D, f(D), f^2(D), \ldots$ are pairwise disjoint;
\item $\orb(x)\cap  \bigcup_{D\in \cF} D = \emptyset$;
\item $\hbox{diam}(f^i(D)) <\rho_0$, for all $D\in \cF$ and $i\in\{0,\ldots,N\}$;
\item $\cF = \cF_0\cup \cdots \cup \cF_{2^d-1}$, where $\cF_i\cap \cF_j=\emptyset$
for $i\neq j$ and the elements of $\cF_j$ are pairwise disjoint for each $j$.
\end{itemize}
One can obtain $\cF$ by tiling by arbitrarily small cubes the compact ball $\De$.

Let $\lambda>0$ be the Lebesgue number of the cover $\cF$.  For any $\eta>0$
we define an increasing sequence $(a_\eta(n))$ by the inductive formula:
$$a_\eta(0)=0;\qquad a_\eta(n+1) = Ca_\eta(n) + \eta.$$
Note that for $n\geq 0$ fixed, we have $a_\eta(n)\to 0$ as $\eta\to 0$.

Let $\eta>0$ be small such that:
\begin{itemize}
\item for each $D\in \cF$ and $0\leq i<N-1$,
the $\eta$-neighborhood $U_\eta(f^i(D))$
is contained in $f^i(U)\setminus\overline{f^{i+1}(U)}$;
in particular the sets $U_\eta(D), \ldots, U_\eta(f^{N-1}(D))$ are pairwise disjoint;
\item the orbit $\orb(x)$ is disjoint from the union $\bigcup_{i\in \{0,\ldots, N\}} U_\eta(f^{i}(D))$;
\item  for each distinct $(j,D)$ and $(j',D')$
with $j,j'\in\{0,\ldots,2^d-1\}$ and $D\in \cF_j$, $D'\in \cF_{j'}$,
we have for all $k,k'\in\{0,\ldots, n_0-1\}$,
$$U_\eta(f^{n_0j + k}(D))\cap U_\eta(f^{n_0j' + k'}(D')) = \emptyset;$$
\item $a_\eta(N) < C^{-N}\lambda$.
\end{itemize}

For $j\in\{0,\ldots, 2^d-1\}$ and $D\in \cF_j$, the set
$f^{n_0j}(D)$ and the point $f^{n_0j}(x)$ satisfy the hypotheses of Proposition~\ref{p.pertAlocal}.
We obtain a perturbation of $f$ supported 
on the $\eta$-neighborhood of $\bigcup_{k\in\{0,\ldots, n_0-1\}}f^{n_0 j + k}(D)$; by our choice of $\eta$, any two
such perturbations for distinct choices of $(j,D)$ will be disjointly supported.
Hence, applying Proposition~\ref{p.pertAlocal} over all pairs $(j,D)$ with
$j\in\{0,\ldots, 2^d-1\}$ and $D\in \cF_j\}$,
we obtain a perturbation $g$ with the following properties:
\begin{enumerate}
\item $d_{C^1}(f,g)<\ve$;
\item\label{c2} $d_{C^0}(f,g) < \eta$;
\item $g=f$ on
$M\setminus U_\eta\left(\bigcup_{k=0}^{N-1} f^k(D)\right)$;
\item for each $j\in\{0,\ldots, 2^d-1\}$ and each $y\in \bigcup_{D\in \cF_j} D$,
 there exists $n\in\{1,\ldots,n_0\}$ such that:
\begin{eqnarray}\label{e.distfg}
\left| \log\det D{g}^n(f^{n_0j}x) - \log\det D{g}^n(f^{n_0j}y)\right| > 2K.
\end{eqnarray}
\end{enumerate}
\medskip

We now prove the large derivative formula.
We fix $y\in \De$.
\begin{affi} There exist $j\in\{0,\ldots, 2^d-1\}$ and
$D\in \cF_j$ such that $g^{n_0j}(y)\in f^{n_0j}(D)$.
\end{affi}
\begin{demo} Choose $D\in \cF$ such that the ball $B(y,\lambda)$ is contained in $D$
and fix $j\in \{1,\ldots, 2^d\}$ such that $D\in \cF_j$. This implies that
$B(f^k(y), C^{-k}\lambda) \subset f^k(D)$, for all $k\in \{0,\ldots, N-1\}$.
Note that
\begin{eqnarray*}
d(f^{k+1}(y), g^{k+1}(y))& \leq& d(f^{k+1}(y), f(g^{k}(y))) + d(f(g^{k}(y)), g^{k+1}(y))\\
&\leq& C d(f^{k}(y), g^k(y)) + \eta,
\end{eqnarray*}
which implies, by Property~\ref{c2} above and our choice of $\lambda$,
that for any $k\in \{0,\ldots, N-1\}$ we have:
$$d(f^k(y), g^k(y)) \leq a_\eta(k)<C^{-N}\lambda.$$
We conclude that $g^{n_0j}(y)\in f^{n_0j}(D)$.
\end{demo}

Since the orbit of $x$ is disjoint from the support of the perturbation,
we have that $f^{n_0j}(x) = g^{n_0j}(x)$, for all $j\in \{1,\ldots, 2^d-1\}$.
From these properties and from (\ref{e.distfg}), there exist
$j\in \{0,\ldots, 2^d-1\}$ and $n\in\{1,\ldots, n_0\}$ such that:
$$\left| \log\det D{g}^n(g^{n_0j}x) - \log\det D{g}^n(g^{n_0j}y)\right| > 2K.$$
The fact that $Dg^{n+n_0j} = Dg^{n}\circ Dg^{n_0j}$
implies that one of these two cases holds:
\begin{itemize}
\item either $\left| \log\det D{g}^{n_0j}(x) - \log\det D{g}^{n_0j}(y)\right| > K$,
\item or $\left| \log\det D{g}^{n+n_0j}(x) - \log\det D{g}^{n+n_0j}(y)\right| > K$.
\end{itemize}
In any case, there exists $n\in \{1,\dots,N\}$ such that the required estimate
$\left| \log\det D{g}^{n}(x) - \log\det D{g}^{n}(y)\right| > K$ holds.

By construction the support of the perturbation is contained in
a finite number of iterates of $U\setminus \overline{f(U)}$;
the iterates of $U\setminus \overline{f(U)}$ for $f$ and $g$ hence coincide.
This implies that $CR(f)=CR(g)$.
\end{demo}

\subsection{Reduction to cocycles}
Proposition~\ref{p.pertAlocal} is a consequence of the following
result about cocycles.
\begin{prop}\label{p.pertAcocycle}
For any $d\geq 1$ and any $C, K, \ve>0$,
there exists $n_1 = n_1(d,C,K,\ve)\geq 1$ with the following property.

Consider any sequence $(A_i)$ in $GL(d,\RR)$ with
$\|A,\|,\|A^ {-1}\|<C$ and the associated cocycle $f$.
Then, for any open set $U\subset \RR^d$,
for any compact set $\De\subset U$ and for any $\eta>0$, there exists a diffeomorphism $g$ of $\ZZ\times\RR^d$ such that:
\begin{itemize}
\item $d_{C^1}(f,g)<\ve$, 
\item $d_{C^0}(f,g) <\eta$,
\item $g = f$ on the complement of $\bigcup_{i=0}^{2n_1-1} f^i(\{0\}\times U)$,
\item for all $y\in \{0\}\times \De$, there exists $n\in\{1,\ldots, n_1\}$ such that 
$$\left| \log\det Df^{n}(y) - \log\det Dg^{n}(y)\right| > K.$$
\end{itemize}
\end{prop}

\begin{demo}[Proof of Proposition~\ref{p.pertAlocal} from Proposition~\ref{p.pertAcocycle}]
Fix $d,C,K,\ve>0$, choose $0<\tilde \ve<\ve$, $K_0> 2K+8\log2$ and set $n_0=2n_1(d,C,K_0,\tilde\varepsilon)$.

Let $f\colon M\to M$ be a diffeomorphism such that  $\|Df\|,\|Df^{-1}\|<C$
and let $\delta>0$ be the constant associated to $f$, $\tilde \ve$, $\varepsilon$,
and $n_0$ by Lemma~\ref{l.linearize}.
Fix $\rho_0\in\left(0,\frac{\delta}{2}C^{-n_1}\right)$.

Consider $\eta$, $\De$ and $x$ as in the statement of Proposition~\ref{p.pertAlocal}.
We fix an open neighborhood $U$ of $\De$ such that:
\begin{itemize}
\item $\diam(U) < \rho_0$; in particular $\diam(f^i(U))<\delta$ for all $i\in\{0,\dots, n_1\}$;
\item $U$ is disjoint from its first $n_0$ iterates;
\item $f^i(x)\notin U $ for $i\in\{0,\dots,n_0\}$.
\end{itemize}
Fix a point $z_0\in \De$ and let $\tilde f$ denote the linear cocycle
induced by the derivative $Df$ along the orbit of $z_0$.

By Lemma~\ref{l.linearize}, there are diffeomorphisms $\Psi_i\colon f^i(U)\to \tilde U_i\subset T_{f^i(z_0)}M$ for $i\in\{0,\dots, n_1\}$, which
conjugate $f$ to $\tilde f$: for every $z\in f^i(U)$
we have $\tilde f(\Psi_i(z))=\Psi_{i+1}(f(z))$.
Moreover,
\begin{itemize}
\item  $\|D\Psi_{i}\|$, $\|D\Psi_{i}^{-1}\|$, $|\det D\Psi_{i}|$ and $|\det D\Psi_{i}^{-1}|$ are bounded by $2$;
\item any $\tilde\varepsilon$-perturbation $\tilde g$ of $\tilde f$ with support in $\bigcup_{i=0}^{n_0-1} \tilde U_i$ induces a diffeomorphism $g$
which is a $\varepsilon$-perturbation of $f$ supported on
$\bigcup_{i=0}^{n_0-1} f^i(U)$ through a conjugacy by the diffeomorphisms
$\Psi_i$.
\end{itemize}
We denote by $\Psi\colon\bigcup_{i=0}^{n_0} f^i(U)\to \bigcup_{i=0}^{n_0} \tilde U_i$ the diffeomorphism that is equal to $\Psi_i$ on $f^i(U)$.

We now apply Proposition~\ref{p.pertAcocycle} to obtain a $\tilde \ve$-perturbation $\tilde g$ of 
$\tilde f$ supported in $\bigcup_{i=0}^{n_0-1} \tilde U_i$ such that for every $y\in\Psi_0(\De)$ we have:
\begin{itemize}
\item $d_{C^0}(\tilde f,\tilde g) <\frac\eta2$,
\item $\tilde g = \tilde f$ on the complement of $\bigcup_{i=0}^{n_0-1} \tilde U_i$,
\item  for all $y\in\Psi_0(\De)$ there exists $n\in\{1,\ldots, n_0\}$ such that 
$$\left| \log\det D\tilde f^{n}(y) - \log\det D\tilde g^{n}(y)\right| > K_0.$$
\end{itemize}
Let $g$ be the corresponding $\ve$-perturbation of $f$. 
Since $\Psi$ satisfies $\|D\Psi^{-1}\|\leq 2$,
we obtain that $$d_{C^0}(f,g)<2d_{C^0}(\tilde f,\tilde g)<\eta.$$
Furthermore, for each $z\in \De$, there exists $n\in\{1,\ldots, n_0\}$ such that 
\begin{eqnarray}\label{e.k0-log16}
\left| \log\det D f^{n}(z) - \log\det Dg^{n}(z)\right| > K_0- 4\log 2.
\end{eqnarray}

For $n\in\{0,\dots,n_0-1\}$,
the maps $f^n$ and $Df^n(z_0)$ are conjugate on $\De$
by the diffeomorphism $\Psi$.
Since $\|\det D\Psi\| \in [\frac12, 2]$, this implies
that, for every $z\in\De$, we have
$\left| \log\det Df^{n}(z_0) - \log\det Df^{n}(z)\right| \leq 2\log 2 $.
If there exists $n\in\{0,\dots, n_0-1\}$ such that
$\left| \log\det Df^{n}(z_0) - \log\det Df^{n}(x)\right| > K+2\log 2$,
then we do not perturb $f$:
in this case, every point $z\in\De$ satisfies $\left| \log\det Df^{n}(z) - \log\det Df^{n}(x)\right| > K$ as required.
Hence we may assume that for every $n\in\{0,\dots, n_0-1\}$ and every $z\in\De$ we have
\begin{eqnarray}\label{e.k+4log2}
\left| \log\det Df^{n}(z) - \log\det Df^{n}(x)\right| \leq K+2\log 2.
\end{eqnarray}

Notice that the  support of the perturbation $g$ is disjoint from the set $\{x,f(x),\dots,f^{n_0}(x)\}$, 
so that $Df^n(x)=Dg^n(x)$ for $n\in\{0,\dots, n_0\}$.
Inequality (\ref{e.k0-log16}) implies that for every $z\in \De$
there exists $n\in\{1,\dots n_0\}$ such that
\begin{eqnarray*}
\left| \log\frac{\det D g^{n}(x)}{\det Dg^{n}(z)}\right| &>& \left| \log\frac{\det D f^{n}(z)}{ \det Dg^{n}(z)}\right| - \left| \log\frac{\det D f^{n}(x)}{\det Df^{n}(z)}\right|\\
 & >&  K_0- 4\log 2 -K-4\log2 \; > K.
\end{eqnarray*}
This completes the proof of Proposition~\ref{p.pertAlocal}, assuming Proposition~\ref{p.pertAcocycle}.
\end{demo}

\subsection{Reduction to a perturbation result in a cube}
We now reduce Proposition~\ref{p.pertAcocycle} to the case where
$U$ is the interior of a cube $Q_\mu$, which has to be chosen to be very thin along
one coordinate, and $\De$ is a smaller cube $\theta Q_\mu$.

We thus consider the space $\RR^d$ as a product
$\RR^{d-1}\times \RR$ and denote by $(v,z)$ its coordinates.
For a constant $\mu>0$, let $E_\mu$
be the linear isomorphism of $\RR^d$ defined by $E_\mu(v,z) = (v,\mu z)$.
The image of the standard cube $Q=[-1,1]^d$ by $E_\mu$ will be denoted by
$Q_\mu=[-1,1]^{d-1}\times [-\mu,\mu]$. 

\begin{prop}\label{p.pertAreduced}  For any $d\geq 1$, $C, K, \ve >0$  and $\theta\in (0,1)$, 
there exists $n_2 = n_2(d,C,K,\ve, \theta)\geq 1$ and for any $\eta>0$ there exists 
$\mu_0=\mu_0(d,C,K,\ve,\theta,\eta) > 0$ with the following property.

Consider any sequence $(A_i)$ in $GL(d,\RR)$ with $\|A_i\|,\|A_i^{-1}\|<C$ and the associated linear cocycle $f$.
Consider any $\mu\in (0,\mu_0)$ and $j\in \{0,\dots, 2^d-1\}$.
Then there exists a diffeomorphism $g$ of $\ZZ\times\RR^d$ such that:
\begin{itemize} 
\item $d_{C^1}(f,g)<\ve$;
\item $g = f$ on the complement of $\bigcup_{i=2jn_2}^{2(j+1)n_2-1} f^i(\{0\}\times Q_\mu)$;
\item for all $x\in \{0\}\times \theta Q_\mu$, 
$$\left| \log\det Df^{n_2}(f^{2jn_2}(x)) - \log\det Dg^{n_2}(f^{2jn_2}(x))\right| > K;$$
\item on $\{0\}\times \RR^d$, we have:
 $$d_{unif}(\id, {E_\mu^{-1} f^{-2(j+1n_2)} g^{2(j+1)n_2} E_\mu}_{|_{\{0\}\times \RR^d}}) < \eta.$$ 
\end{itemize}
\end{prop}

\begin{demo}[Proof of Proposition~\ref{p.pertAcocycle} from Proposition~\ref{p.pertAreduced}]
The proof is partly similar to the proof of Proposition~\ref{p.tidycocycle}.
Let $C, K, \ve>0$ be given. We fix $\theta=\frac{9}{10}$ and set
$$n_1(d,C,K,\varepsilon) = 2^{d+1}n_2(C,2K,\ve,\theta)\geq 1.$$
Consider $(A_i)$, $\Delta$, $U$, $\eta$ as in the statement of Proposition~\ref{p.pertAcocycle}.
Up to rescaling by homothety, we may assume that:
\begin{eqnarray}\label{e.dilatation1}
C^{2^{d+1}n_1}2\theta^{-1}\sqrt{d}&< &d(\De,\RR^d\setminus U),\\
C^{2^{d+1}n_1}2\theta^{-1}\sqrt{d}&<&\eta.\label{e.dilatation2}
\end{eqnarray}
Indeed let us consider some $a\in (0,1)$ and the homothety $h_{\frac 1 a}$ of $\RR^d$ with ratio $\frac 1 a$.
If $g_0$ is a perturbation of $f$ satisfying the conclusions of Proposition~\ref{p.pertAcocycle} for  
$C,K,\varepsilon$, $h_{\frac 1 a}(\De)$, $h_{\frac 1 a}(U)$ and $\frac \eta a$,
then $g=h_a g_0 h_{\frac 1a}$ is a perturbation of $f$ satisfying
the conclusions of Proposition~\ref{p.pertAcocycle} for  $C,K,\varepsilon$, $\De$, $U$ and $\eta$.
Hence we may assume that the estimates~(\ref{e.dilatation1}) and~(\ref{e.dilatation2}) hold by choosing $a$ small enough.
\medskip

In order to apply Proposition~\ref{p.pertAreduced}, we tile the set $U$.  The tiling
we use has the same structure as the tiling used in Proposition~\ref{p.tidycocycle}, producing
$2^d$ disjoint families of tiles, but the geometry is different (see Figure 1).
\begin{figure}\label{f.grid}
\begin{center}
\includegraphics[scale=.75]{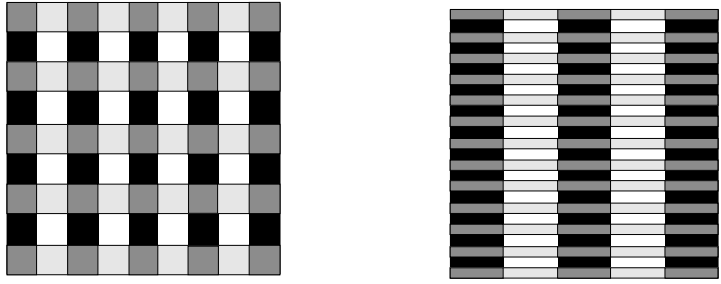}
\end{center}
\caption{Tiling of the plane used in (LD) perturbation theorem (left) and (UD) perturbation theorem (right).  The different shadings correspond to the partition into $2^d=4$ families. }
\end{figure}
We first consider the regular tilings $\cT$ and $\tilde \cT$ of $\RR^d$ by the cubes 
$$Q_{i_1,\dots,i_d}=[-\theta^{-1},\theta^{-1}]^d+(i_1,\dots,i_d),$$ 
$$\tilde Q_{i_1,\dots,i_d}=[-\theta^{-2},\theta^{-2}]^d+(i_1,\dots,i_d),$$
respectively, where $(i_1,\dots, i_d)\in\ZZ^d$.
Let $\lambda$ be the Lebesgue number of the cover $\cT$ of $\RR^d$.  
We choose $\eta_0>0$ smaller than $2^{-d}\lambda$.

Now fix $\mu\in (0,1)$ smaller than $\mu_0(C, 2K, \varepsilon,\theta,\eta_0)$
and denote by $\cT_\mu$ and $\tilde \cT_\mu$ the images of the tilings $\cT$ and $\tilde \cT$ by
$E_\mu$. These are the tilings of $\RR^d$ by the cubes
$$Q_{\mu,(i_1,\dots,i_d)}=
[-\theta^{-1},\theta^{-1}]^{d-1}\times[-\theta^{-1}\mu,\theta^{-1}\mu]+(i_1,\dots,i_{d-1},\mu i_d),$$ 
$$\tilde Q_{\mu,(i_1,\dots,i_d)}=[-\theta^{-2},\theta^{-2}]^{d-1}\times[-\theta^{-2}\mu,\theta^{-2}\mu]+(i_1,\dots,i_{d-1},\mu i_d),$$
respectively. By our assumption~(\ref{e.dilatation1}), any cube $Q_{\mu,(i_1,\dots,i_d)}$ such that
$\tilde Q_{\mu,(i_1,\dots,i_d)}$ intersects the compact set $\Delta$ is contained in $U$.
We denote by $\Ga_\mu$ the family of cubes $Q_{\mu,(i_1,\dots,i_d)}$
such that $\tilde Q_{\mu,(i_1,\dots,i_d)}$ intersects $\Delta$.

We next consider the partition  $\Ga_\mu=\bigcup_{\ell=0}^{2^d-1}\Ga_{\mu,\ell}$ such that 
for pair of any elements $Q,Q'$ in  $\Ga_{\mu,\ell}$, the enlarged cubes $\tilde Q,\tilde Q'$ are disjoint.
(As in Section~\ref{ss.tidycocycle} we set $Q_{\mu,(i_1,\dots,i_d)}\in\Ga_{\mu,\ell}$
if $\ell=\sum_{j=1}^{d} \alpha_j 2^{j-1}$ where $\alpha_j=0$ if $i_j$ is even and $\alpha_j=1$ if $i_j$ is odd). 
\medskip

For each cube $Q\in \Gamma_{\mu,\ell}$, we also introduce the time interval
$I_\ell=\{2n_2\ell,\dots,2n_2(\ell+1)-1\}$.
Proposition~\ref{p.pertAreduced} produces a perturbation $g_Q$
supported on the union
$$W_Q=\bigcup_{i\in I_\ell}f^i(\{0\}\times \tilde Q).$$
The construction of the famillies $\Gamma_{\mu,\ell}$ and of the intervals $I_\ell$ implies that
$W_Q$ and $W_{Q'}$ are disjoint if $Q,Q'\in \Gamma_\mu$ are distinct.
We finally define $g$ by $g_Q$ on $W_Q$, for $Q\in \Gamma_{\mu,\ell}$, and $g=f$ elsewhere.
In particular, we have $g=f$ on the complement of
$\bigcup_{i=0}^{n_1-1}f^i(\{0\}\times U)$.
The perturbation $g$ is supported in a disjoint union of images
$f^i(\{0\}\times \tilde Q_{\mu,(i_1,\dots,i_d)})$ for $i\in\{0,\dots, 2^dn_1-1\}$
and by~(\ref{e.dilatation2}) satisfies $d_{C^0}(f,g)<\eta$.
Each perturbation $g_Q$ satisfies $d_{C^1}(g_Q,f)<\varepsilon$, and hence
$d_{C^1}(g,f)<\varepsilon$ holds.
It remains to prove the last part of the conclusion of the proposition.

\begin{affi} For every $y\in\{0\}\times\De$, there exist $\ell\in\{0,\dots,2^d-1\}$ and $Q\in \Gamma_\ell$
such that $g_{\eta}^{2n_2\ell}(y)\in f^{2n_2\ell}(Q)$.
\end{affi}
\begin{demo}
By definition of $\lambda$, the ball $B(E_\mu^{-1}(y),\lambda)$ is contained in some cube
$Q_{(i_1,\dots,i_d)}$.
The cube $Q=Q_{\mu,(i_1,\dots,i_d)}=E_\mu(Q_{(i_1,\dots,i_d)})$ contains $y\in \Delta$ and hence belongs
to $\Gamma_{\mu,\ell}$.
Let $\ell\in \{0,\dots, 2^d-1\}$ be such that $Q_{\mu,(i_1,\dots,i_d)}\in \Gamma_{\mu,\ell}$.

Since $d_{unif}(\id,E_\mu^{-1}f^{-2n_2(j+1)}g_Q^{2n_2}f^{2n_2(j)}E_\mu)<\eta_0$ for each $j\in \{0,\dots,2^d-1\}$,
we obtain the bound
\begin{equation*}
\begin{split}
d_{unif}(\id, E_{\mu}^{-1}f^{-2n_2\ell}g^{2n_2\ell}E_\mu)&\leq
\sum_{j=0}^{\ell-1}d_{unif}(\id,E_\mu^{-1}f^{-2n_2(j+1)}g^{2n_2}f^{2n_2(j)}E_\mu)\\
&<2^\ell\eta_0\leq 2^d\eta_0<\lambda,
\end{split}
\end{equation*}
by our choice of $\eta_0$.

By our estimate above, the point
$E_\mu^{-1}f^{-2n_2\ell}g^{2n_2\ell}E_\mu(E_\mu^{-1}(y))$ belongs to the ball
$B(E_\mu^{-1}(y),\lambda)$ and hence to $Q_{(i_1,\dots,i_d)}$.
This proves that $f^{2n_2\ell}(Q)$ contains the point $g^{2n_2\ell}(y)$ as required.
\end{demo}

In order to conclude, we fix a point $y\in \{0\}\times\Delta$
and $\ell\in\{0,\dots,2^d-1\}$, $Q\in \Gamma_\ell$
such that $g_{\eta}^{2n_2\ell}(y)\in f^{2n_2\ell}(Q)$.
Note that if
$$|\log\det Df^{2n_2\ell}(y)-\log\det Dg^{2n_2\ell}(y)|>K,$$
then the last conclusion of the proposition already holds for $y$.

Otherwise, since $g^{2n_2\ell}(y)$ belongs to $f^{2n_2\ell}(Q)$, we have
$$Dg^{2n_2}(g^{2n_2\ell}(y))=Dg^{2n_2}_Q(g^{2n_2\ell}(y));$$
since $f$ is a linear cocycle, we obtain $$\log\det Df^{2n_2}(g^{2n_2\ell}(y))=\log\det Df^{2n_2}(f^{2n_2\ell}(y)).$$
Thus the property satisfied by $g_Q$ implies that
$$|\log\det Df^{2n_2}(f^{2n_2\ell}(y))-\log\det Dg^{2n_2}(g^{2n_2\ell}(y))|>2K.$$
Using the fact that
$$|\log\det Df^{2n_2\ell}(y)-\log\det Dg^{2n_2\ell}(y)|\leq K,$$
we then obtain that
$$|\log\det Df^{2n_2(\ell+1)}(y)-\log\det Dg^{2n_2(\ell+1)}(y)|>K.$$
This gives again the last conclusion of the proposition and concludes the proof.
\end{demo}

%% file: UD-perturbation0402.tex
%%%%%%%%%%%%%%%%%%%%%%%%%%%%%%%%%%%%%%%%%%%%%%%%%%%%%%%%%%%%%%%%%%%%%%%%%%
\section{Almost tidy perturbation in a cube}
%%%%%%%%%%%%%%%%%%%%%%%%%%%%%%%%%%%%%%%%%%%%%%%%%%%%%%%%%%%%%%%%%%%%%%%%%%

The aim of this section is to give the proof of Proposition~\ref{p.pertAreduced}, which finishes the proof of Theorem A and hence of the Main Theorem. 

This proposition is very close in spirit to the original idea of C. Pugh for the famous $C^1$-closing lemma. Pugh wanted  to perturb the orbit of a point $x$ in a cube in such a way that it exits the support of the perturbation through the orbit of another given point $y$. Pugh noticed that, at each time, one has much more freedom to perturb the orbits ``in the direction of the smallest dimension of the image of the cube". As he needed to perform a perturbation in arbitrary directions, a linear algebra lemma allowed him to choose the pattern of the cube such that each direction would be at some time the smallest one.

Here we just want to perform a perturbation that modifies the jacobian: hence we can do it by a perturbation ``in an arbitrary direction". However, in principle our perturbations have to be tidy, or very close to tidy.
For this reason the support of the perturbation will have length $2n_2$: we obtain the perturbation of the jacobian during the first $n_2$ iterates, and then try to remove the perturbation during the last $n_2$ iterates.
For this reason, we need that the smallest dimension of the image of our cube corresponds to the same direction, all along the time support of the perturbation. This is obtained by choosing a pattern of our cubes having one direction much smaller than the others.

An additional difficulty comes from the fact that we were not able
to obtain a tidy perturbation. After several iterations, the thin edges of the cube could become
strongly sheared above its base; for this reason, it was not possible to remove the perturbation
in a neighborhood of these edges. By choosing the cubes's height small enough, we are able to remove the perturbation
in a rectified cube which differs from the sheared one only in a small region (see Figure 2).
This ``almost tidy perturbation'' turns out to be sufficient for our purposes.

\begin{figure}[ht]\label{f.pancake}
\begin{center}
\psfrag{f}{$f^n$}
\includegraphics[scale=.5]{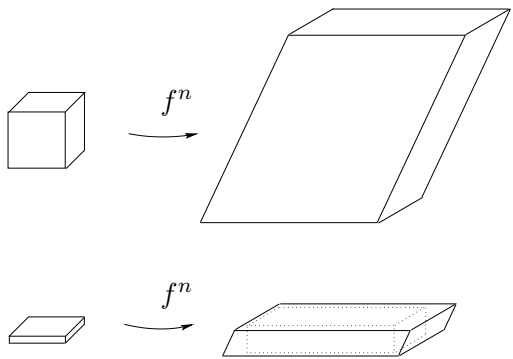}
\end{center}
\caption{The image of a cube under many iterates of a linear cocycle can be quite distorted (top). The image of a wafer, on the other hand, stays flat and wafer-like (bottom). }
\end{figure}

\subsection{Choice of $n_2$}

Given $d\geq 1$ and $\theta\in (0,1)$, we fix $\theta_0=\theta^{\frac 13}$ and choose:
\begin{itemize}
\item a smooth function $\zeta\colon\RR^{d-1}\to [0,1]$ such that
\begin{itemize}
\item $\zeta(v)=1$ if $v\in [-\theta_0^2,\theta_0^2]^{d-1}$ and
\item $\zeta(v)=0$ if $v\in \RR^{d-1}\setminus[-\theta_0,\theta_0]^{d-1}$;
\end{itemize}
\item  a smooth function $\xi\colon[-1,1]\to\RR $ such that:
\begin{itemize}
\item $\xi(z)=z$ for $z\in[-\theta_0^2,\theta_0^2]$,
\item $\xi(z)=0$ for $z$ in a neighborhood of $[-1,1]\setminus[-\theta_0,\theta_0]$, and
\item $\xi(-z)=-\xi(z)$.
\end{itemize}
\end{itemize}
We denote by $X$ the vector field on $[-1,1]$ defined by $X(z)=\xi(z)\frac\partial{\partial z}$ and by $(X_s)_{s\in \RR}$ the induced flow on $[-1,1]$. 
For every non-trivial closed interval $I \subset \RR$ and any $s$, we denote by $X^{I}_s$
the diffeomorphism of $I$ obtained from $X_s$ by considering a parametrization of $I$ by $[-1,1]$ with constant derivative. 

\begin{rema}
The distance $d_{C^1}(X_s^{I},\id_I)$ has two parts.
The $C^0$ part is proportional to the length of $I$.
The part that measures the distance between the derivatives does not depend on $I$: $$\max{\{|D(X^{I}_s)-1|,|D(X^{I}_{-s})-1|\}}=\max{\{|D(X_s)-1|,|D(X_{-s})-1|\}}.$$
\end{rema}

We fix constants $C,\ve>0$ and a real number $s_0>0$ satisfying  $\max{\{|D(X_{s_0})-1)|,|D(X_{-s_0})-1|\}}<\frac\varepsilon{C}$.
We then define the diffeomorphisms $h_s=X_{s.s_0}$ and
$h^I_s=X^{I}_{s.s_0}$.

Fix $K>0$. Note that the derivative of $h_{-1}$ on $[-\theta_0^2,\theta_0^2]$ is a constant:
it is equal to $e^{-s_\ve}<1$ for some constant $s_\ve$. We fix $n_2\geq 1$ such that
$$n_2 s_\ve>K.$$ 

\subsection{Construction of a perturbation}\label{ss.perturbation}
Consider now a sequence of matrices $(A_i)_{i\in \ZZ}$
in $GL(d,\RR)$ such that $\|A \|,\|A_i^{-1}\|<C$,
and denote by $f$ the associated linear cocycle.
Up to a change of coordinates by isometries on each $\{i\}\times \RR$,
we may assume that the $A_i$ leave invariant the hyperplane $\RR^{d-1}\times \{0\}$
and take the form 
$$
A_i=\left(\begin{array}{cccl}&& &\alpha_1^i\\
                           &B_i& &\vdots\\
                           &&    &\alpha_{d-1}^i\\
                           0&\dots&0&b_i
\end{array}\right).
$$
For $i\in\{0,\dots,2^{d+1}n_2\}$ and $\mu>0$ we set
$$I_{i}=\left[-\mu\prod_{k=0}^{i-1}b_k,\mu\prod_{k=0}^{i-1}b_k\right].$$ 

We next fix $j\in \{0,\dots, 2^d-1\}$. For each $i\in \ZZ$,
we denote by $H_{i}$ the diffeomorphism of $\{i\}\times\RR^d$ defined as follows:
\begin{itemize}
\item $H_{i}=\id$ if $i\not \in \{2jn_2,\dots,2(j+1)n_2-1\}$.
\item For $i\in\{2jn_2,\dots,2(j+1)n_2-1\}$, the map $H_i$ coincides with $\id$ outside
$\{i\}\times\RR^{d-1}\times I_i$.
\item For $i\in\{2jn_2,\dots,2(j+1)n_2-1\}$ and $(v,z)\in \RR^{d-1}\times I_{i}$, we set
$$s(v)=\zeta\left((B_{i-1}\dots B_0)^{-1}(v)\right),$$
and define
$$H_{i}(v,z)=(v,h^{I_i}_{-s(v)}(z)) \text{ if } i<(2j+1)n_2,$$
$$H_{i}(v,z)=(v,h^{I_i}_{+s(v)}(z) \text{ if } i\geq (2j+1)n_2.$$
\end{itemize}
Note that the diffeomorphisms $H_i$ depend on the choice of $\mu, (A_i)$ and $j$.

We define by $g$ the cocycle that coincides with
$f\circ H_{i}$ on $\{i\}\times \RR^d$. It will satisfy the conclusions of
Proposition~\ref{p.pertAreduced}, provided $\mu$ has been chosen smaller than
some constant $\mu_0=\min\{\mu_1,\mu_2,\mu_3,\mu_4\}$ which will be defined in the 
following subsections.

\subsection{Support of the perturbation: first choice of $\mu$}

Let $\cC_{i}=P_{i}\times I_{i}$ where $P_i= B_{i-1}\dots B_0([-1,1]^{d-1})$.
Let $Q_i=A_{i-1}\dots A_0(Q_\mu)$.
The two parallelepipeds $\cC_i$ and $Q_i$ have the same base $P_i\times \{0\}$
in $\RR^{d-1}\times \{0\}$
and the same height along the coordinate $z$, but $\cC_i$ has been rectified.

In these notations $g$ is a perturbation of $f$ with support contained in $ \bigcup_{i=0}^{2n_2-1} \{i+2jn_2\}\times \theta_0 \cC_{i+2jn_2}$.

\begin{lemm}\label{l.cuttingcube2} Given $d,C,K,\ve>0$ and $\theta\in (0,1)$,
there exists $\mu_1=\mu_1(d,C,K, \ve,\theta)$
such that, for every $\mu\in(0,\mu_1)$,
for every sequence $(A_i)$ in $GL(d,\RR)$ with $\|A_i\|,\|A_i^{-1}\|<C$ and for every
$i\in\{0,\dots,2^{d+1}n_2\}$, we have:
$$
\theta_0 \cC_{i} \subset Q_{i}.
$$
In particular the perturbation $g$ defined in Section~\ref{ss.perturbation} is supported in 
$$\bigcup_{i=2jn_2}^{2(j+1)n_2-1} f^i(\{0\}\times Q_\mu)=\bigcup_{i=2jn_2}^{2(j+1)n_2-1} Q_i.$$
\end{lemm}

\begin{demo} Choose
$$\mu_1 < \left(2^{d+1}n_2C^{2^{d+2} n_2}\right)^{-1} (\theta_0^{-1}-1).$$
Consider a point $(v_i,z_i)$ in $\cC_i$, and for $0\leq k<i$,
its image $(v_k,z_k)$ under $(A_{i-1}\dots A_k)^{-1}$.
We need to show that $(v_0,z_0)$ belongs to $\theta_0^{-1}Q_\mu$.

For the second coordinate, we have $z_0=z_i\prod_{k=0}^{i-1}b_k^{-1}$, and hence $z_0$ belongs to
$I_0\subset \theta_0^{-1} I_0$. It remains to control the first coordinate:
for each $k$, it decomposes as $v_k=B_{k}^{-1} v_{k+1}+w_k$ where
$w_k$ is the projection on the first coordinate of $A_{k}^{-1}(0,z_{k+1})$.
Note that
$$\|w_k\|\leq C |I_{k+1}|\leq \mu C^{k+2}\leq \mu C^{i+1}.$$
Decomposes $v_k$ as a sum $\tilde v_k+r_k$ where
$\tilde v_k=(A_{i-1}\dots A_k)^{-1}(v_i)$ belongs to $P_k$ and
$$r_k=\sum_{j=k}^{i-1} (A_{j-1}\dots A_k)^{-1} w_{j}.$$
We thus have
$$\|r_k\|\leq (i-k)C^{i-k-1}\max\{w_j\}\leq iC^{2i}\mu_1.$$

Since $v_0$ belongs to $P_0=[-1,1]^{d-1}\times \{0\}$,
it is now enough to show that
$\|r_0\|$ is smaller than the distance between the complement of $\theta_0^{-1} P_0$ and $P_0$.
This distance is bounded from below by $\theta_0^{-1}-1$.
Our choice of $\mu_1$ now implies that $v_0$ belongs to $\theta_0^{-1}P_0$, as required.
\end{demo}

\subsection{Size of the perturbation: second choice of $\mu$}
We check here that the perturbation $g$ is $C^1$-close to $f$.
\begin{lemm} Given $d,C,K, \ve,\theta>0$
there exists $\mu_2=\mu_2(d,C,K,\varepsilon,\theta)$
such that for any $\mu\in (0,\mu_2)$,
for any sequence $(A_i)$ in $GL(d,\RR)$ with
$\|A_i\|,\|A_i^ {-1}\|<C$ and for any $j\in \{0,\dots,2^d-1\}$,
the diffeomorphisms $f$ and $g$ defined at Section~\ref{ss.perturbation} satisfy
$$d_{C^1}(g, f)<\ve.$$
\end{lemm}
\begin{demo}
We choose
$$\mu_2<\frac \ve{C^{2^{d+2}n_2} s_0 \|\xi\|\|D\zeta\|}.$$
For all $i\in\{0,\dots,2^ {d+1}n_2-1\}$, the quantity
$\|DH_i(\frac\partial{\partial z})-\frac\partial{\partial z}\|$ is bounded by the maximum of $|Dh_1-1|$ and $|Dh_{-1}-1|$ and thus is bounded by $\ve/C$, by our choice of $s_0$. For any unit vector $u$ of $\RR^ {d-1}\times \{0\}$
we have 
\begin{eqnarray*}
\left\|DH_i\left(\frac\partial{\partial u}\right)-\frac\partial{\partial u}\right\|&\leq&
\|\xi\|\|D\zeta\|\|(B_{i-1}\dots,B_0)^{-1}\| s_0 \mu\prod_{j=0}^{i-1}b_i \\
 &\leq& C^{2(2^{d+1}n_2-1)} s_0 \mu\|\xi\|\|D\zeta\|.
\end{eqnarray*}
It follows that for $\mu<\mu_2$, we have
$\|DH_{i}-\id\|<\frac\ve{C},$ and
consequently we obtain $\|Dg-Df\|<\ve$.

\end{demo}

\subsection{Perturbation of the jacobian: third choice of $\mu$\label{sss.jac}}
We define the \emph{effective support} of the perturbation as the set where
$$|\log \det Df- \log \det Dg|=|\log \det DH|=s_\ve.$$
By construction, this effective support contains $ \bigcup_{i=0}^{n_2-1} \{i+2jn_2\}\times \theta_0^2 \cC_{i+2jn_2}$. 
We now prove that for small $\mu$ the orbits of the points in $\{0\}\times \theta Q_\mu$
under $g$ meet the effective support; this implies that the perturbation has the expected effect
on the jacobian along these orbits.

\begin{lemm}\label{l.Jaccube}
Given $d,C,K, \ve,\theta>0$  there exists $\mu_3=\mu_3(d,C,K,\ve,\theta)$ such that
for any $\mu\in (0,\mu_2)$,
for any sequence $(A_i)$ in $GL(d,\RR)$ with
$\|A_i\|,\|A_i^ {-1}\|<C$ and for any $j\in \{0,\dots,2^d-1\}$,
the perturbation $g$ defined at Section~\ref{ss.perturbation} satisfies
for every $i\in\{0,\dots n_2-1\}$:
$$g^{i}\left(\{2jn_2\}\times \theta Q_{2jn_2}\right)\subset \{i+2jn_2\}\times \theta_0^2 \cC_{i+2jn_2}.$$ 
In particular, for every $x\in \{0\}\times \theta Q_\mu$ one has
$$|\log \det Df^{n_2}(f^{2jn_2}(x))-\log \det Dg^{n_2}(f^{2jn_2}(x))|
=n_2s_\varepsilon>K.$$ 
\end{lemm}
\begin{demo}
The proof is quite similar to the proof of Lemma~\ref{l.cuttingcube2}.
We choose
$$\mu_3 < \left(2^{d+1}n_2C^{2^{d+2} n_2}\right)^{-1} (\theta_0^{-1}-1).$$
Let $(0,v_0,z_0)$ be a point in $\{0\}\times \theta Q_\mu$, and let
$(i,v_i,z_i)$ be its image under $g^i$.
We will show that $(v_i,z_i)$ belongs to $\theta_0^2\cC_i$,
for every $i\in\{0,\dots,(2j+1)n_2-1\}$.

For the second coordinate, we have $z_i\leq z_0\prod_{k=0}^{i-1}b_k$,
since $H_i$ is the identity, for $i<2jn_2$, and
shrinks the second coordinate, for $i\in\{2jn_2,\dots,(2j+1)n_2-1\}$.
Hence $z_i$ belongs to $\theta I_i\subset \theta_0^2 I_i$.
It remains to control the first coordinate:
$H_i$ has the form $(v,z)\mapsto (v,h_i(v,z))$
and $v_i$ decomposes as $B_{i-1} v_{i-1}+w_i$ where
$w_i$ is the projection on the first coordinate of $A_{i-1}(0,h_{i-1}(v_{i-1},z_{i-1}))$.
Note that
$$\|w_i\|\leq C\|h_{i-1}(v_{i-1},z_{i-1})\|\leq C|z_{i-1}|\leq C^i\mu_3\theta.$$

In particular, $v_i$ decomposes as a sum $\tilde v_i+r_i$, where
$\tilde v_i=B_{i-1}\dots B_0(v_0)$ belongs to $\theta P_i$ and
$$r_i=\sum_{k=1}^i A_{i-1}\dots A_k w_{k}.$$
It follows that $\|r_i\|\leq iC^i\mu_3\theta$.

Since $\bar v_i$ belongs to $\theta P_i$, it is now enough to show that
$\|r_i\|$ is smaller than the distance between the complement of $\theta^2_0 P_i$ and $\theta P_i=\theta_0^3 P_i$.
This distance is bounded from below by $C^{-i}\theta(\theta_0^{-1}-1)$.
Our choice of $\mu_1$ now implies that $v_i$ belongs to $\theta_0^2P_i$, as required.
\end{demo}

\subsection{Almost tidy perturbation: fourth choice of $\mu$}
To finish the proof of Proposition~\ref{p.pertAreduced}, we are left to show that, after a rescaling by the linear map $E_\mu$, the action of the perturbation $g$
on the orbits of $f$ tends uniformly to the identity as $\mu\to 0$.

\begin{lemm}\label{l.atidycube}
Given $d,C,K, \ve,\theta$ and $\eta>0$ there exists $\mu_4=\mu_4(d,C,K,\ve,\theta,\eta)$ such that
for any $\mu\in (0,\mu_4)$,
for any sequence $(A_i)$ in $GL(d,\RR)$ with
$\|A_i\|,\|A_i^ {-1}\|<C$ and for any $j\in \{0,\dots,2^d-1\}$,
the diffeomorphisms $f$ and $g$ defined in Section~\ref{ss.perturbation} satisfy
on $\{0\}\times \RR^d$,
$$d_{unif}\left(\id, {E_\mu^{-1}f^{-2(j+1)n_2}g^{2(j+1)n_2}E_\mu}_{|_{\{0\}\times \RR^d}}\right)<\eta.$$
\end{lemm}
\begin{demo}
First of all, notice that: 
\begin{itemize}
\item $H_{i}=\id$ if $i<2jn_2$ or $i\geq 2(j+1)n_2-1$,
\item $H_i=(E_\mu \bar A_{i-1}\dots \bar A_0)H^{-1}(E_\mu \bar A_{i-1}\dots \bar A_0)^{-1}$
for $2jn_2\leq i <(2j+1)n_2$,
\item $H_i=(E_\mu \bar A_{i-1}\dots \bar A_0)H(E_\mu \bar A_{i-1}\dots \bar A_0)^{-1}$
for $(2j+1)n_2\leq i <2(j+1)n_2$,
\end{itemize}
where $H$ is a diffeomorphism of $\{0\}\times \RR^d$, and the $\bar A_k$ are matrices, defined by
$$H(v,z)=(v, h_{\zeta(v)}(z)),$$
$$
\bar A_k=\left(\begin{array}{cccl}&& &0\\
                           &B_k& &\vdots\\
                           &&    &0\\
                           0&\dots&0&b_k
\end{array}\right).
$$

In this notation, we can write
\begin{eqnarray*}
{E_\mu^{-1}f^{-2(j+1)n_2}g^{2(j+1)n_2}E_\mu}_{|_{\{0\}\times \RR^d}}=
E_\mu^{-1}\left(\prod_{i=0}^{2(j+1)n_2-1}f^{-i}H_i f^i \right)E_\mu\\
=Q_l(P_{2n_2-1}H\dots P_{n_2}H)(P_{n_2-1}H^{-1}\dots P_0H^{-1}) Q_r,
\end{eqnarray*}
where
$$P_i=(\bar A_{2jn_2+i-1}\dots\bar A_0)^{-1}(\bar A_{2jn_2+i}^{-1}E_\mu^{-1}A_{2jn_2+i}
E_\mu)(\bar A_{2jn_2+i-1}\dots\bar A_0),$$
$$Q_r=(\bar A_{2jn_2-1}\dots\bar A_0)^{-1}E_\mu^{-1}(A_{2jn_2-1}\dots A_0)E_\mu,\quad\hbox{ and}$$
$$Q_l=E_\mu^{-1}(A_{2(j+1)n_2-1}\dots A_0)^{-1}E_\mu(\bar A_{2(j+1)n_2-1}\dots \bar A_0).$$
Note that
$$E_\mu^{-1}A_k E_\mu\underset{\mu\to 0}\longrightarrow \bar A_k.$$
This implies that $P_i$, $Q_r$ and $Q_l$ tend to $I_d$ when $\mu$ goes to $0$.
As a consequence $E_\mu^{-1}f^{-2(j+1)n_2}g^{2(j+1)n_2}E_\mu$
tends uniformly to the identity on $\{0\}\times \RR^d$ as $\mu\to 0$.
This implies the conclusion of the lemma.
\end{demo}

%% file: appendix0402.tex
\section*{Appendix: The (LD) property is not generic}
\addcontentsline{toc}{section}{Appendix: The (LD) property is not generic}

The proof of our main theorem would have been much easier if the large derivative property were a generic property; 
the aim of this section is to show that, indeed, it is not a generic property. 

Consider the set $LD\subset \diff^1(M)$ of diffeomorphisms having the (LD)-property.  

\begin{rema} If $f$ is Axiom A, then $f\in LD$. More generally, if the chain recurrent set $CR(f)$ is a finite union of invariant compact sets, each of them admiting a dominated splitting, then $f\in LD$: the dominated splitting implies that the vectors in one bundle are exponentially more expanded than the vectors in the other bundle, implying that $\sup\{\|Df^n(x)\|,\|Df^{-n}(x)\|\}$ increases exponentially with $n$ for $x\in CR(f)$. 

Consider the open set $\cT$ of \emph{tame diffeomorphisms}: these are the diffeomorphisms such that all the diffeomorphisms in a $C^1$-neighborhood have the same finite number of chain recurrence classes.  A consequence of \cite{BDP} and \cite{BC} is that any diffeomorphism in an open and dense subset  $\cO\subset \cT$ admits a dominated splitting on each of its chain recurrence classes.  This implies that  $\cO$ is contained in $LD$. 
\end{rema}

To find an open set in which $LD$ is not residual, we therefore must look
among the the so-called {\em wild} diffeomorphisms, whose chain recurrent
set has no dominated splitting.

\begin{prop} For any compact manifold $M$ with $dim(M)>2$, there exists a non-empty open subset $U_M\subset \diff^1(M)$ such that $LD\cap U_M$ is meager.
\end{prop}
 
Let $V_M$ be the set of diffeomorphisms $f$ possessing a periodic point $x=x_f$ such that $Df^{\pi(x)}(x)=id$, where $\pi(x)$ is the period of $x$. 
We denote by $U_M$ the interior of the closure of $V_M$. 
In \cite{BD}, it is shown that  $U_M$ is nonempty, for every compact manifold $M$ of 
dimension at least $3$.  
\begin{rema}
The least period of $x_f$ is not locally bounded in $U_M$.  This is because, for every $n>0$,  the set of diffeomorphisms whose periodic orbits of period less than $n$ are hyperbolic is open and dense. 
\end{rema}

To prove the proposition, it suffices to show that $LD\cap U_M$ is meager.

For every $K>1$ and $n\in\NN$, we consider the set $W(K,n)\subset \diff^1(M)$ of diffeomorphisms $f$ such that there exist $m>n$ and  two compact balls $B_0$, $B_1$ with the following properties.

\begin{itemize}
\item $B_1$ is contained in the interior of $B_0$;
\item  $f^m(B_0)$ is contained in the interior of $B_0$, and $B_1$ is contained in the interior of  $f^{m}(B_1)$;
\item for every $x\in B_0$:  
$$\sup\{\|Df^m(f^j(x))\|, \|Df^{-m}(f^{j}x)\|, x\in \La_f \mbox{ and } j\geq 0\}<K.$$
\end{itemize}

Note that $W(K,n)$ is an open set, for every $K$ and $n$. Furthermore, for a given $K>1$, the diffeomorphisms in $\cR=\cap_{n\in\NN} W(K,n)$ do not satisfy the (LD) property.
To see this,  let $f\in \cR$, and fix an arbitrary integer $n>0$.  Consider 
an integer $m>n$ and  two balls $B_0,B_1$ given  by the definition of $W(K,n)$.
Then the  points in the open set $\interior(B_1)\setminus f^{-m}(B_1)$ are not periodic, since
their entire orbits lie in $B_0$, and along these orbits, the quantities $\|Df^m\|$ and $\|Df^{-m}\|$ are 
bounded by $K$.  Hence $f$ does not satisfy the (LD) property.

\begin{lemm} For every $K>1$ and every $n\in \NN$, $W(K,n)\cap U_M$ is (open and) dense in $U_M$. 
\end{lemm}
\begin{demo}Consider $f_0\in U_M$, and let $\cU\subset U_M$ be a neighborhood of $f_0$. There exists $f_1\in \cU$ and $m>n$ such that $f_1$ has a periodic point $x$ of period $m$ with $Df^m(x)=Id$. Then there exist $f_2\in \cU$, arbitrarily close to $f_1$, and a small ball $D$ contained in  an arbitrarily small neighborhood of $x$ such that $f_2^m(D)=D$, the restriction of $f_2^m$ to $D$ is the identity map, and $f_2^i(D)\cap D=\emptyset$ for $i\in\{1,\dots,m-1\}$. Let  $\De=\bigcup_0^m f_2^i(D)$.

Observe that there is a neighborhood $\cU_1\subset \cU$ of $f_2$ such that $\|Df^m\|$ and $\|Df^{-m}\|$ are  bounded by $\frac 12 K$ on $\De$.
We conclude the proof by noting that there exist $f_3\in \cU_1$ and two compact balls $B_1\subset \interior (B_0)\subset B_0 \subset \interior D$
with the desired properties.
\end{demo}

B. Fayad has observed that, for the proof of the Main Theorem, it suffices to obtain the (LD)-property
merely on a dense subset of $M$.  This weaker (LD) condition is not residual either: 
the proof of the proposition above shows that for the diffeomorphisms 
in the residual subset $\cR$, the (LD)-property is not satisfied on any dense subset of $M$.

We conclude this section by discussing another strange feature of the (LD) property: we have proved the density of the set of diffeomorphisms satisfying the large derivative property and whose periodic orbits are all hyperbolic. On the one hand, the large derivative property is a uniform property on the non-periodic orbits (in the definition, the integer $n(K)$ does not depend on the point $x$). On the other hand, for the hyperbolic periodic orbits, the norm of the derivative tends exponentially to infinity. So it is natural to ask if we could also include the periodic orbits in the definition of (LD)-property. Let us say that $f$ satisfies the \emph{strong (LD) property} if, for every $K>1$ there exists $n_K$ such that for every $n\geq n_K$ and every $x\in M$ one has $\sup\{\|Df^n(f^i(x))\|,\|Df^{-n}(f^i(x)\|, i\in\ZZ\}>K$. 
\begin{ques} Does the strong (LD) property hold on a dense subset of $\diff^1(M)$?
\end{ques}

%% file: conjecture0402.bbl
\begin{thebibliography}{BCVW}

\bibitem[BC]{BC} Bonatti, Ch.; Crovisier, S.,
R\'ecurrence et g\'en\'ericit\'e.
{\em Invent. Math.} {\bf 158} (2004), 33--104.

\bibitem[BCW1]{BCW} Bonatti, Ch.; Crovisier, S.; Wilkinson, A.,
$C^1$-generic conservative diffeomorphisms have trivial centralizer.
{\em J. Mod. Dyn. } {\bf 2} (2008), 359--373.
A previous version was: {\em Centralizers of $C^1$-generic diffeomorphisms.}
Preprint (2006) arXiv:math/0610064.

\bibitem[BCW2]{BCW2} Bonatti, Ch.; Crovisier, S.; Wilkinson, A.,
{\em  The centralizer of a $C^1$ generic diffeomorphism is trivial.} Preprint arXiv:0705.0225.

\bibitem[BCVW]{BCVW} Bonatti, Ch.; Crovisier, S.; Vago, G; Wilkinson, A.,
{\em Local density of diffeomorphisms with large centralizers.} Preprint (2007) arXiv:0709.4319.

\bibitem[BD]{BD} Bonatti, Ch.; D\'{\i}az, L.,
On maximal transitive sets of generic diffeomorphisms.
{\em Publ. Math. Inst. Hautes \'Etudes Sci.}  \textbf{96} (2002), 171--197. 

\bibitem[BDP]{BDP} Bonatti, Ch.; D\'{\i}az, L.; Pujals, E.,
A $C\sp 1$-generic dichotomy for diffeomorphisms: weak forms of hyperbolicity or infinitely many sinks or sources.  {\em Ann. of Math.} \textbf{158} (2003), 355--418.

\bibitem[Bu1]{Bu} Burslem, L.,
Centralizers of partially hyperbolic diffeomorphisms.
{\em Ergod. Th. \& Dynam. Sys.} {\bf 24}  (2004), 55--87.

\bibitem[Bu2]{Bu2} Burslem, L., 
Centralizers of area preserving diffeomorphisms on $S\sp 2$.
{\em Proc. Amer. Math. Soc.}  {\bf 133 } (2005), 1101--1108.

\bibitem[Fi]{Fi} Fisher, T.,
Trivial centralizers for Axiom A dif\-fe\-o\-mor\-phisms. {\em Pre\-print}.

\bibitem[FRW]{FRW} Foreman, M.; Rudolph, D.; Weiss, L.,
On the conjugacy relation in ergodic theory.
{\em C. R. Math. Acad. Sci. Paris}  {\bf 343} (2006), 653--656.

\bibitem[G]{ghys} Ghys, \'E.,
Groups acting on the circle.
{\em L'Enseign. Math.} \textbf{47} (2001), 329--407.

\bibitem[Ko]{Ko} Kopell, N.,
Commuting diffeomorphisms. In {\em Global Analysis},
Proc. Sympos. Pure Math., Vol. XIV, AMS (1970), 165--184.

\bibitem[N]{navas} Navas, A.,
{\em Three remarks on one dimensional bi-Lipschitz conjugacies.}
Preprint 2007, arXiv:0705.0034.

\bibitem[P]{P} Palis, J.,
Vector fields generate few diffeomorphisms.
{\em Bull. Amer. Math. Soc.} {\bf 80} (1974), 503--505. 

\bibitem[PY1]{PY1} Palis, J.; Yoccoz, J.-C.,
Rigidity of centralizers of diffeomorphisms.
{\em Ann. Sci. \'Ecole Norm. Sup.} {\bf 22} (1989), 81--98.

\bibitem[PY2]{PY2} Palis, J.; Yoccoz, J.-C.,
Centralizers of Anosov diffeomorphisms on tori.
{\em Ann. Sci. \'Ecole Norm. Sup.} {\bf 22} (1989), 99--108.

\bibitem[Pu]{Pu} Pugh, C.,
The closing lemma. {\em Amer. J. Math.} {\bf 89} (1967), 956--1009.

\bibitem[R]{R} Rees, M., A minimal positive entropy homeomorphism of the $2$-torus.
{\em J. London Math. Soc.} \textbf{23} (1981), 537--550.

\bibitem[Sm1]{Sm1} Smale, S.,
Dynamics retrospective: great problems, attempts that failed.
{\em Nonlinear science: the next decade} (Los Alamos, NM, 1990).
Phys. D  {\bf 51}  (1991), 267--273.

\bibitem[Sm2]{Sm2} Smale, S.,
Mathematical problems for the next century.
{\em Math. Intelligencer} {\bf 20} (1998), 7--15.

\bibitem[To1]{To1} Togawa, Y.,
Generic Morse-Smale diffeomorphisms have only trivial symmetries.
{\em Proc. Amer. Math. Soc.} {\bf 65} (1977), 145--149.

\bibitem[To2]{To2} Togawa, Y.,
Centralizers of $C\sp{1}$-diffeomorphisms.
{\em Proc. Amer. Math. Soc.} {\bf 71} (1978), 289--293.

\end{thebibliography}
